\documentclass[11pt]{amsart}

\usepackage[utf8]{inputenc}
\usepackage{lmodern}
\usepackage[T1]{fontenc}
\usepackage[english]{babel}

\usepackage{geometry}
\usepackage{amsmath, amssymb, amsfonts, amsthm}
\usepackage{mathrsfs}
\usepackage{pgf, tikz, tikz-cd}
	\usetikzlibrary{calc, shapes, arrows, trees, matrix, patterns, decorations}
\usepackage{dynkin-diagrams}

\usepackage{enumitem}
\usepackage{array}

\usepackage{hyperref}
	\hypersetup{pdftex, pdfauthor={Geoffrey Janssens, Doryan Temmerman, and Fran\c{c}ois Thilmany}, bookmarksopen, bookmarksnumbered, colorlinks={true}, linkcolor={gray}, citecolor={gray}, urlcolor={gray}}
\usepackage[hyperpageref]{backref} 
\usepackage[english]{cleveref}

\newcommand{\rA}{\mathrm{A}}

\newcommand{\rC}{\mathrm{C}}
\newcommand{\rD}{\mathrm{D}}
\newcommand{\rE}{\mathrm{E}}

\newcommand{\rM}{\mathrm{M}}
\newcommand{\rN}{\mathrm{N}}

\newcommand{\rP}{\mathrm{P}}

\newcommand{\rR}{\mathrm{R}}

\newcommand{\rr}{\mathrm{r}}

\newcommand{\bB}{\mathbf{B}}

\newcommand{\bG}{\mathbf{G}}
\newcommand{\bH}{\mathbf{H}}

\newcommand{\bP}{\mathbf{P}}
\newcommand{\bQ}{\mathbf{Q}}

\newcommand{\bS}{\mathbf{S}}

\newcommand{\bX}{\mathbf{X}}

\newcommand{\bZ}{\mathbf{Z}}
\newcommand{\cD}{\mathcal{D}}

\newcommand{\cO}{\mathcal{O}}

\newcommand{\cS}{\mathcal{S}}

\let\tempepsilon\epsilon
\let\epsilon\varepsilon
\let\varepsilon\tempepsilon
\let\tempphi\phi
\let\phi\varphi
\let\varphi\tempphi

\newenvironment{smallbmatrix}{\left[\begin{smallmatrix}}{\end{smallmatrix}\right]}

\newcommand{\C}{\mathbb{C}}
\newcommand{\F}{\mathbb{F}}
\renewcommand{\H}{\mathbb{H}} % Conflict with existing command

\newcommand{\N}{\mathbb{N}}
\newcommand{\Q}{\mathbb{Q}}
\newcommand{\R}{\mathbb{R}}
\newcommand{\Z}{\mathbb{Z}}

\DeclareMathOperator{\im}{im}
\DeclareMathOperator{\End}{End}
\DeclareMathOperator{\Mat}{M}
\DeclareMathOperator{\Span}{Span}

\DeclareMathOperator{\proj}{proj}

\DeclareMathOperator{\Tr}{Tr} % Exists in package

\DeclareMathOperator{\Nrd}{Nrd}

\DeclareMathOperator{\diag}{diag}

\newcommand{\restr}[2]{ % Restricion of a map
	\sbox0{\ensuremath{#1}}
	\ensuremath{
	{#1}\raisebox{-1ex}{%
		\(\mkern2mu {\vrule height 2ex} \mkern2mu {\scriptstyle{#2}}\)}
	}}
\newcommand{\set}[2]{ % Set-builder from predicate
	\ensuremath{\left\{#1\, \middle| \, #2\right\}}
	}

\newcommand{\indx}[2]{\ensuremath \left|#1 : #2 \right|} % Index of a subgroup

\DeclareMathOperator{\Aut}{Aut}

\DeclareMathOperator{\Stab}{Stab}

\DeclareMathOperator{\rad}{rad}

\DeclareMathOperator{\Char}{char}

\DeclareMathOperator{\GL}{GL}
\DeclareMathOperator{\SL}{SL}

\DeclareMathOperator{\U}{U}

\DeclareMathOperator{\PGL}{PGL}
\DeclareMathOperator{\PSL}{PSL}

\DeclareMathOperator{\Res}{Res}
\DeclareMathOperator{\Ind}{Ind}

\DeclareMathOperator{\Irr}{Irr}

\DeclareMathOperator{\Ad}{Ad}

\newcommand{\wt}[1]{\widetilde{#1}}

\newcommand{\ot}{\otimes}

\newcommand{\join}{\vee}

\renewcommand{\P}{\mathbf{P}}	%Projective space
\newcommand{\Aa}{A}			%Subgroup 1 without the trivial elements
\newcommand{\Bb}{B}			%Subgroup 2 without the trivial elements
\newcommand{\Ant}{\Aa^*}		%Subgroup 1 without the trivial elements
\newcommand{\Bnt}{\Bb^*}		%Subgroup 2 without the trivial elements
\newcommand{\Hnt}{H^*}			%H without the trivial elements
\newcommand{\Fi}{F}			%Field in the ping-pong theorem
\renewcommand{\O}{\mathcal{O}}	%Order
\newcommand{\Di}{D}			%Division algebra in the ping-pong theorem
\newcommand{\LFi}{K}			%Local field in the ping-pong theorem
\newcommand{\LDi}{D}			%Local division algebra in the ping-pong theorem
\newcommand{\Alg}{M}			%Semisimple algebra in section 4
\newcommand{\Subalg}{N}		%Semisimple subalgebra in section 4
\newcommand{\fset}{I}			%Finite set of prescribed elements
\newcommand{\rhos}{\rho_{\mathrm{st}}}	%Standard representation
\newcommand{\fod}{\Delta}		%First-order deformation

\newcommand{\Pnai}[1]{\textup{({P}\textsubscript{na\"i})}}		%Pnai
\newcommand{\HNN}[2]{#1 \ast_{#2}}					%HNN extensions
\newcommand{\sbic}[1]{\mathrm{sb}_{#1}}				%Shifted bicyclic unit
\newcommand{\bic}[1]{\mathrm{b}_{#1}}					%Usual bicyclic unit
\newcommand{\qa}[3]{\left(\frac{#1, #2}{#3}\right)}			%Quaternion algebra

\DeclareMathOperator{\ZZ}{\mathrm{Z}}			%Center operator
\renewcommand{\U}{\operatorname{\mathcal{U}}}	%Units
\DeclareMathOperator{\V}{\mathcal{V}}			%Units of augmentation 1
\DeclareMathOperator{\PCI}{PCI}				%Principal central idempotents
\DeclareMathOperator{\Fir}{Fir}					%Set of faithful irreps
\DeclareMathOperator{\Cir}{Cir}				%Set of centrally-faithful irreps
\DeclareMathOperator{\Sir}{Sir}					%Set of not fixed-point-free irreps
\DeclareMathOperator{\Mir}{Mir}				%Set of not division algebra irreps
\DeclareMathOperator{\Bic}{Bic}				%Group generated by bicyclic units
\DeclareMathOperator{\SocA}{SocA}				%Abelian part of the socle
\DeclareMathOperator{\Fit}{Fit}					%Fitting subgroup

\newcommand{\llangle}{\langle \! \langle}		%Left double angles
\newcommand{\rrangle}{\rangle \! \rangle}		%Right double angles

\numberwithin{equation}{section}

\theoremstyle{plain}
	\newtheorem{theorem}[equation]{Theorem}
	\newtheorem{lemma}[equation]{Lemma}
	\newtheorem{proposition}[equation]{Proposition}
	\newtheorem{corollary}[equation]{Corollary}
	\newtheorem{conjecture}[equation]{Conjecture}

\theoremstyle{plain}
	\newcommand{\thistheoremname}{}%
	\newtheorem{generictheorem}[equation]{\thistheoremname}
		
	\newtheorem*{generictheorem*}{\thistheoremname}
	\newenvironment{namedtheorem*}[1]
		{\renewcommand{\thistheoremname}{#1}%
		\begin{generictheorem*}}
		{\end{generictheorem*}}

\theoremstyle{definition}
	\newtheorem{definition}[equation]{Definition}
	\newtheorem{question}[equation]{Question}
	\newtheorem{questions}[equation]{Questions}

\theoremstyle{remark}
	\newtheorem{remark}[equation]{Remark}
	
	\newtheorem{examples}[equation]{Examples}
	\newtheorem*{remark*}{Remark}
	\newtheorem*{problem*}{Problem}

\makeatletter
\newcommand{\customitem}[3][item]{%
	\item[#3] %argument #3 is name of custom item
	\begingroup
	\cref@constructprefix{#1}{\cref@result}%
	\protected@edef\@currentlabel{#3}%
	\protected@edef\@currentlabelname{#3}%
	\protected@edef\cref@currentlabel{%
	[#1][][\cref@result]% %optional argument #1 is (cleveref) label type
	#3}
	\label[#1]{#2}% %argument #2 is label name
	\endgroup  
}
\makeatother

\crefname{condition}{condition}{conditions} \crefformat{condition}{Condition~#2#1#3}

\crefname{condition}{condition}{conditions} \crefformat{condition}{Condition~#2#1#3}

\title[Simultaneous ping-pong for finite subgroups of reductive groups]{Simultaneous ping-pong for\\ finite subgroups of reductive groups}

\author{Geoffrey Janssens}
	\address{Geoffrey Janssens \newline Institut de Recherche en Math\'ematiques et Physique, UCLouvain, 1348 Louvain-la-Neuve, Belgium and \newline Department of Mathematics and Data Science, Vrije Universiteit Brussel, Pleinlaan $2$, 1050 Elsene}
	\email{geoffrey.janssens@uclouvain.be}

\author{Doryan Temmerman}
	\address{Doryan Temmerman \newline The AI lab, Vrije Universiteit Brussel, Pleinlaan 9, 1050 Elsene}
	\email{doryan.temmerman@vub.be}

\author{Fran\c{c}ois Thilmany}
	\address{Fran\c{c}ois Thilmany \newline Departement Wiskunde, KU Leuven: 200B Celestijnenlaan, 3001 Leuven, Belgium}
	\email{francois.thilmany@kuleuven.be}

\thanks{The first and third authors are grateful to the Fonds Wetenschappelijk Onderzoek Vlaanderen --- FWO (grants 2V0722N \& 1221722N) and to the Fonds de la Recherche Scientifique --- FNRS (grants 1.B.239.22 \& FC 4057) for their financial support.}

\dedicatory{Dedicated to the memory of Jacques Tits with great admiration for his legacy.}

\begin{document}

\begin{abstract}
Let $\Gamma$ be a Zariski-dense subgroup of a reductive group $\mathbf{G}$ defined over a field $F$. 
Given a finite collection of finite subgroups $H_i$ ($i \in I$) of $\mathbf{G}(F)$ avoiding the center, we establish a criterion to ensure that the set of elements of $\Gamma$ that form a free product with every $H_i$ (the so-called \emph{simultaneous ping-pong partners for $H_i$}) is both Zariski- and profinitely dense in $\Gamma$. 
This criterion applies namely to direct products $\bG$ of inner $\mathbb{R}$-forms of $\operatorname{(P)GL}_n$, and gives a positive answer to this particular case of a question asked by Bekka, Cowling and de la Harpe. 
For torsion elements, a complication arises due to the fact that a finite cyclic group can split into a direct product. 
When $\mathbf{G}$ is the multiplicative group of a semisimple algebra, we also give a more explicit method to obtain free products between two given finite subgroups, via \emph{first-order deformations}. 

In the second half, we investigate the case where $\mathbf{G}$ is the multiplicative group of the group algebra $FG$ of a finite group $G$, and $\Gamma$ is the group of units of an order in $FG$. 
In this regard, we prove that the set of bicylic units that play ping-pong with a given shifted bicyclic unit, is Zariski- and profinitely dense, addressing a long-standing belief in the field of group rings. 
This result is deduced from the criterion above, combined with sharp existence results for well-behaved irreducible representations of $G$ that are \emph{center-preserving} on a given subgroup. 
\end{abstract}



\maketitle

\setcounter{tocdepth}{2}
{\tableofcontents}

%%%%% Document %%%%%

\section{Introduction} \label{sec:introduction}
\addtocontents{toc}{\protect\setcounter{tocdepth}{1}}

\subsection{Background} 
The construction and study of free products in linear groups is a classical topic going back to the early days of group theory. 
A groundbreaking step was Tits' celebrated alternative \cite{Tits72}, establishing existence of free subgroups in linear groups which are not virtually solvable. 
In fact, he proved the stronger statement that if $\Gamma$ is a finitely generated Zariski-connected linear group over a field $\Fi$, then either $\Gamma$ is solvable, or $\Gamma$ contains a Zariski-dense free subgroup of rank 2. 

Tits' ideas have been used abundantly since, and there are now many generalizations of the alternative. 
For instance, a version for non-Zariski-connected groups was given by Breuillard and Gelander in \cite{BreuillardGelander07}, where they also prove an analogous (in fact, stronger) theorem for the topology induced by a local field. 
Other aspects, such as the speed at which a given finite set produces two elements generating a free subgroup, or the dependency on the field of definition, have been studied. 
We refer the reader to \cite{AbelsMargulisSoifer95,Aoun11,BreuillardGelander08,Breuillard08,Breuillard11,BreuillardGreenGuralnickTao12} for some recent results. 
\medbreak

In the present article, we are interested in constructing free subgroups of linear groups, with a prescribed generator. 
More generally, given a finite subset $\fset$ of a linear group $\Gamma$, the question of interest is: does there exists $\gamma \in \Gamma$ such that for all $h \in \fset$, the subgroup $\langle h, \gamma \rangle$ is \emph{freely} generated by $h$ and $\gamma$? 
Such an element $\gamma$ is called a \emph{simultaneous ping-pong partner (in $\Gamma$) for the set $\fset$}. 

Following \cite{BekkaCowlingdelaHarpe95}, a discrete group $\Gamma$ is said to have \emph{property \Pnai{}} if any finite subset $\fset$ of $\Gamma$ admits a simultaneous ping-pong partner. 
In 1995, Bekka, Cowling and de la Harpe proved \cite[Theorem 3]{BekkaCowlingdelaHarpe95} that Zariski-dense subgroups of connected simple real Lie groups of real rank $1$ with trivial center, have property \Pnai{}, and asked in \cite[Remark 3]{BekkaCowlingdelaHarpe95} whether the same holds for semisimple groups of arbitrary rank. 
This question was again highlighted by de la Harpe \cite[Question 17]{delaHarpe07}, and we record it here under the following form. 

\begin{namedtheorem*}{\Cref{que:delaHarpe}}[Bekka--Cowling--de la Harpe]
Let ${G}$ be a connected adjoint semisimple real Lie group without compact factors, and let $\Gamma$ be a Zariski-dense subgroup of ${G}$. 
Let $\fset$ be a finite subset of $\Gamma$. 
Does there exist an element $\gamma \in \Gamma$ such that for every $h \in \fset$, the subgroup $\langle h, \gamma \rangle$ is canonically isomorphic to the free product $\langle h \rangle \ast \langle \gamma \rangle$? 
\end{namedtheorem*}

It is well known (see \cite[Lemmas 2.1 \& 2.2]{BekkaCowlingdelaHarpe95}) that property \Pnai{} for $\Gamma$ implies that the reduced $\rC^*$-algebra $\rC^*_\rr(\Gamma)$ of $\Gamma$ is simple and has a unique tracial state; this is, in fact, one of the historical reasons for the interest in property \Pnai{}. 
Over the years, the simplicity and unique trace property of $\rC^*_\rr(\Gamma)$ was established for large classes of groups (see namely \cite{Powers75,delaHarpe07,Poznansky06,BreuillardKalantarKennedyOzawa17}). 
The stronger property \Pnai{} however remains poorly understood. 

Besides \cite[Theorem 3]{BekkaCowlingdelaHarpe95} just mentioned, we are aware of the work of Soifer and Vishkautsan \cite[Theorem 1.3]{SoiferVishkautsan10}, which gives a positive answer for $\Gamma = \PSL_n(\Z)$ and $\fset$ consisting of elements whose semisimple part is either biproximal\footnote{\cite{SoiferVishkautsan10} uses for \emph{biproximal} the term `hyperbolic', whereas \cite{Poznansky06} uses `very proximal'; see \Cref{def:proximal} for the terminology used here. } or torsion. 
Contemporary to this work, a positive answer was claimed by Poznansky in his thesis \cite[Theorem 6.5]{Poznansky06}, for arbitrary finite subsets $\fset$ of a semisimple adjoint algebraic group $\bG$ containing no factor of type $\rA_n, \rD_{2n+1}$ or $\rE_6$. 
Unfortunately the proof of this theorem contains an error, as it relies on \cite[Proposition 2.11]{Poznansky06} whose statement is not true. 
Nonetheless, if one assumes that $\fset$ consists of elements satisfying the conclusion of \cite[Proposition 2.11]{Poznansky06} (that is, of elements whose conjugacy class intersect the big Bruhat cell, see \Cref{rem:poznanskypaper} for more details) and generating a subgroup that embeds in an adjoint simple quotient of $\bG$, then the proof of \cite[Theorem 6.5]{Poznansky06} given by Poznansky goes through as far as we know, and was an instructive source for our work. 
\medbreak

In contrast with the general case, the literature on free subgroups with a prescribed generator in the multiplicative group of a semisimple algebra is more abundant. 
In division algebras for example, Lichtman's conjecture \cite{Lichtman87} has attracted a lot of attention (see \cite{GoncalvesShirvani12,BellGoncalves20} for some recent developments). 

For group algebras, there has been a great deal of research about ping-pong with so-called bicyclic units inside group algebras. 
The conjecture that two {generic} bicyclic units generate a free group, has circulated in this field for now almost three decades, and was the initial motivation for the present work. 
We refer the reader to \cite{GoncalvesdelRio13} for a comprehensive survey (until $2013$), and to \cite{GoncalvesPassman06,GoncalvesdelRio08,GoncalvesdelRio11,GoncalvesGuralnickdelRio14,JespersOlteanudelRio12,Raczka21} for some recent results.

\subsection{Outline}
This article essentially consists of two parts. 
\medskip

In the first half, we consider the variant of Bekka, Cowling and de la Harpe's question in which $\fset$ is actually a finite set of \emph{finite subgroups} in a reductive algebraic group $\bG$. 
In this setting, the statement of \Cref{thm:pingpongdense} gives two conditions which joint imply the existence of simultaneous ping-pong partners for $\fset$ inside an arbitrary Zariski-dense subgroup $\Gamma$ of $\bG$. 
When these conditions are satisfied, we also show that the set of simultaneous ping-pong partners for $\fset$ is both Zariski- and profinitely dense in $\Gamma$. 

Subsequently, in \Cref{thm:pingpongdenseSLn}, we establish these conditions for $\bG$ a product of inner $\R$-forms of $\GL_n$ or $\SL_n$. 
As a corollary, we answer in the affirmative (this slightly stronger version of) Bekka, Cowling and de la Harpe's question for subsets $\fset$ consisting of torsion elements of $\bG(\R)$. 

Special care needs to be taken for torsion elements: the answer to \Cref{que:delaHarpe} when $\bG$ is not simple actually depends on the kind of torsion elements appearing in $\fset$ (this was seemingly overlooked by Bekka, Cowling and de la Harpe). 
Indeed, to even admit a ping-pong partner at all, the elements of $\fset$ must (almost) embed in an adjoint simple quotient of $\bG$ (see \Cref{prop:amalgaminproduct} and \Cref{def:almostembed}). 
\Cref{thm:pingpongdenseSLn} shows that this is in fact sufficient. 

A particular instance of the above setting is when $\bG$ is the multiplicative group of a finite-dimensional semisimple algebra $\Alg$ over a number field. 
In that instance, we show in \Cref{thm:pingpongdensesemisimplealgebra} that ping-pong partners for a finite subgroup $H$ of $\Alg^\times$ exist in an order of $\Alg$ if and only if $H$ almost embeds in a factor of $\Alg$ that is neither a field nor a totally definite quaternion algebra. 
Assuming such an almost embedding exists, \Cref{sec:pingpongsemisimplealgebra} aims at providing a concrete method to construct a ping-pong partner for $H$. 
We introduce in \Cref{subsec:deformations} the convenient formalism of first-order deformations of a subgroup (see \Cref{def:firstorderdeformation}), which will be used to deform finite subgroups $A$ and $B$ of $\Alg^\times$ in order to put them in ping-pong position (\Cref{thm:pingpongdeformations}). 
In passing, we record in \Cref{thm:separabledeformationsinner} that, in line with Hochschild's cohomology theory, such deformations are always obtained via conjugation by specific elements of $\Alg$. 
\medskip

In the second half of the article, we delve into the unit group $\U(\Fi G) = (\Fi G)^\times$ of the group algebra of a finite group $G$. 
With units of group rings arises interesting interplay with the representation theory of $G$: whether a subgroup of $G$ almost embeds in an appropriate quotient translates into a problem about the centers of the irreducible representations of $G$. 
As an illustration, it follows from \Cref{thm:almostfaithfulembedding} that a subgroup of $G$ which embeds center-preservingly in an adjoint simple factor of $\Fi G$, also almost embeds in a simple factor that admits proximal dynamics. 
Another interesting piece related to \Cref{thm:almostfaithfulembedding} is the characterization, given in \Cref{thm:DedekindiffimagesFrobeniuscomplement}, of groups whose image under every irreducible representation is a fixed-point-free group. 

The goal of the second part is to obtain answers to old open problems in the field of group rings. 
There has been an active interest in this field for a particular kind of unipotent elements, the so-called {bicyclic units}, which arise naturally in the study of the group ring $RG$ where $R$ is the rings of integers of $\Fi$. 
For several decades, the belief has been that two generic bicyclic units should generate a free group; we substantiate this claim for their shifted analogues in \Cref{cor:pingpongbicyclicwithdeformation}. 

This is done by first proving that the group of bicyclic units is Zariski-dense in the group of unimodular elements of an appropriate factor of the group ring. 
At this stage, using the main results of the first half of the article, one is reduced to determining whether a given subgroup $H$ of $G$ almost embeds in an appropriate simple factor. 
\Cref{thm:almostfaithfulembedding} establishes the existence of such a factor for subgroups $H$ admitting a faithful irreducible representation, with increasing degrees of precision (fixed-point-free, non-division, or non-quaternionic) in terms of the ambient group $G$ and the field $\Fi$. 
\medbreak

We now give a more detailed account of the main results.

\subsection{Simultaneous ping-pong with finite subgroups}

After a short recollection on the structure of free amalgamated products in \Cref{sec:generalamalgams}, we will consider in \Cref{sec:pingpongreductivegroups} a slightly more general version of \Cref{que:delaHarpe}, in which we allow the finite set $\fset$ to consist of finite subgroups (not just elements). 
In fact, our results will also encompass the situation in which both elements of $\fset$ and their ping-pong partners are (conjugates of) prescribed finite subgroups. 
To this end, we study the dynamics of linear transformations on projective spaces over division algebra; the details are contained in \Cref{subsec:proximaldynamics}, and are mostly a rework of classical results of Tits, themselves already revisited by several authors. 
The main result of that section is \Cref{prop:proximaldense}, stating the abundance of simultaneously biproximal elements (when they exist). 

These developments are necessary to prove the main result of \Cref{sec:pingpongreductivegroups}: 

\begin{namedtheorem*}{\Cref{thm:pingpongdense}}
Let $\bG$ be a connected algebraic $\Fi$-group with center $\bZ$. 
Let $\Gamma$ be a Zariski-connected subgroup of $\bG(\Fi)$. Let $(A_i, B_i)_{i \in I}$ be a finite collection of finite subgroups of $\bG(\Fi)$. 
Suppose that for each $i \in I$ there exists a local field $\LFi_i$ containing $\Fi$ and a projective $\LFi_i$-representation $\rho_i: \bG \to \PGL_{V_i}$, where $V_i$ is a finite-dimensional module over a finite-dimensional division $\LFi_i$-algebra $\LDi_i$, with the following two properties: 
\begin{enumerate}[leftmargin=3cm,itemsep=1ex,topsep=1ex]
	\item[\textup{(Proximality)}] $\rho_i(\Gamma)$ contains a proximal element;
	\item[\textup{(Transversality)}] For every $h \in (A_i \cup B_i) \setminus \bZ(F)$ and every $p \in \P(V_i)$, the span of the set $\{\rho_i(xhx^{-1}) p \mid x \in \Gamma \}$ is the whole of $\P(V_i)$.
\end{enumerate}

Let $S$ be the collection of regular semisimple elements $\gamma \in \Gamma$ of infinite order, such that for all $i \in I$, the canonical maps
\begin{align*}
\langle \gamma, C_{A_i} \rangle \ast_{C_{A_i}} A_i &\to \langle \gamma, A_i \rangle && \text{where $C_{A_i} = A_i \cap \bZ(F)$,}\\
\langle \gamma, C_{B_i} \rangle \ast_{C_{B_i}} B_i &\to \langle \gamma, B_i \rangle && \text{where $C_{B_i} = B_i \cap \bZ(F)$,}\\
\intertext{and provided $\indx{A_i}{C_{A_i}} > 2$ or $\indx{B_i}{C_{B_i}} > 2$, also}
\langle A_i, C_i \rangle \ast_{C_i} \langle B_i, C_i \rangle &\to \langle A_i, B_i^{\gamma} \rangle && \text{where $C_i = C_{A_i} \cdot C_{B_i}$,}
\end{align*}
are all isomorphisms. 
Then $S$ is dense in $\Gamma$ for the join of the profinite topology and the Zariski topology. 
\end{namedtheorem*}

The transversality condition gets its name for its role in \Cref{lem:transversality}. 
As hinted in \Cref{lem:transversality} and \Cref{rem:transversalitycondition}, transversality is a kind of ``higher irreducibility'' condition. 
\medbreak

Next, in \Cref{subsec:constructingproximaltransverse} we verify the proximality and transversality conditions for finite subgroups in products of inner forms of $\SL_n$ and $\GL_n$. 
In light of \Cref{thm:pingpongdense}, this proves the abundance of simultaneous ping-pong partners in Zariski-dense subgroups in this setting: 

\begin{namedtheorem*}{\Cref{thm:pingpongdenseSLn}}
Let $\bG$ be a reductive $\R$-group with center $\bZ$. 
Let $\Gamma$ be a subgroup of $\bG(\R)$ whose image in $\Ad \bG$ is Zariski-dense. 
Let $(A_i, B_i)_{i \in I}$ be a finite collection of finite subgroups of $\bG(\R)$. 

Suppose that for each $i \in I$, there exists a quotient of $\bG$ of the form $\PGL_{\LDi_i^{n_i}}$ for $\LDi_i$ some finite-dimensional division $\R$-algebra and $n_i \geq 2$, for which the kernels of the restrictions $A_i, B_i \to \PGL_{\LDi_i^{n_i}}(\R)$ are contained in $\bZ(\R)$. 
Then the set of regular semisimple elements $\gamma \in \Gamma$ of infinite order such that for all $i \in I$, the canonical maps
\begin{align*}
\langle \gamma, C_{A_i} \rangle \ast_{C_{A_i}} A_i &\to \langle \gamma, A_i \rangle && \text{where $C_{A_i} = A_i \cap \bZ(F)$,}\\
\langle \gamma, C_{B_i} \rangle \ast_{C_{B_i}} B_i &\to \langle \gamma, B_i \rangle && \text{where $C_{B_i} = B_i \cap \bZ(F)$,}\\
\intertext{and provided $\indx{A_i}{C_{A_i}} > 2$ or $\indx{B_i}{C_{B_i}} > 2$, also}
\langle A_i, C_i \rangle \ast_{C_i} \langle B_i, C_i \rangle &\to \langle A_i, B_i^{\gamma} \rangle && \text{where $C_i = C_{A_i} \cdot C_{B_i}$,}
\end{align*}
are all isomorphisms, is dense in $\Gamma$ for the join of the profinite topology and the Zariski topology. 
\end{namedtheorem*}

Given a reductive $\Fi$-group $\bG$ with center $\bZ$ and a subgroup $H \leq \bG(\Fi)$, we say that \emph{$H$ almost embeds in a (simple, adjoint) quotient $\bQ$ of $\bG$} if there exists a (simple, adjoint) quotient $\bQ$ of $\bG$ for which the kernel of the restriction $H \to \bQ(\Fi)$ is contained in $\bZ(\Fi)$. 
As will be mentioned in \Cref{rem:almostembeddingnecessary}, this is a necessary condition for the subgroup $H$ to admit a ping-pong partner in $\bG(\Fi)$.

\subsection{The multiplicative group of a semisimple algebra}
In \Cref{sec:pingpongsemisimplealgebra}, we aim to provide more precise results for certain finite subgroups $H$ and Zariski-dense subgroups $\Gamma$ of the multiplicative group $\bG$ of a finite-dimensional semisimple $\Fi$-algebra $\Alg$. 
We also investigate whether two given finite subgroups $A$ and $B$ of $\bG(\Fi) = \Alg^\times$ can appear as the factors of a free product $A \ast B$ inside $\bG(\Fi)$. 
Despite their abundance, the results from \Cref{sec:pingpongreductivegroups} do not give an explicit construction of simultaneous ping-pong partners. 
To this end, we introduce in \Cref{subsec:deformations} \emph{first-order deformations} of $H$, which are linear deformations of $H$ suitable to obtain quasi-proximal dynamics. 
The main result of that section is \Cref{thm:pingpongdeformations}. 

By the Artin--Wedderburn theorem, $\Alg$ is isomorphic as $\Fi$-algebra to
\[
\Alg \cong \Mat_{n_1}(\Di_1) \times \cdots \times \Mat_{n_m}(\Di_m),
\]
for $D_i$ some finite-dimensional division algebras over $\Fi$. 
In particular, the $\Fi$-group $\U(A) = A^\times$ of units of $A$ is 
the reductive group
\(
\bG \cong \GL_{\Di_{1}^{n_1}} \times \cdots \times \GL_{\Di_{m}^{n_m}}. 
\)
If $\Fi$ is a number field and $\mathcal{O}$ is an order in $A$, then by a classical result of Borel and Harish-Chandra, $\Gamma = \U(\O)$ is an arithmetic subgroup of $\U(A)$, placing us in the setting of \Cref{thm:pingpongdenseSLn}. 
We derive from it the following criterion for the existence of simultaneous ping-pong partners for finite subgroups of $\U(A)$. 

\begin{namedtheorem*}{\Cref{thm:pingpongdensesemisimplealgebra}}
Let $\Gamma$ be a Zariski-dense subgroup of $\U(\O)$. 
Let $A$ and $B$ be finite subgroups of $\U(\O)$, and assume that $\indx{A}{A \cap \ZZ(\Alg)} > 2$. 

There exists $\gamma \in \Gamma$ of infinite order with the property that the canonical maps
\begin{align*}
\langle \gamma, C_{A} \rangle \ast_{C_{A}} A &\to \langle \gamma, A \rangle && \text{where $C_{A} = A \cap \ZZ(\Alg)$,}\\
\langle \gamma, C_{B} \rangle \ast_{C_{B}} B &\to \langle \gamma, B \rangle && \text{where $C_{B} = B \cap \ZZ(\Alg)$,}\\
\langle A, C \rangle \ast_{C} \langle B, C \rangle &\to \langle A, B^{\gamma} \rangle && \text{where $C = C_{A} \cdot C_{B}$,}
\end{align*}
are all isomorphisms, if and only if $A$ and $B$ almost embed in $\PGL_{\Di_i^{n_i}}$ for some common index $i$ for which $\Mat_{n_i}(\Di_i)$ is neither a field nor a totally definite quaternion algebra. 

Moreover, in the affirmative, the set of such elements $\gamma$ which are regular semisimple and of infinite order, is dense in $\Gamma$ for the join of the Zariski and the profinite topology. 
\end{namedtheorem*}

\subsection{Center-preserving representations}
In the specific case where $\Alg = \Fi G$ is a group algebra, the simple components $\Mat_{n_i}(\Di_i)$ of $\Alg$ are not arbitrary. 
Indeed, each $\Mat_{n_i}(\Di_i)$ is precisely the span of the image of $G$ under the irreducible $\Fi$-representation corresponding to the $i$th factor of this decomposition. 
In other words, the simple factors of $\Fi G$ are dictated by the representation theory of one common finite group $G$. 
Call a representation $\rho$ of $G$ \emph{center-preserving on $H$} if $H \cap \ZZ(\rho) = H \cap \ZZ(G)$; if $\rho$ is irreducible, this is another way of saying that $H$ almost embeds in the corresponding adjoint simple factor $\PGL_{\rho(\Fi G)}$. 
As an illustration of the impact of this additional group ring structure, we show that the condition of $H$ almost embedding in an appropriate simple quotient (the kind appearing in \Cref{thm:pingpongdensesemisimplealgebra}) can be deduced from $G$ admitting an irreducible representation which is center-preserving on $H$, without further constraints on (the span of) the representation. 
This is a consequence of the main result of \Cref{sec:faithfulembedding}:

\begin{namedtheorem*}{Excerpt of \Cref{thm:almostfaithfulembedding}}
Let $\Fi$ be a number field. 
Let $G$ be a non-abelian finite group and $H \leq G$ be a subgroup admitting a faithful irreducible $\Fi$-representation. 
If $G$ and $H$ satisfy \Cref{cond:NQ}, and
\begin{enumerate}[itemsep=1ex,topsep=1ex,label=\textup{(\roman*)},leftmargin=2em]
\item[\textup{(iii)}] $\Fi$ is not totally real, or $G$ is not isomorphic to $Q_8 \times C_2^n$ for any $n \in \N$,
\end{enumerate}
then there exists an irreducible $\Fi$-representation $\rho$ of $G$ which is center-preserving on $H$ and such that 
\begin{enumerate}[itemsep=1ex,topsep=1ex,label=\textup{(\roman*)},leftmargin=2em]
	\item[\textup{(iii)}] $\rho(\Fi G)$ is neither a field nor a totally definite quaternion algebra. 
\end{enumerate}
\end{namedtheorem*}

Note that on an arbitrary subgroup (even on a cyclic one satisfying \ref{cond:NQ}), $\Fi G$ need not admit a center-preserving representation.  
In that sense, the statement of \Cref{thm:almostfaithfulembedding} is truly representation-theoretical, and so will be its proof. 

Combined with \Cref{thm:pingpongdensesemisimplealgebra}, \Cref{thm:almostfaithfulembedding} implies under the same hypotheses the existence of free amalgamated products in $\U(RG)$.

\begin{namedtheorem*}{\Cref{cor:existenceamalgaminordersharp}}
Let $R$ be the ring of integers of the number field $\Fi$. 
Let $G$ be a finite group, and $H$ be a subgroup of $G$ which admits a faithful irreducible $\Fi$-representation. 
If $G$ and $H$ satisfy \Cref{cond:NQ}, and either $\Fi$ is not totally real, or $G \ncong Q_8 \times C_2^n $ for any $n \in \N$, then the set of elements $\gamma \in \U(RG)$ of infinite order such that
\[
\langle \gamma, H \rangle \cong \langle \gamma, C \rangle \ast_C H, \quad \text{with $C = H \cap \ZZ(G)$},
\]
is dense in $\U(RG)$ for the join of the Zariski and profinite topologies. 
\end{namedtheorem*}

Note that when $G \cong Q_8 \times C_2^n$ and $\Fi$ is totally real, $\U(R G)$ consists of only the obvious units, hence contain no non-trivial free amalgamated products. 
Slightly weaker results are provided when \ref{cond:NQ} does not hold. 
For more information about \Cref{cond:NQ}, the reader should consult \Cref{sec:faithfulembedding} (namely \Cref{rem:conditionNQ}).

\subsection{Bicyclic units generate free groups}
Finally, we focus on the construction of free products with certain specific units in $\U(RG)$. 
Unipotent elements of $RG$ of the form
\[
\bic{h,x} = 1 + \wt{h} x (1-h) \quad \text{and} \quad \bic{x,h} = 1 + (1-h)x \wt{h}, \quad \text{for $x \in RG$, $h \in G$, and $\wt{h} := \textstyle{\sum_{i=1}^{o(h)} h^{i}}$}
\]
are called \emph{bicyclic units}. 
Let $\Bic_R(G) = \langle \bic{h,x}, \bic{x,h} \mid h \in G, x \in RG \rangle$ be the group they generate. Bicyclic units constitute one of the few generic constructions known for units in $RG$.

The conjecture that two {generic} bicyclic units generate a free group, has circled in the field of group rings for almost 30 years. That being said, the meaning to give to `generic' has, to our knowledge, never been made precise. This conjecture was the original motivation for our work.
\medbreak

To start, we propose a new point of view on bicyclic units, as being first-order deformations of torsion units. 
These deformations are group morphisms of the form
\begin{equation*} 
\fod : H \rightarrow \U(\Fi G): h \mapsto \fod(h) = h + \delta_h,
\end{equation*}
see \Cref{def:firstorderdeformation}. 
(Taking $\delta_h = \tilde{h}x(1-h)$ or $(1-h) x \tilde{h}$ recovers bicyclic units.) 
We prove in \Cref{thm:separabledeformationsinner} that any first-order deformation is given by conjugation in the ambient algebra, and in \Cref{thm:pingpongdeformations} that they can be used to construct free products between finite subgroups of $\Fi G$. 

In the literature it has often been an obstacle that, when $\Fi G$ has certain low-rank simple factors, $\Bic_R(G)$ can be of infinite index in $\SL_1(RG)$. 
Nevertheless, \Cref{lem:bicyclicZariskidense} shows that $\Bic_R(G)$ is always Zariski-dense in $\prod \SL_1(\rho(\Fi G))$, where the product is taken over all irreducible representations of $G$ which are not fixed-point-free. 
In consequence, to find ping-pong partners for a bicyclic unit, one should look for irreducible representations whose images are not fixed-point-free. 

Our main contribution in this direction states that if one replaces bicyclic units by the closely related \emph{shifted bicyclic units}, the aforementioned long-standing conjecture holds true for a profinitely and Zariski-dense subset of bicyclic units. 

\begin{namedtheorem*}{\Cref{cor:pingpongbicyclicwithdeformation}}
Let $G$ be a finite group, and pick $x \in RG$. 
Let $h \in G$ have no non-trivial central powers. 
Then the set of elements $\alpha \in \Bic_R(G)$ for which the canonical map
\[
\langle \alpha \rangle \ast \langle h \bic{h,x} \rangle \to \langle \alpha, h \bic{h,x} \rangle
\]
is an isomorphism, is dense in $\Bic_R(G)$ for the join of the profinite and Zariski topologies. 
\end{namedtheorem*}

Much like \Cref{cor:existenceamalgaminordersharp}, \Cref{cor:pingpongbicyclicwithdeformation} is a combination of two results, both already mentioned and interesting in their own right. 
The first is \Cref{thm:almostfaithfulembedding}, which under reasonable assumptions, guarantees the existence of center-preserving representations that are not fixed-point-free. 
The second is \Cref{thm:pingpongdenseSLn}, providing the general machinery to exhibit free products once one has the appropriate representation at hand. 

Altogether, \Cref{cor:pingpongbicyclicwithdeformation} turns out to be an unavoidable but fruitful combination of the structure and dynamics of infinite linear groups, and the representation theory of finite groups.

\subsection*{Acknowledgments} 
Andreas B\"achle was involved in earlier stages of this project, and we thank him heartily for many helpful discussions and enjoyable moments.
We thank Emmanuel Breuillard and \v{S}pela \v{S}penko for helpful conversations on \Cref{sec:pingpongreductivegroups}. 
We also thank Leo Margolis for his is interest and some clarifications around the Zassenhaus conjecture. 
Furthermore, we are grateful to Jean-Pierre Tignol for providing us with a better proof of \Cref{thm:separabledeformationsinner}, and to Pierre-Emmanuel Caprace, Miquel Mart\'inez and Justin Vast for useful discussions about \Cref{sec:faithfulembedding}. 
Lastly, the second author wishes to thank the late Jairo Zacarias Gon\c{c}alves for his hospitality during a visit at the University of S\~{a}o Paulo, where \Cref{sec:pingpongsemisimplealgebra} was started.

\section{Amalgams in almost direct products} \label{sec:generalamalgams}
\addtocontents{toc}{\protect\setcounter{tocdepth}{2}}

In this section, we recall the classical ping-pong lemma for amalgamated products. 
Thereafter we exhibit a necessary condition for a subgroup of an almost direct product to be an amalgamated product. 
\medskip

Given a subgroup $C$ of a group $G$, we will denote by $T_C^G$ a set of representatives of the left cosets of $C$ in $G$, containing the identity element. 
\smallskip

The ping-pong lemma for amalgams and its variant for HNN extensions can be found in \cite[Propositions 12.4 \& 12.5]{LyndonSchupp01}. 
For the convenience of the reader, we provide a proof as it will be instrumental in the rest of this paper.

\begin{lemma}[Ping-pong for amalgams]\label{lem:pingpongamalgam}
Let $A$, $B$ be subgroups of a group $G$ and suppose $C = A \cap B$ satisfies $\indx{A}{C} > 2$. 
Let $G$ act on a set $X$. 
If $P, Q \subset X$ are two subsets with $P \not\subset Q$, such that for all elements $a \in T_C^A \setminus \{e\}$, $b \in T_C^B \setminus \{e\}$ and $c \in C$, we have
\[
aP \subset Q, \quad bQ \subset P, \quad cP \subset P, \quad \text{and} \quad cQ \subset Q,
\]
then the canonical map $A \ast_C B \rightarrow \langle A,B \rangle $ is an isomorphism.
\end{lemma}

As in the case of free products, the proof of \Cref{lem:pingpongamalgam} is straightforward once one knows the normal form for elements in an amalgamated product. 
The normal form also allows us to unambiguously speak of \emph{words starting with $A$} and \emph{words starting with $B$}. 
In the next lemma, these are the elements for which $\dot{a}_1 \notin C$, resp.\ for which $\dot{a}_1 \in C$. 

\begin{lemma}[Normal form in free amalgamated products]\label{lem:amalgamnormalform}
Let $A, B \leq G$ be groups and $C \leq A \cap B$. 
The following are equivalent. 
\begin{enumerate}[itemsep=1ex,topsep=1ex,label=\textup{(\roman*)},leftmargin=2em]
	\item {The canonical map $A \ast_C B \rightarrow \langle A,B \rangle$ is an isomorphism.} 
	\item {Every element in $\langle A,B \rangle$ has a unique decomposition of the form $\dot{a}_1 b_1 \cdots a_n \dot{b}_n c$, 
where $a_i\in T_C^A\setminus \{e\}$, $b_i \in T_C^B\setminus \{e\}$, $\dot{a}_1 \in T_C^A$, $\dot{b}_n \in T_C^B$, and $c \in C$.} 
	\item {Given $a_i \in A \setminus C$, $b_i \in B \setminus C$, $\dot{a}_1 \in A$, and $\dot{b}_n \in B$, the product $\dot{a}_1 b_1 \cdots a_n \dot{b}_n$ belongs to $C$ only if $n =1$ and $\dot{a}_1, \dot{b}_n \in C$.}
\end{enumerate}
In consequence of the affirmative, $C = A \cap B$.
\end{lemma}
\begin{proof}[Sketch of proof.]
The implication (i)$\implies$(ii) is the existence and uniqueness of a normal form (see for instance \cite[Theorem IV.2.6]{LyndonSchupp01}), and its converse amounts to checking the injectivity of the canonical map, which follows from the uniqueness of the decomposition in $\langle A, B \rangle$. 

After replacing $\dot{b}_n, a_n, \dots, b_1, \dot{a}_1$ by the appropriate coset representatives, (ii)$\implies$(iii) becomes obvious. For the contrapositive of its converse, note that two different decompositions of an element in $\langle A, B \rangle$ result in a non-trivial expression of the form $\dot{a}_1 b_1 \cdots a_n \dot{b}_n$ in $C$.
\end{proof}

\begin{proof}[Proof of \Cref{lem:pingpongamalgam}]
Note that the assumptions imply that $a P_1 \subset P_2$ for all $a \in A\setminus C$, $b P_2 \subset P_1$ for all $b \in B \setminus C$, and $c P_1 = P_1$, $cP_2 = P_2$ for every $c \in C$. 

Suppose that given $a_i \in A \setminus C$, $b_i \in B \setminus C$, $\dot{a}_1 \in A$ and $\dot{b}_n \in B$, the non-empty word $c = \dot{a}_1 b_1 \cdots a_n \dot{b}_n$ lies in $C$.
The possible cases for $\dot{a}_1$ and $\dot{b}_n$ to belong to $C$ are: 
\begin{enumerate}[itemsep=1ex,topsep=1ex,label=$\bullet$,leftmargin=\parindent]
	\item $\dot{a}_1 \notin C$, $\dot{b}_n \in C$. 
We have $\dot{b}_n P_1 = P_1$, $a_n \dot{b}_n P_1 \subset P_2$, $b_{n-1} a_n \dot{b}_n P_1 \subset P_1$, etc., so that eventually, $c P_1 = \dot{a}_1 b_1 \cdots a_n P_1 \subset P_2$. 
Since $c P_1 = P_1$ and $P_1 \not\subset P_2$, this case cannot occur. 
	\item $\dot{a}_1 \in C$, $\dot{b}_n \notin C$. 
Pick $a \in A \setminus C$, and let $a' \in A$ and $c' \in C$ be such that $a^{-1} c a = a' c'$. 
We have $a a' \notin C$, hence the word $c' = (a a')^{-1} b_1 \cdots a_n \dot{b}_n a$ starts and ends with an element of $A \setminus C$. 
This case thus reduces to the first one. 
	\item $\dot{a}_1 \notin C$, $\dot{b}_n \notin C$. 
As $\indx{A}{C} > 2$, we may pick $a \in A \setminus (C \cup \dot{a}_1 C)$, so that $a^{-1} \dot{a}_1 P_1 \subset P_2$ hence $\dot{a}_1 P_1 \subset a P_2$.
As in the first case, we have $c P_2 \subset \dot{a}_1 P_1$. 
Since $c P_2 = P_2$, this would imply $a P_1 \subset P_2 \subset a P_2$, hence this case does not occur either. 

	\item $\dot{a}_1 \in C$, $\dot{b}_n \in C$. 
If $n > 1$, replacing $c$ by $c^{-1}$ reduces to the third case. 
The only remaining possibility is thus $n=1$ and $\dot{a}_1,\dot{b}_n \in C$, as expected. 
\end{enumerate} 
We conclude from \Cref{lem:amalgamnormalform} that the canonical map $A \ast_C B \to \langle A,B \rangle$ is an isomorphism. 
\end{proof}

\begin{lemma}\label{lem:preimageamalgam}
Let $A \ast_{C} B$ be a free amalgamated product. If $f$ is a surjective morphism from a group $\Gamma$ to $A \ast_{C} B$, then $\Gamma$ is the free product with amalgamation $f^{-1}(A) \ast_{f^{-1}(C)} f^{-1}(B)$. 

If moreover $\Gamma$ is generated by two subgroups $\Gamma_1, \Gamma_2$ with the properties $f(\Gamma_1) \subseteq A$, $f(\Gamma_2) \subseteq B$, the induced map $\Gamma_1 \to A / C$ is injective, and $\Gamma_1 (\Gamma_2 \cap f^{-1}(C))$ is a subgroup, then $f^{-1}(B) = \Gamma_2$ and $\Gamma \cong (\Gamma_1 f^{-1}(C)) \ast_{f^{-1}(C)} \Gamma_2$. 
\end{lemma}

\begin{proof}
The first part of the lemma is standard (see for instance \cite[Lemma 3.2]{Zieschang76}). 
For the second part, let $g = b_0 a_1 b_1 \cdots a_n b_n$ with $a_i \in \Gamma_1 \setminus \{e\}$ and $b_i \in \Gamma_2$, be an element of $f^{-1}(B)$. 
Since $(\Gamma_2 \cap f^{-1}(C)) \Gamma_1 = \Gamma_1 (\Gamma_2 \cap f^{-1}(C))$, after perhaps reducing the expression for $g$, we may assume that $b_i \notin f^{-1}(C)$ for $0 < i < n$. 
Because $f(\Gamma) = A \ast_C B$ and $f(b_i) \in B \setminus C$, \Cref{lem:amalgamnormalform} implies that $n = 0$, hence $g = b_0 \in \Gamma_2$. Thus $f^{-1}(B) = \Gamma_2$, and in consequence $f^{-1}(C) \leq \Gamma_2$. 
On the other hand, if $g = b_0 a_1 b_1 \cdots a_n b_n \in f^{-1}(A)$, we may assume as before that $a_i \neq e$ for $1 \leq i \leq n$ and $b_i \notin f^{-1}(C)$ for $0 < i < n$. 
Applying $f$ again then shows that $n \leq 1$ and $b_i \in f^{-1}(C)$ for $i \leq n$, so that $g \in \Gamma_1 f^{-1}(C)$. 
\end{proof}

The following folkloric terminology is inspired by \Cref{lem:pingpongamalgam}.

\begin{definition}
Let $A$ and $B$ be subgroups of a group $G$. 
We say that $A$ is \emph{a ping-pong partner for $B$}, or that \emph{$A$ and $B$ play ping-pong}, if the subgroup $\langle A, B \rangle$ is freely generated by $A$ and $B$, or in other words if the canonical map
\(
A \ast B \to \langle A, B \rangle 
\)
is an isomorphism. 
Similarly, we say that $a \in A$ is \emph{a ping-pong partner for $B$ in $A$}, or that \emph{$a$ and $B$ play ping-pong}, if the subgroup $\langle a, B \rangle$ is freely generated by $\langle a \rangle$ and $B$; 
same for $a \in A$ and $b \in B$ generating $\langle a, b \rangle$ freely. 
We will also use this terminology for free amalgamated products when the subgroup being amalgamated is unambiguous. 
\end{definition}

In the subsequent sections, we will look to play ping-pong inside a group $G =\prod_{i=1}^n G_i$ which decomposes into a direct product of subgroups $G_i$. 
Using some simple facts about free (amalgamated) products, the next proposition will show that this requires an embedding of the ping-pong partners in one of the factors $G_i$. 

Given subgroups $H_1, \dots, H_n$ of a group $G$, let $[H_1, \dots, H_n] = [H_1, [H_2, \dots, H_n]]$ denote the \emph{left-iterated (or right-normed) commutator subgroup} of the $H_i$. 

\begin{lemma}
\label{lem:normalsubgroupinamalgam}
Let $N$, $N_1$, \dots, $N_n$ be normal subgroups of $A \ast_C B$, where $\indx{A}{C} > 2$. 
\begin{enumerate}[itemsep=1ex,topsep=1ex,label=\textup{(\roman*)},leftmargin=2em]
	\item Either $N \subset C$, or $N$ contains a non-abelian free group. 
	\item If $[N_1, N_2] \subset C$, then either $N_1 \subset C$ or $N_2 \subset C$. 
\end{enumerate}
In consequence, if $[N_1, \dots, N_n]$ admits no non-abelian free subgroups, there exists $i \in \{1, \dots, n\}$ for which $N_i \subset C$. 
\end{lemma}
\begin{proof}
First, suppose that $N$ is a normal subgroup of $A \ast_C B$ not contained in $C$. 
Pick $x \in N \setminus C$; by \Cref{lem:amalgamnormalform}, we may assume after conjugation that $x$ either belongs to $B \setminus C$, belongs to $A \setminus C$, or is cyclically reduced starting with $a_1 \in A \setminus C$. 

\begin{enumerate}[itemsep=1ex,topsep=1ex,label=$\bullet$,leftmargin=\parindent]
	\item If $x \in B \setminus C$, pick $a, a' \in A \setminus C$ such that $a \notin a'C$. 
Using \Cref{lem:amalgamnormalform}, one readily checks that the cyclically reduced words $w = [a,x]$ and $w' = [a',x]$ generate a free group, as every non-empty word in $w$ and $w'$ remains a non-empty word alternating in elements of $A \setminus C$ and $B \setminus C$. 
(Only simplifications of the form $[a, x] [a',x]^{-1}= a x (a^{-1} a') x^{-1} a'^{-1}$ occur, and the condition on $a$ and $a'$ ensures no further cancellations arise.) 
	\item If $x \in A \setminus C$, pick $b \in B \setminus C$ and $a, a' \in A \setminus C$ such that $a \notin a'C$, and consider $w = [x,bab^{-1}]$ and $w' = [x,ba'b^{-1}]$ instead. 
	\item In the last case, write $x = a_1 b_1 \cdots, a_n b_n$ with $n \geq 1$ and $a_i \in A \setminus C$, $b_i \in B \setminus C$. 
Pick $b \in B \setminus C$ and $a \in A \setminus C$ such that $a \notin a_1 C$. 
Then the words $w = x$ and $w' = aba^{-1} x ab^{-1}a^{-1}$ generate a free group. 
\end{enumerate}
This proves part (i).
\smallbreak

Second, suppose that there exist elements $x \in N_1 \setminus C$ and $x' \in N_2 \setminus C$. 
By \Cref{lem:amalgamnormalform}, we may assume after conjugation that $x$, $x'$ either belong to $A \setminus C$, belongs to $B \setminus C$, or is cyclically reduced starting with $A$. 
We exhibit in each case a commutator in $[N_1, N_2] \setminus C$. 

\begin{enumerate}[itemsep=1ex,topsep=1ex,label=$\bullet$,leftmargin=\parindent]
	\item If $x = a_1$ and $x' = b_1'$, then $[x, x'] \notin C$. 
	\item If $x$ is cyclically reduced starting with $a_1$ and $x' = a_1'$, then $[x, b x'b^{-1}] \notin C$ for any $b \in B \setminus (C \cup b_n^{-1} C)$. 
	\item If $x$ is cyclically reduced starting with $a_1$ and $x' = b_1'$, then $[a^{-1} x a, x'] \notin C$ for any $a \in A \setminus (C \cup a_1 C)$. 
	\item If $x = a_1$ and $x' = a_1'$, then $[x,bx'b^{-1}] \notin C$ for any $b \in B \setminus C$. 
	\item If $x = b_1$ and $x' = b_1'$, then $[axa^{-1}, x'] \notin C$ for any $a \in A \setminus C$. 
	\item If $x$, $x'$ are both cyclically reduced starting with $a_1$, and ending with $b_{n'}'$ respectively, then $[a^{-1} x a, b^{-1} x' b] \notin C$ for any $a \in A \setminus (C \cup a_1C)$ and $b \in B \setminus (C \cup b_{n'}'^{-1}C)$. 
\end{enumerate}
This proves part (ii). 
\smallbreak

Lastly, if $[N_1, \dots, N_n]$ admits no non-abelian free subgroups, we deduce from part (i) that $[N_1, \dots, N_n] \subset C$. 
Part (ii) then implies that either $N_1 \subset C$, or $[N_2, \dots, N_n] \subset C$, and recursively, that eventually $N_i \subset C$ for some $i \in \{1, \dots, n\}$. 
\end{proof}

\begin{definition} \label{def:almostdirectproduct}
Let $\cS$ be a class of groups closed under subquotients and extensions. 
For the purposes of the following proposition, we will say that $G$ is an \emph{$\cS$-almost direct product of $G_1, \dots, G_n$} if $G$ has a normal subgroup $K \in \cS$ such that $G / K$ is the direct product $G_1 \times \cdots \times G_n$. 

Equivalently, if there are normal subgroups $M_1, \dots, M_n$ of $G$ such that $\textstyle \bigcap_{i=1}^n M_i \in \cS$ and $M_i(M_{i+1} \cap \cdots \cap M_n) = G$ for $i = 1, \dots, n-1$, then $G$ is the {$\cS$-almost direct product of $G/M_1, \dots, G/M_n$}. 
Indeed, the second condition ensures that the canonical map
\(
G / \textstyle \bigcap_{i=1}^n M_i \to G/M_1 \times \cdots \times G/M_n
\)
is surjective; conversely, writing $M_i$ for the kernel of $G \to G_i$, it is obvious that $K = \textstyle \bigcap_{i=1}^n M_i$ and $M_j (\textstyle \bigcap_{i\neq j} M_i) = G$. 
\end{definition}

Almost direct products with respect to the class containing only the trivial group are just direct products. 
In the literature, almost direct products appear most often for $\cS$ the class of finite groups. 
Here are a few straightforward observations: 

\begin{enumerate}[itemsep=1ex,topsep=1ex,label=\textup{(\roman*)},leftmargin=2em]
	\item If $\cS \subset \cS'$, then any $\cS$-almost direct product is an $\cS'$-almost direct product (of the same groups). 
	\item Any group in $\cS$ is an empty $\cS$-almost direct product; so of course the notion is meaningful only for groups outside of $\cS$. 
	\item An $\cS$-almost direct product of groups $G_1, \dots, G_n$ themselves $\cS$-almost direct products of respectively $H_{i 1}, \dots, H_{i n_j}$ ($i=1, \dots, n$), is an $\cS$-almost direct product of the $H_{i j}$, $i=1, \dots, n$, $j=1, \dots, n_j$. 
	\item Any quotient or extension of an $\cS$-almost direct product by a group in $\cS$ is again an $\cS$-almost direct product. 
\end{enumerate}

Sometimes, almost direct products are defined by the following variant: 
$G$ is the quotient of a direct product $G_1 \times \cdots \times G_n$ by a normal subgroup $H \in \cS$. 
An almost direct product in this second sense is also an $\cS$-almost direct product in the sense of \Cref{def:almostdirectproduct}. 
Indeed, if $G = (G_1 \times \cdots \times G_n) / H$, denoting $\pi_i$ the projection onto $G_i$ and $K = \pi_1(H) \times \cdots \times \pi_n(H)$, we see that $G / (K/H) \cong (G_1 \times \cdots \times G_n) / K = G_1 / \pi_1(H) \times \cdots \times G_n / \pi_n(H)$. 
The converse however does not always hold, as the images of the factors $G_i$ in $(G_1 \times \cdots \times G_n) / H$ are commuting normal subgroups, and this may not happen in $G$ even if $G/K$ is a direct product. 

\begin{proposition}[Free subgroups in almost direct products]\label{prop:amalgaminproduct}
Let $\cS$ be the class of groups not containing a non-abelian free group. 
Let $G$ be the $\cS$-almost direct product of groups $G_1, \dots, G_m$, and suppose that $G_{n+1}, \dots, G_m$ belong to $\cS$. 
If $A$ and $B$ are subgroups of $G$ whose intersection $C$ satisfies $\indx{A}{C} > 2$, and are such that the canonical map $A \ast_C B \to \langle A, B \rangle$ is an isomorphism, then there exists $i \in \{1, \dots, n\}$ for which the kernel of the projection $\langle A, B \rangle \to G_i$ is contained in $C$. 
\end{proposition}
\begin{proof}
Since $G_{n+1}, \dots, G_m$ belong to $\cS$, it is clear that $G$ is also the $\cS$-almost direct product of $G_1, \dots, G_n$. 
Let $\pi_i$ denote the projection $G \to G_i$ and set $M_i = \ker \pi_i$. 
Identify $\langle A, B \rangle$ with $A \ast_C B$ and set $N_i = M_i \cap (A \ast_C B)$. 

By assumption, $\textstyle \bigcap_{i=1}^n M_i$ does not contain a non-abelian free group. 
The same then holds for $[N_1, \dots, N_n] \subset [M_1, \dots, M_n] \subset \bigcap_{i=1}^n M_i$, and \Cref{lem:normalsubgroupinamalgam} implies the existence of an index $i \in \{1, \dots, n\}$ for which $N_i \subset C$. 
\end{proof}

There are versions of \Cref{lem:normalsubgroupinamalgam} and \Cref{prop:amalgaminproduct} for HNN extensions. 
We leave their statement and proof to the reader.

\section{Simultaneous ping-pong partners for finite subgroups of reductive groups} \label{sec:pingpongreductivegroups}

Let $\Fi$ be a field. Let $\bG$ be a reductive\footnote{In this paper, all reductive (in particular, all semisimple) algebraic groups are connected by definition. This convention sometimes differs in the literature. 
We also call \emph{simple} a non-commutative algebraic group whose proper normal subgroups are finite (sometimes called `almost simple' in the literature). } 
algebraic $\Fi$-group, $\Gamma$ a Zariski-dense subgroup of $\bG(\Fi)$, and $H$ a finite subgroup of $\bG(\Fi)$. 
This section is concerned with finding elements $\gamma$ of $\Gamma$ which are ping-pong partners for $H$.

\subsection{Existence of simultaneous ping-pong partners in linear groups}
The construction and study of free products in linear groups is a classical topic, going back way beyond Tits' celebrated work \cite{Tits72} establishing existence of free subgroups in linear groups which are not virtually solvable. 
Given a subset $F$ of a linear group $G$, the existence of \emph{simultaneous} ping-pong partners for elements of $F$ (that is, elements which are ping-pong partners for every $h \in F$) has also been studied, see namely the works of Poznansky \cite[Theorem 6.5]{Poznansky06} and Soifer and Vishkautsan \cite[Theorem 1.3]{SoiferVishkautsan10}. We also recall from the introduction the open question going back to Bekka, Cowling and de la Harpe, cases of which are answered in the two works just cited. 

\begin{question}[see {\cite[Remark 3]{BekkaCowlingdelaHarpe95}} and {\cite[Question 17]{delaHarpe07}}] \label{que:delaHarpe}
Let $G$ be a connected semisimple adjoint real Lie group without compact factors, and let $\Gamma$ be a Zariski-dense subgroup of $G$. 
Let $F$ be a finite set of non-trivial elements of $\Gamma$. 
Does there exists an element $\gamma \in \Gamma$ of infinite order such that $\langle h, \gamma \rangle \cong \langle h \rangle \ast \langle \gamma \rangle$ for every $h \in F$?
\end{question}

Of course, if $F$ is a subgroup, the condition that $\langle h, \gamma \rangle$ be freely generated for every element $h \in F$ does not imply that the subgroup $\langle F, \gamma \rangle$ is freely generated by $F$ and $\gamma$. 
For instance, even when the subgroup $\langle (h_1,h_2), (\gamma_1, \gamma_2) \rangle$ of $G_1 \times G_2$ is freely generated for every pair $(h_1, h_2) \in F_1 \times F_2$, the group $\langle F_1 \times F_2, (\gamma_1, \gamma_2)\rangle$ is never freely generated, as $(\gamma_1 h_1 \gamma_1^{-1}, 1)$ commutes with $(1, h_2)$. 

Similarly, even if the projections $\langle h_i, \gamma_i \rangle$ ($i=1,2$) of a subgroup $\langle (h_1,h_2), (\gamma_1, \gamma_2) \rangle$ of $G_1 \times G_2$ are freely generated subgroups, it may be that $\langle (h_1,h_2), (\gamma_1, \gamma_2) \rangle$ is itself not freely generated: if $h_1, h_2$ have distinct finite orders $n_1$ and $n_2$, then again $(\gamma_1, \gamma_2) (h_1, h_2)^{n_2} (\gamma_1, \gamma_2)^{-1}$ and $(h_1, h_2)^{n_1}$ commute. 
\Cref{prop:amalgaminproduct} shows in fact that for \Cref{que:delaHarpe} to possibly have a positive answer, one must at the very least require that each $h \in F$ generates a subgroup $\langle h \rangle$ that embeds into one of the factors of $G$ (cf.\ \Cref{rem:almostembeddingnecessary} in general). 

For these reasons and others, we cannot directly use the works mentioned above; but we will use similar techniques to prove the following. 

\begin{theorem}\label{thm:pingpongdense}
Let $\Fi$ be a field. 
Let $\bG$ be a connected algebraic $\Fi$-group with center $\bZ$. 
Let $\Gamma$ be a Zariski-connected subgroup of $\bG(\Fi)$. Let $(A_i, B_i)_{i \in I}$ be a finite collection of finite subgroups of $\bG(\Fi)$. 
Suppose that for each $i \in I$ there exists a local field $\LFi_i$ containing $\Fi$ and a projective $\LFi_i$-representation $\rho_i: \bG \to \PGL_{V_i}$, where $V_i$ is a finite-dimensional module over a finite-dimensional division $\LFi_i$-algebra $\LDi_i$, with the following two properties: 
\begin{enumerate}[leftmargin=3cm,itemsep=1ex,topsep=1ex]
	\item[\textup{(Proximality)}] $\rho_i(\Gamma)$ contains a proximal element;
	\item[\textup{(Transversality)}] For every $h \in (A_i \cup B_i) \setminus \bZ(F)$ and every $p \in \P(V_i)$, the span of the set $\{\rho_i(xhx^{-1}) p \mid x \in \Gamma \}$ is the whole of $\P(V_i)$.
\end{enumerate}

Let $S$ be the collection of regular semisimple elements $\gamma \in \Gamma$ of infinite order, such that for all $i \in I$, the canonical maps
\begin{align*}
\langle \gamma, C_{A_i} \rangle \ast_{C_{A_i}} A_i &\to \langle \gamma, A_i \rangle && \text{where $C_{A_i} = A_i \cap \bZ(F)$,}\\
\langle \gamma, C_{B_i} \rangle \ast_{C_{B_i}} B_i &\to \langle \gamma, B_i \rangle && \text{where $C_{B_i} = B_i \cap \bZ(F)$,}\\
\intertext{and provided $\indx{A_i}{C_{A_i}} > 2$ or $\indx{B_i}{C_{B_i}} > 2$, also}
\langle A_i, C_i \rangle \ast_{C_i} \langle B_i, C_i \rangle &\to \langle A_i, B_i^{\gamma} \rangle && \text{where $C_i = C_{A_i} \cdot C_{B_i}$,}
\end{align*}
are all isomorphisms. 
Then $S$ is dense in $\Gamma$ for the join of the profinite topology and the Zariski topology. 
\end{theorem}

\begin{remark}
The conditions defining $S$ in the statement of the theorem amount to the kernel of the canonical maps
\begin{align*}
\langle \gamma \rangle \ast A_i \to \langle \gamma, A_i \rangle, \quad
\langle \gamma \rangle \ast B_i \to \langle \gamma, B_i \rangle, \quad
A_i \ast B_i \to \langle A_i, B_i \rangle
\end{align*}
being respectively $\langle\!\langle [\gamma, C_{A_i}] \rangle\!\rangle$, $\langle\!\langle [\gamma, C_{B_i}] \rangle\!\rangle$, and $\langle\!\langle [C_{A_i}, {B_i}], [{A_i}, C_{B_i}] \rangle\!\rangle$. 

When $\bZ(\Fi)$ is trivial, the theorem states that for any $\gamma \in S$ and for all $i \in I$, the subgroups $\langle \gamma, A_i \rangle$, $\langle \gamma, B_i \rangle$, and $\langle A_i, B_i^{\gamma} \rangle$ of $\bG(\Fi)$ are freely generated. 
\end{remark}

\begin{remark} \label{rem:transversalitycondition}
Note that the transversality condition implies that every $\rho_i$ is irreducible. 
Moreover, the transversality condition holds equivalently for $\Gamma$ or for its Zariski closure (it is a \emph{Zariski-closed} condition). 
Thus, if $\Gamma$ happens to be Zariski-dense (as is most common), this condition can be replaced by the analogue for $\bG(\LFi_i)$:
\emph{\begin{enumerate}[leftmargin=3cm,itemsep=1ex,topsep=1ex]
	\item[\textup{(Transversality$'$)}] For every $h \in H_i \setminus C_i$ and every $p \in \P(V_i)$, the span of the set $\{\rho_i(xhx^{-1}) p \mid x \in \bG(\LFi_i) \}$ is the whole of $\P(V_i)$.
\end{enumerate}}
\end{remark}

\begin{remark}
\Cref{thm:pingpongdense} is only meaningful for pseudo-reductive groups. 
Indeed, the $\Fi$-unipotent radical $\rR_{\mathrm{u},\Fi}(\bG)$ must acts trivially under $\rho_i$, as the fixed-point set of $\rR_{\mathrm{u},\Fi}(\bG)$ is non-empty by the Lie--Kolchin theorem, hence is the whole of $V_i$. 
Thus each $\rho_i$ factors through the pseudo-reductive quotient $\bG / \rR_{\mathrm{u},\Fi}(\bG)$ of $\bG$. 
We remind the reader that if $\Char \Fi = 0$, the full unipotent radical $\rR_{\mathrm{u}}(\bG)$ of $\bG$ is defined over $\Fi$, hence pseudo-reductive groups are reductive (the converse always holding). 

In subsequent sections, we will mostly be concerned with number fields and their archimedean completions, leaving aside the usual complications arising in positive characteristic. 
\end{remark}

\begin{remark}
There is no obvious analogue of \Cref{thm:pingpongdense} for HNN extensions. Indeed, $\bG(\Fi)$ may admit finite subgroups $H$ containing a proper subgroup $H_1$ whose centralizer in $\bG(\Fi)$ is trivial. 
For instance, $\PGL_2(\C)$ contains a copy of the symmetric group on $4$ letters, whose alternating subgroup has trivial centralizer (see for instance \cite[Proposition 1.1]{Beauville10}). 
In such a situation, there is no HNN extension in $\bG(\Fi)$ of $H$ with respect to the identity $H_1 \to H_1$, as any $g \in \bG(\Fi)$ centralizing $H_1$ is trivial, but $\HNN{H}{H_1}$ is not. 
\end{remark}

\subsection{Proximal dynamics in projective spaces} \label{subsec:proximaldynamics}
Before proving \Cref{thm:pingpongdense}, we need to extend a few known facts about the dynamics of the action of $\GL(V)$ on $\P(V)$ to projective spaces over division algebras. 
Foremost, we will need the contents of \cite[\S3]{Tits72} over a division algebra, but the proofs given by Tits are valid with minor adaptations to keep track of the $\LDi$-structure and the fact $\LDi$ is not necessarily commutative. 
All of this is straightforward, so we will not rewrite arguments whenever they apply in the same way. 
\smallbreak

In this subsection, let $\LFi$ be a local field, $\LDi$ a division algebra of dimension $d$ over $\LFi$, and $V$ a finite-dimensional right $\LDi$-module. 
Recall that the absolute value $| \cdot |$ of $\LFi$ extends uniquely to an absolute value on $\LDi$ which will also denote by $| \cdot |$; we have the formula $|x| = |\rN(x)|^{1/d}$ for $x \in \LDi$. 
For each $K$-variety $\mathbf{V}$, the topology of $\LFi$ induces a locally compact topology on $\mathbf{V}(\LFi)$; this topology is often called the \emph{local topology}, to distinguish it from the Zariski topology when needed. 

With little deviation, we will follow the notations and conventions of \cite{Tits71} and \cite{Tits72}, which the reader may consult along with \cite{BorelTits65} for background material on the representation theory of algebraic groups (including over division algebras). 

Recall that $\GL_V$ is the algebraic $\LFi$-group of automorphisms of the $\LDi$-module $V$, so that for any $\Fi$-algebra $A$, the group $\GL_V(A)$ is the group of automorphisms of the right $(\LDi \ot_{\LFi} A)$-module $V \ot_{\LFi} A$. 
Provided $\dim V \geq 2$, the $\LFi$-group $\PGL_V$ is the quotient of $\GL_V$ by its center (which is the multiplicative group of the center of $\LDi$). 
The \emph{projective general linear group} $\PGL_V$ acts on the \emph{projective space} $\P(V)$ of $V$, which is the space of right $\LDi$-submodules of $V$ of dimension $1$. 
The $\LDi$-submodules of $V$ and their images in $\P(V)$ are both called \emph{($\LDi$-linear) subspaces}. 
A projective representation $\rho: \bG \to \PGL_V$ of a $\LFi$-group $\bG$ is called \emph{irreducible} if there are no proper non-trivial linear subspaces of $\P(V)$ stable under $\rho(\bG)$. 
A representation $\bG \to \GL_V$ is then irreducible if and only if its projectivization is. 

Given two subspaces $X, Y$ of $\P(V)$, we denote their span by $X \join Y$. If $X \cap Y = \emptyset$ and $X \join Y = \P(V)$, we denote by $\proj(X,Y)$ the mapping $\pi: X \to Y$ defined by $\{\pi(p)\} = (X \join \{p\}) \cap Y$. 
We will denote by $\mathring{C}$ the interior (for the local topology) of a subset $C$ of $\P(V)$. 

When it is needed to view $V$ as a $\LFi$-module instead of a $\LDi$-module, we will add the corresponding subscript to the notation. 

\begin{definition} \label{def:proximal}
Let $g$ be an element of $\GL_V(\LFi)$ or $\PGL_V(\LFi)$. 
\begin{enumerate}[itemsep=1ex,topsep=1ex,leftmargin=2em]
	\item Momentarily view $V$ as a vector $\LFi$-space, so as to identify $\GL_{V}$ with the subgroup of $\GL_{V,\LFi}$ centralizing the right action of $\LDi$ on $V$, and likewise for $\PGL_{V}$. 
The \emph{attracting subspace of $g$} is the subspace $\rA(g)$ of $V$ which is the direct sum of the generalized eigenspaces (over some algebraic closure) associated to the eigenvalues of maximal absolute value of (any lift to $\GL_V$ of) $g$. 
The complementary set $\rA'(g)$ is defined to be the direct sum of the remaining generalized eigenspaces of $g$. By construction, $V = \rA(g) \oplus \rA'(g)$. 
\end{enumerate}

Note that since the Galois group of any extension of $\LFi$ preserves the absolute value, it permutes the generalized eigenspaces of maximal absolute value, hence $\rA(g)$ and $\rA'(g)$ are stable under the Galois group and are indeed defined over $\LFi$. 
Moreover, if $g$ commutes with the action of $\LDi$, then $\LDi$ preserves the generalized eigenspaces of $g$ (after perhaps extending scalars). In this case, $\rA(g)$ and $\rA'(g)$ are themselves stable under $\LDi$, i.e.\ they are $\LDi$-subspaces of $V$. 

The subspaces $\rA(g)$ and $\rA'(g)$ only depend on the image of $g$ in $\PGL_V$. In what follows, we will often omit projectivization from the notation as long as it causes no confusion between $V$ and $\P(V)$. 

\begin{enumerate}[itemsep=1ex,topsep=1ex,leftmargin=2em,resume]
	\item We call $g$ \emph{proximal} if $\dim_{\LDi} \rA(g) = 1$, in other words if $\rA(g)$ is a point in $\P(V)$. In case $\LDi = \LFi$, this means that $g$ has a unique eigenvalue (counting with multiplicity) of maximal absolute value. In general, this means that $g$ has $d$ (possibly different) eigenvalues of maximal absolute value. If both $\rA(g)$ and $\rA(g^{-1})$ are one-dimensional, we call $g$ \emph{biproximal}\footnote{Biproximal elements are sometimes called `very proximal' or `hyperbolic' in the literature.}. 
We call a (projective) representation $\rho: \Gamma \to (\rP)\!\GL_V(\LFi)$ \emph{proximal} if $\rho(\Gamma)$ contains a proximal element. 
\end{enumerate}

Proximal elements have contractive dynamics on $\P(V)$: if $g$ is proximal, then for any $p \in \P(V) \setminus \rA'(g)$ the sequence $(g^n \cdot p)_{n \in \N}$ converges to the point $\rA(g)$ (see \Cref{lem:proximalcriterion}). 
\end{definition}
\smallbreak

The complement $\P(V) \setminus X$ of a hyperplane $X \subset \P(V)$ can be identified with an affine space over $\LDi$ by choosing for $V$ a system of coordinate functions $\xi= (\xi_0, \dots, \xi_{\dim \P(V)})$, $\xi_i \in V^*$, such that $X = \ker \xi_0$. The functions $\xi_i \xi_0^{-1}$ (${i=1, \dots, \dim \P(V)}$) then define affine coordinates on $\P(V) \setminus X$. 
If $g \in \PGL_V(\LFi)$ stabilizes $X$, its restriction to $\P(V) \setminus X$ need not be an affine map in these coordinates, but will be semiaffine (with respect to conjugation by the factor by which $g$ scales $\xi_0$). In particular, if $\P(V) \setminus X$ is seen as an affine space over $\LFi$, then the restriction of $g$ is $\LFi$-affine. 

For the rest of this section, we fix an \emph{admissible} distance $d$ on $\P(V)$, that is, a distance function $d: \P(V) \times \P(V) \to \P(V)$ inducing the local topology on $\P(V)$ and satisfying the property that for every compact subset $C$ contained in an affine subspace of $\P(V)$, there exist constants $M, M' \in \R$ such that
\[
M \cdot \restr{d_\xi}{{C \times C}} \leq \restr{d}{{C \times C}} \leq M' \cdot \restr{d_\xi}{{C \times C}}. 
\]
Here $d_\xi$ is the supremum distance with respect to the affine coordinates $(\xi_i \xi_0^{-1})_{i=1}^{\dim \P(V)}$ described above. 
Note that two different coordinate systems on the same affine subspace $A$ of $\P(V)$ define comparable distance functions on this affine subspace. Moreover, if instead of using $\LDi$-coordinates one views $A$ as an affine $\LFi$-space, the supremum distance in any set of affine $\LFi$-coordinates will again be comparable to $d_\xi$. 

As indicated by Tits, when $\LFi$ is an archimedean local field, any elliptic metric on $\P(V)$ is admissible. Tits also indicates in \cite[\S3.3]{Tits72} how to construct an admissible metric in the non-archimedean case by patching together different $d_\xi$'s; this construction works identically over a division algebra. 

Having fixed an (admissible) distance $d$ on $\P(V)$, the \emph{norm} of a mapping $f: X \to \P(V)$ defined on some subset $X \subset \P(V)$ is the quantity
\[
\|f\| = \sup_{\substack{p, q \in X\\ p \neq q}} \frac{d(f(p),f(q))}{d(p,q)}. 
\]
Note that the norm is submultiplicative: given mappings $f: X \to \P(V)$ and $g: Y \to X$, we have $\|f \circ g \| \leq \|f\| \cdot \|g\|$. 
Projective transformations always have finite norm \cite[Lemma 3.5]{Tits72}. Indeed, given $g \in \PGL_V(\LFi)$, the distance function $d^g$ defined by
\(
d^g(p,q) = d(g p, g q)
\)
is again admissible. Since $\P(V)$ is compact, it can be covered by finitely many compact sets contained in affine subspaces, on which the ratio between $d^g$ and $d$ is uniformly bounded, by admissibility. 
\medskip

We can now state the needed results from \cite[\S3]{Tits72} in our setting. 
The following two lemmas describe the dynamics of $\LDi$-linear transformations. 

\begin{lemma}[Lemma 3.8 in \cite{Tits72}] \label{lem:proximalcriterion}
Let $g \in \PGL_V(\LFi)$, let $C$ be a compact subset of $\P(V)$ and let $r \in \R_{>0}$. 

\begin{enumerate}[itemsep=1ex,topsep=1ex,label=\textup{(\roman*)},leftmargin=2em]
	\item \label{lem:proximalcriterion1}
Suppose that $g$ is proximal and that $C \cap \rA'(g) = \emptyset$. Then there exists an integer $N$ such that $\|\restr{g^n}{C}\|< r$ for all $n > N$; and for any neighborhood $U$ of $\rA(g)$, there exists an integer $N'$ such that $g^n C \subset U$ for all $n > N'$. 

	\item \label{lem:proximalcriterion2}
Assume that, for some $m \in \N$, one has $\|\restr{g^m}{C}\| < 1$ and $g^m C \subset \mathring{C}$. Then $\rA(g)$ is a point contained in $\mathring{C}$. 
\end{enumerate}
\end{lemma}

Note that in loc.~cit.~Tits assumes the existence of a semisimple proximal element; but as he indicates in the footnotes, this assumption is superfluous and the proof of the lemma is identical with an arbitrary proximal element. 

\begin{proof}
The argument given by Tits applies, taking into account the following adaptations.

In part (i), the transformation $g$ restricted to $\P(V) \setminus A'(g)$ is not necessarily $\LDi$-linear, as was already mentioned. It is nevertheless $\LFi$-linear, with eigenvalues of absolute value strictly smaller than 1 by assumption. So one can apply \cite[Lemma 3.7 (i)]{Tits72} over $\LFi$ and use that the norms defined over $\LDi$ or $\LFi$ are comparable to conclude. 

In part (ii), one cannot pick a representative of $g$ in $\GL_V$ whose eigenvalues corresponding to the fixed point $p \in \P(V)$ equal one (as $g$ may have different eigenvalues on the $\LDi$-line $p$). Nevertheless, they are all of the same absolute value, which we can assume to be $1$. If there is another eigenvalue of the same absolute value (i.e.\ if $\rA(g) \neq \{p\}$), then the restriction of $g$ to $\rA(g)$ is a block-upper-triangular matrix in a well-chosen basis. Since the compact set $C$ has non-empty interior, this contradicts the hypothesis of (ii). 
\end{proof}

\begin{lemma}[Lemma 3.9 in \cite{Tits72}] \label{lem:proximaldynamics}
Let $g \in \PGL_V(\LFi)$ be semisimple, let $\bar{g} \in \GL_V(\LFi)$ be a representative of $g$, let $\Omega$ be the set of eigenvalues of $\bar{g}$ (over an appropriate field extension of $\LFi$) whose absolute value is maximum, let $C$ be a compact subset of $\P(V) \setminus \rA'(g)$, set $\pi= \proj(\rA'(g), \rA(g))$, and let $U$ be a neighborhood of $\pi(C)$ in $\P(V)$. 

\begin{enumerate}[itemsep=1ex,topsep=1ex,label=\textup{(\roman*)},leftmargin=2em]
	\item There exists an infinite set $N \subset \N$ such that $\displaystyle \lim_{\substack{n \in N\\ n \to \infty}} (\lambda^{-1} \mu)^n = 1$ for all $\lambda, \mu \in \Omega$.
	\item The set $\{\|\restr{g^n}{C}\| \mid n \in \N\}$ is bounded. 
	\item If $N$ is as in (i), $g^n C \subset U$ for almost all $n \in N$. 
\end{enumerate}
\end{lemma}
\begin{proof}
The easiest way to obtain this lemma over the division algebra $\LDi$ is to take a representative of $g$ in $\GL_V$, see it as an $\LFi$-linear transformation in $\GL_{V,\LFi}$ and apply Tits' original lemma \cite[Lemma 3.9]{Tits72}. Part (i) is then immediate.

For part (ii) and (iii), denote $\P_{\LFi}(V)$ the projective space of $V$ seen as a vector $\LFi$-space. Since the canonical $\GL_V$-equivariant map $q: \P_{\LFi}(V) \to \P(V)$ is proper and continuous, $C' = q^{-1}(C)$ is compact, and $U' = q^{-1}(U)$ is open. Thus \cite[Lemma 3.9]{Tits72} applies with $C'$ and $U'$ over $\LFi$, and in turn yields the same conclusions over $\LDi$, since the norms of $g$ restricted to $C$ and $C'$ bound each-other. 
\end{proof}

We will also make use of a version of part (i) of \Cref{lem:proximaldynamics} for multiple representations, due to Margulis and Soifer. 
They initially stated it for multiple vector spaces over the same local field, but as already observed in \cite[Lemma 3.1]{Poznansky06}, the proof is identical. 

\begin{lemma}[Lemma 3 in \cite{MargulisSoifer81}] \label{lem:eigenvalueaccumulation}
Let $\{\LFi_i\}_{i \in I}$, be a finite collection of local fields and $V_i$ be a finite-dimensional vector $\LFi_i$-space. 
Let ${g}_i$ be a semisimple element of $\GL_{V_i}(\LFi)$, and let $\Omega({g}_i)$ be the set of eigenvalues of ${g}_i$ whose absolute value is maximum. 
There exists an infinite subset $N \subset \N$ such that $\displaystyle \lim_{\substack{n \in N\\ n \to \infty}} (\lambda^{-1} \mu)^n = 1$ for all $i \in I$ and $\lambda, \mu \in \Omega({g}_i)$. 
\end{lemma}

We are now ready to prove the following slight generalization of \cite[Corollary 3.7]{Poznansky06}, which is itself a refinement of both \cite[Proposition 3.11]{Tits72} and \cite[Lemma 8]{MargulisSoifer81}. This proposition is a crucial piece of the proof of \Cref{thm:pingpongdense}: it will be used to find enough biproximal elements in $\Gamma$. 

\begin{proposition}[Abundance of simultaneously biproximal elements] \label{prop:proximaldense}
Let $\bG$ be a connected algebraic $\Fi$-group and let $\Gamma$ be a Zariski-dense subgroup of $\bG(\Fi)$. 
Let $\{\LFi_i\}_{i \in I}$ be a finite collection of local fields each containing $\Fi$. 
For each $i \in I$, let $\rho_i: \bG \to \PGL_{V_i}$ be an irreducible projective $\LFi_i$-representation, where $V_i$ is a finite-dimensional module over a finite-dimensional division $\LFi_i$-algebra $\LDi_i$. 

Suppose that for each $i \in I$, $\rho_i(\Gamma)$ contains a proximal element. Then the set of regular semisimple elements $\gamma \in \Gamma$ such that $\rho_i(\gamma)$ is biproximal for every $i \in I$, is dense in $\Gamma$ for the join of the Zariski topology and the profinite topology. 
\end{proposition}

\begin{proof}
We follow the line of arguments given in \cite{Tits72,MargulisSoifer81,Poznansky06}, keeping track of the different representations, and using the extension of Tits' work to projective representations over a division algebra laid out above. 

Given an arbitrary element $g \in \bG(\Fi)$, let us abbreviate $\rho_i(g)$ by $g_i$.
\medbreak

\noindent{\underline{{Step 1:}} The set of simultaneously proximal elements in $\Gamma$ is Zariski-dense if it is non-empty. }\smallskip

Let $g \in \Gamma$ be such that $g_i$ is proximal for all $i \in I$. 
Since $\rho_i$ is irreducible, for each $i \in I$ the set of elements $x$ of $\bG(\Fi)$ such that $x_i\rA(g_i) \not\in \rA'(g_i)$ is non-empty and Zariski-open. Because $\bG$ is Zariski-connected, the intersection of these sets remains non-empty (and Zariski-open). Let us then pick $x \in \Gamma$ satisfying $x_i \rA(g_i) \not\in \rA'(g_i)$ for every $i \in I$. 

By construction of $x$, we can pick a compact neighborhood $C_i$ of $\rA(g_i)$ in $\P(V_i)$ such that $x_i C_i$ is disjoint from $A'(g)$. 
Since projective transformations have finite norm, we have $\max_{i \in I}\|\restr{x_i}{{C_i}}\| < r$ for some $r \in \R$. 
By \Cref{lem:proximalcriterion}.\ref{lem:proximalcriterion1}, for each $i \in I$ there exists an integer $N_i$ such that
\begin{align*}
\|\restr{g_i^n}{{x_i C_i}}\| &< r^{-1}
&\text{and}
&&g_i^n (x_i C_i) &\subset \mathring{C}_i
&&\text{for $n > N_i$.}
\intertext{\indent Set $N_x = \max_{i \in I} N_i$. Then for any $i \in I$, we have that}
\|\restr{g_i^n x_i}{C_i}\| &< 1
&\text{and}
&&(g_i^n x_i) C_i &\subset \mathring{C}_i
&&\text{for $n > N_x$.}
\end{align*}
We deduce from \Cref{lem:proximalcriterion}.\ref{lem:proximalcriterion2} that $g_i^n x_i = \rho_i(gx)$ is proximal for every $n > N_x$. 

Observe that the Zariski closure $Z$ of $\{g^n \mid n > N_x\}$ in $\Gamma$ has the property that $gZ \subset Z$. Since the Zariski topology is Noetherian, we deduce that $g^{m+1} Z = g^m Z$ for some $m \in \N$. This implies that $g^n Z = Z$ for every $n \in \Z$, and in particular that $g \in Z$.
Let now $\overline{S}$ denote the Zariski closure in $\Gamma$ of the set $S$ of elements of $\Gamma$ which are proximal under every $\rho_i$. We have shown that $S$ contains $g^n x$ for each $x \in \Gamma$ chosen as above and $n > N_x$. 
By our last observation, $\overline{S}x^{-1}$ contains $g$, hence $gx \in \overline{S}$. As this holds for every $x$ in a Zariski-dense (open) subset of $\Gamma$, we conclude that $\overline{S}$ contains $g\Gamma = \Gamma$, as claimed. 

\medbreak

\noindent{\underline{{Step 2:}} $\Gamma$ contains a semisimple element that is simultaneously proximal. }\smallskip

We argue by induction on $\# I$. 
Fix $j \in I$, and suppose that there are elements $g, h \in \Gamma$ such that $\rho_j(h)$ is proximal and $\rho_i(g)$ is proximal for $i \in I \setminus \{j\}$. By Step 1, we may in addition assume that $g$ and $h$ are semisimple. 
Write $\pi_i = \proj(\rA'(h_i),\rA(h_i))$ for $i \neq j$, and $\pi_j = \proj(\rA'(g_j),\rA(g_j))$. 

Let $N \subset \N$ be an infinite set such as afforded by \Cref{lem:eigenvalueaccumulation} applied to the elements $h_i$ for $i \neq j$ and $g_j$ for $i = j$, so that we have $\displaystyle \lim_{\substack{n \in N\\ n \to \infty}} (\lambda^{-1} \mu)^n = 1$ for $\lambda, \mu \in \Omega(h_i)$ if $i \neq j$, and for $\lambda, \mu \in \Omega(g_j)$. 

Since $\rho_i$ is irreducible and $\Gamma$ is Zariski-dense, as before we can fix $x \in \Gamma$ such that 
\begin{align*}
x_i \rA(g_i) &\not\subset \rA'(h_i) \quad \text{for every $i \in I$. }
\intertext{Similarly, the elements $y \in \bG(\Fi)$ satisfying }
y_i \cdot \pi_i(x_i \rA(g_i)) &\not\in \rA'(g_i) \quad \text{for $i \in I \setminus \{j\}$,}\\
\text{and} \quad y_j \rA(h_j) &\not\in \left(x_j^{-1} \rA'(h_j) \cap \rA(g_j)\right) \join \rA'(g_j),
\end{align*}
form a non-empty Zariski-open subset of $\bG(\Fi)$. Let us then fix $y$ such an element in $\Gamma$. 

For $i \neq j$, let $B_i$ be a compact neighborhood of $y_i \cdot \pi_i(x_i \rA(g_i))$ disjoint from $\rA'(g_i)$, and let $B_j$ be a compact neighborhood of $x_j \cdot \pi_j(y_j\rA(h_j))$ disjoint from $\rA'(h_j)$. 
The latter exists because $\pi_j^{-1}(x_j^{-1} \rA'(h_j)) \subset (x_j^{-1} \rA'(h_j) \cap \rA(g_j)) \vee \rA'(g_j)$ does not contain $y_j\rA(h_j)$. 
We also choose for $i \neq j$ a compact neighborhood $C_i$ of $\rA(g_i)$ disjoint from $x_i^{-1}\rA'(h_i)$ and small enough to satisfy $y_i \cdot \pi_i(x_i C_i) \subset \mathring{B}_i$; and choose a compact neighborhood $C_j$ of $\rA(h_j)$ disjoint from $y_j^{-1} \rA'(g_j)$ and satisfying $x_j \cdot \pi_j(y_j C_j) \subset \mathring{B}_j$. 

The careful choice of $B_i$, $C_i$ and $N$ sets us up for the following applications of \Cref{lem:proximaldynamics,lem:proximalcriterion}. 
By \Cref{lem:proximaldynamics}, for each $i \neq j$ there exists $r_i \in \R$ and $N_i \in \N$ such that
\begin{align*}
\|\restr{h_i^n}{x_i C_i}\| < r_i \text{ for $n \in \N$ }
&&\text{and}
&&y_i h_i^n x_i C_i &\subset \mathring{B}_i
&&\text{for $n \in N$, $n > N_i$.}
\intertext{Similarly, there exists $N_j \in \N$ and $r_j \in \R$ such that }
\|\restr{g_j^n}{y_j C_j}\| < r_j \text{ for $n \in \N$ }
&&\text{and}
&&x_j g_j^n y_j C_j &\subset \mathring{B}_j
&&\text{for $n \in N$, $n > N_j$.}
\intertext{By \Cref{lem:proximalcriterion}\ref{lem:proximalcriterion1}, for each $i \neq j$ there exists $N'_i \in \N$ such that }
\|\restr{g_i^n}{B_i}\| < \big( \|\restr{y_i}{y_i^{-1} B_i}\| \cdot r_i \cdot \|\restr{x_i}{C_i}\| \big)^{-1} 
&&\text{and}
&&g_i^n B_i &\subset \mathring{C}_i
&&\text{for $n > N'_i$.}
\intertext{Similarly, there exists $N'_j \in \N$ such that }
\|\restr{h_j^n}{B_j}\| < \big( \|\restr{x_j}{x_j^{-1}B_j}\| \cdot r_j \cdot \|\restr{y_j}{C_j}\| \big)^{-1}
&&\text{and}
&&h_j^n B_j &\subset \mathring{C}_j
&&\text{for $n > N'_j$.} \\
\intertext{\indent Set $N' = \{n \in N \mid n > N_i \text{ and } n > N'_i \text{ for all $i \in I$}\}$. 
For $i \neq j$, we have by construction that }
\| g_i^m y_i h_i^n \restr{x_i}{C_i} \| < 1 
&&\text{and}
&&g_i^m y_i h_i^n x_i C_i &\subset \mathring{C}_i
&&\text{for $m, n \in N'$.}
\intertext{Similarly, we have that }
\| h_j^n x_j g_j^m \restr{y_j}{C_j} \| < 1
&&\text{and} 
&&h_j^n x_j g_j^m y_j C_j &\subset \mathring{C}_j
&&\text{for $m, n \in N'$.}
\end{align*}
We conclude from \Cref{lem:proximalcriterion}.\ref{lem:proximalcriterion2} that for all $m, n \in N'$, the element $g_i^m y_i h_i^n x_i$ is proximal for $i \neq j$, and so is $h_j^n x_j g_j^m y_j$. But $h_j^n x_j g_j^m y_j$ and $g_j^m y_j h_j^n x_j$ are conjugate, so $g^m y h^n x \in \Gamma$ is proximal under $\rho_i$ for every $i \in I$. 

In view of Step 1, the set of simultaneously proximal elements in $\Gamma$ is Zariski-dense, so there is also a semisimple one as claimed. 
\medbreak

\noindent{\underline{{Step 3:}} $\Gamma$ contains an element which is simultaneously biproximal. }\smallskip

By Steps 1--2, there is a semisimple element $g \in \Gamma$ such that $\rho_i(g^{-1})$ is proximal for every $i \in I$. 
Let $N$ be an infinite set such as afforded by \Cref{lem:eigenvalueaccumulation}. 
Replacing $N$ by an appropriate subset, we may assume that the set $g^N = \{g^n \mid n \in N\}$ is Zariski-connected. 

Since $\rho_i$ is irreducible and $\Gamma$ is Zariski-dense, the elements $x \in \bG(\Fi)$ such that
\[
x_i \rA(g_i) \not\subset \rA'(g_i^{-1}) \quad \text{and} \quad x_i^{-1} \rA(g_i) \not\subset \rA'(g_i^{-1}) \quad \text{for every $i \in I$}
\]
form a non-empty Zariski-open subset. Fix such an element $x \in \Gamma$. 
For the same reasons, the set $U$ of elements $y \in \bG(\Fi)$ satisfying
\begin{align*}
y_i \rA(g_i^{-1}) &\not\in x_i \rA'(g_i) \join (x_i \rA(g_i) \cap \rA'(g_i^{-1})), &&\\
\text{and} \quad y_i^{-1} x_i \rA(g_i^{-1}) &\not\in \rA'(g_i) \join (\rA(g_i) \cap x_i \rA'(g_i^{-1})) &&\text{for every $i \in I$}
\end{align*}
is also non-empty and Zariski-open; fix $y \in U \cap \Gamma$. 

Write $\pi_i = \proj(\rA'(g_i),\rA(g_i))$ and $\pi_i' = \proj(x_i \rA'(g_i),x_i \rA(g_i))$. 
For each $i \in I$, let $B_i$ be a compact neighborhood of $\pi_i'(y_i \rA(g_i^{-1}))$ disjoint from $\rA'(g_i^{-1})$, and let $B_i'$ be a compact neighborhood of $\pi_i(y_i^{-1} x_i \rA(g_i^{-1}))$ disjoint from $x_i \rA'(g_i^{-1})$. 
We also choose a compact neighborhood $C_i$ of $\rA(g_i^{-1})$ disjoint from $y_i^{-1} x_i \rA'(g_i)$ satisfying $\pi_i'(y_i C_i) \subset \mathring{B}_i$, and a compact neighborhood $C_i'$ of $y_i^{-1} x_i \rA(g_i^{-1})$ disjoint from $\rA'(g_i)$ satisfying $\pi_i(C_i') \subset \mathring{B}_i'$.

By \Cref{lem:proximaldynamics} (ii), for each $i \in I$ there exist $N_i, N_i' \in \N$ and $r_i, r_i' \in \R$ such that
\begin{align*}
\|\restr{x_i g_i^n x_i^{-1}}{y_i C_i}\| < r_i \text{ for $n \in \N$ }
&&\text{and}
&&x_i g_i^n x_i^{-1} y_i C_i &\subset \mathring{B}_i
&&\text{for $n \in N$, $n > N_i$,} \\
\|\restr{g_i^n}{C_i'}\| < r_i' \text{ for $n \in \N$ }
&&\text{and}
&&g_i^n C_i' &\subset \mathring{B}_i'
&&\text{for $n \in N$, $n > N_i'$.}
\intertext{By \Cref{lem:proximalcriterion}.\ref{lem:proximalcriterion1}, for each $i \in I$ there exist $M_i, M_i' \in \N$ such that }
\|\restr{g_i^{-n}}{B_i}\| < \big(r_i \cdot \|\restr{y_i}{C_i}\| \big)^{-1}
&&\text{and}
&&g_i^{-n} B_i &\subset \mathring{C}_i
&&\text{for $n > M_i$,} \\
\|\restr{x_i g_i^{-n} x_i^{-1}}{B_i'}\| < \big( \|\restr{y_i^{-1}}{y_i C_i'}\| \cdot r_i' \big)^{-1}
&&\text{and}
&&x_i g_i^{-n} x_i^{-1} B_i' &\subset y_i \mathring{C}_i'
&&\text{for $n > M_i'$.}
\intertext{\indent Set $N_{x,y} = \{n \in N \mid n > \max \bigcup_{i \in I} \{N_i, N_i', M_i, M_i'\} \}$. 
We then have by construction that }
\| g_i^{-n} x_i g_i^n x_i^{-1} \restr{y_i}{C_i} \| < 1 
&&\text{and}
&&g_i^{-n} x_i g_i^n x_i^{-1} y_i C_i &\subset \mathring{C}_i
&&\text{for $n \in N_{x,y}$,} \\
\| y_i^{-1} x_i g_i^{-n} x_i^{-1} \restr{g_i^n}{C_i'} \| < 1
&&\text{and} 
&&y_i^{-1} x_i g_i^{-n} x_i^{-1} g_i^n C_i' &\subset \mathring{C}_i'
&&\text{for $n \in N_{x,y}$.}
\end{align*}
We conclude from \Cref{lem:proximalcriterion}.\ref{lem:proximalcriterion2} that for all $n \in N_{x,y}$ and for each $i \in I$, the element $g^{-n} x g^n x^{-1} y$ is biproximal under $\rho_i$. 

\medbreak

\noindent{\underline{{Step 4:}} The set of regular semisimple simultaneously biproximal elements is dense. }\smallskip

Let $S$ denote the set of elements in $\Gamma$ which are biproximal under every $\rho_i$. 
Let $\Lambda$ be a normal subgroup of finite index in $\Gamma$, and let $\gamma \in \Gamma$. 
Because the set of regular semisimple elements is Zariski-open, it suffices to show that $S \cap \Lambda \gamma$ is Zariski-dense to prove the proposition. 

Since $\Gamma$ is Zariksi-connected and $\Lambda$ has finite index in $\Gamma$, every coset of $\Lambda$ is Zariski-dense. 
Moreover, if $h \in \Gamma$ is such that $h_i$ is proximal, then $h^{\indx{\Gamma}{\Lambda}}$ is also proximal under $\rho_i$, and belongs to $\Lambda$. 
We can thus apply Steps 1--3 to $\Lambda$, to find an element $g \in \Lambda$ such that $g_i$ is biproximal for every $i \in I$. 

As before, the set $U$ of elements $x \in \bG(\Fi)$ such that
\[
x_i \gamma_i \rA(g_i) \not\in \rA'(g_i) \quad \text{and} \quad \gamma_i^{-1} x_i^{-1} \rA(g_i^{-1}) \not\in \rA'(g_i^{-1}) \quad \text{for every $i \in I$}
\]
is Zariski-open and non-empty. In particular, $\Lambda \cap U$ is Zariski-dense in $\Gamma$; pick $x \in \Lambda \cap U$. 

Let $C_i^\pm$ be a compact neighborhood of $\rA(g_i^{\pm 1})$ such that $(x \gamma)_i^{\pm 1} C_i^{\pm}$ is disjoint from $\rA'(g_i^{\pm 1})$. 
Since projective transformations have finite norm, we have that 
\(
\max_{i \in I} \| \restr{(x \gamma)_i^{\pm 1}}{C_i^{\pm}} \| < r
\) 
for some $r \in \R$. 
By \Cref{lem:proximalcriterion}.\ref{lem:proximalcriterion1}, there exist integers $N_i^+$ and $N_i^-$ such that
\begin{align*}
\|\restr{g_i^{n}}{x_i \gamma_i C_i^{+}}\| &< r^{-1}
&\text{and}
&&g_i^{n} x_i \gamma_i C_i^{+} &\subset \mathring{C}_i^{+}
&&\text{for $n > N_i^{+}$.}\\
\|\restr{g_i^{-n}}{(x \gamma)_i^{-1} C_i^{-}}\| &< r^{-1}
&\text{and}
&&g_i^{-n} (x \gamma)_i^{-1} C_i^{-} &\subset \mathring{C}_i^{-}
&&\text{for $n > N_i^{-}$.}
\intertext{\indent For $N_x = \max \bigcup_{i \in I} \{N_i^+, N_i^-\}$, we then have for every $i \in I$ that }
\|\restr{g_i^{n} x_i \gamma_i}{C_i^{+}}\| &< 1 
&\text{and}
&&g_i^{n} x_i \gamma_i C_i^{+} &\subset \mathring{C}_i^{+}
&&\text{for $n > N_x$.}\\
\|\restr{g_i^{-n} \gamma_i^{-1} x_i^{-1}}{C_i^{-}}\| &< 1 
&\text{and}
&&g_i^{-n} \gamma_i^{-1} x_i^{-1} C_i^{-} &\subset \mathring{C}_i^{-}
&&\text{for $n > N_x$.}
\end{align*}
We deduce from \Cref{lem:proximalcriterion}.\ref{lem:proximalcriterion2} that $g_i^{n} x_i \gamma_i$ and $g_i^{-n} \gamma_i^{-1} x_i^{-1}$ are proximal for every $i \in I$ and for $n > N_x$. But $g_i^{-n} \gamma_i^{-1} x_i^{-1}$ and $\gamma_i^{-1} x_i^{-1} g_i^{-n}$ are conjugate, so $g_i^{n} x_i \gamma_i$ is in fact biproximal for every $i \in I$. 
Of course $g^{n} x \gamma \in \Lambda \gamma$, so we have shown that $S \cap \Lambda \gamma$ contains $g^n x \gamma$ for every $x \in \Lambda \cap U$ and $n > N_x$. 

As was observed in Step 1, the Zariski closure of $\{g^n \mid n > N_x\}$ in $\Gamma$ contains $g$. 
Thus the Zariski closure of $S \cap \Lambda \gamma$ contains $g x \gamma$ for every $x \in \Lambda \cap U$. As $\Lambda \cap U$ is Zariski-dense, so is $S \cap \Lambda \gamma$. This concludes the proof of the proposition. 
\end{proof}

\subsection{Towards the proof of \texorpdfstring{\Cref{thm:pingpongdense}}{Theorem \ref{thm:pingpongdense}}} \label{subsec:pingpongdense}

Before starting the proof of \Cref{thm:pingpongdense}, we record the following lemmas. 

\begin{lemma} \label{lem:attractingorbitclosed}
Let $\LFi$, $\LDi$ and $V$ be as in \S\ref{subsec:proximaldynamics}. 
Let $\bG$ be a connected $\LFi$-subgroup of $\PGL_V$, acting irreducibly on $\P(V)$. Suppose that $\bG(\LFi)$ contains a proximal element $g_0$. Then the set
\[
X = \{\rA(g) \mid \text{$g \in \bG(\LFi)$ is proximal}\} \subseteq \P(V)
\]
coincides with the orbit $\bG(\LFi) \cdot \rA(g_0)$ and constitutes the unique irreducible projective subvariety of $\P(V)$ stable under $\bG(\LFi)$. In consequence, $\Stab_{\bG}(\rA(g_0))$ is a parabolic subgroup of $\bG$. 
\end{lemma}
\begin{proof}
By a theorem of Chevalley, there is a Zariski-closed $\bG(\LFi)$-orbit $Y \subseteq \P(V)$. 
Let $g \in \bG(\LFi)$ be proximal. Because $\bG$ acts irreducibly on $\P(V)$, there exists $y \in Y \setminus \rA'(g)$. 
We then have $g^n \cdot y \xrightarrow{n \to \infty} \rA(g)$, thus $\rA(g)$ lies in the closure of $Y$ in the local hence in the Zariski topology. As $Y$ was Zariski-closed, $\rA(g) \in Y$. Since this happens for any proximal element $g$, we deduce that $X \subseteq Y$. As $X$ is $\bG(\LFi)$-stable and $Y$ is a single orbit, equality holds. 
It is now clear that $X$ is the set of $\LFi$-points of a projective variety $\bX$, which is irreducible because $\bG$ is. 

Let $\bP = \Stab_{\bG}(\rA(g))$ denote the stabilizer of $\rA(g)$ in $\bG$. The above shows that orbit map yields an isomorphism $\bG / \bP \to \bX$, hence $\bG / \bP$ is a complete variety, meaning that $\bP$ is parabolic. 
The same holds for every other proximal element. 
\end{proof}

\begin{remark}
\Cref{lem:attractingorbitclosed} can also be proven by arguing that if $g_0$ is proximal, $\rA(g_0)$ must be a highest-weight line. 
\end{remark}

\begin{lemma}[Transversality] \label{lem:transversality}
Let $\bG$ be as in \Cref{lem:attractingorbitclosed}, and suppose that $\bG(\LFi)$ contains a proximal element $g$. For any $h \in \bG(\LFi)$, the set
\[
U_{h,g} = \{x \in \bG(\LFi) \mid xhx^{-1} \rA(g) \not\in \rA'(g) \cup \rA'(g^{-1})\}
\]
is Zariski-open in $\bG(\LFi)$. If $h \in \bG(\LFi)$ is such that the span of $\{xhx^{-1} \rA(g) \mid x \in \bG(\LFi) \}$ is the whole of $\P(V)$, then $U_{h,g}$ is non-empty. 
\end{lemma}
\begin{proof}
The two sets
\begin{align*}
U_1 &= \{x \in \bG(\LFi) \mid xhx^{-1} \rA(g) \not\in \rA'(g)\} \\
U_2 &= \{x \in \bG(\LFi) \mid xhx^{-1} \rA(g) \not\in \rA'(g^{-1})\}
\end{align*}
are Zariski-open by a standard argument: for any subspaces $W_1, W_2 \subseteq V$, the set $\{x \in \bG(\LFi) \mid x \cdot W_1 \subseteq W_2\}$ is Zariski-closed. We have to show they are both non-empty. 

There is a minimal parabolic $\LFi$-subgroup $\bB$ of $\bG$ that contains $h$. By \Cref{lem:attractingorbitclosed}, there is a conjugate $x\bB x^{-1}$ of $\bB$ which fixes $\rA(g)$. But then for this choice of $x$, we surely have $xhx^{-1} \rA(g) \not\in \rA'(g)$. This shows that $U_1$ is not empty. 

Finally, $U_2$ is non-empty because of the assumption made on $h$. Indeed, $U_2$ being empty means $xhx^{-1} \rA(g) \in \rA'(g^{-1})$ for every $x \in \bG(\LFi)$, but the latter is a proper subspace of $\P(V)$. 
\end{proof}

\begin{remark}\label{rem:poznanskypaper}
At first glance, \Cref{lem:transversality} above may seem to be weaker than \cite[Proposition 2.17]{Poznansky06}. Unfortunately, the proof of \cite[Proposition 2.17]{Poznansky06} relies on \cite[Proposition 2.11]{Poznansky06}, whose statement is erroneous. The set of elements whose conjugacy class intersects a big Bruhat cell is in fact smaller than stated there (see for instance \cite{EllersGordeev04,EllersGordeev07,ChanLuTo10} for a description in the case of $\SL_n$). In consequence, the results of \cite{Poznansky06} are only valid under the additional assumption that the conjugacy classes of the elements $h$ under consideration intersect a big Bruhat cell. 
Note that there are non-central torsion elements whose conjugacy class does not intersect the big Bruhat cell. 
We will address this in the next section by arranging for the transversality assumption of \Cref{lem:transversality} and \Cref{thm:pingpongdense} to hold. 

We note in addition that the proof of \cite[Theorem 6.5]{Poznansky06} overlooks the possibility that the subgroup generated by a given torsion element $h$ may not embed in any simple quotient of $\bG$. 
As will be emphasized in \Cref{rem:almostembeddingnecessary}, this condition is necessary for constructing a ping-pong partner for $h$. 
\end{remark}

\bigbreak

\begin{proof}[Proof of \Cref{thm:pingpongdense}]
For an arbitrary element $g \in \bG(\Fi)$, let us abbreviate $\rho_i(g)$ by $g_i$. 
For simplicity, we also write $\Ant_i = A_i \setminus \bZ(\Fi)$ and $\Bnt_i = B_i \setminus \bZ(\Fi)$. 
Recall that $C_{A_i} = A_i \cap \bZ(\Fi)$, $C_{B_i} = B_i \cap \bZ(\Fi)$, and $C_{i} = C_{A_i} \cdot C_{B_i}$. 
\medbreak

Fix a normal subgroup $\Lambda$ of finite index in $\Gamma$, and fix $\gamma_0 \in \Gamma$. 
First, because of the proximality hypothesis, \Cref{prop:proximaldense} applied to the Zariski-closure $\bH$ of $\Gamma$ in $\bG$ states that the set $S'$ of regular semisimple elements $\gamma' \in \Lambda \gamma_0$ such that $\gamma'_i = \rho_i(\gamma')$ is biproximal for every $i \in I$, is Zariski-dense in $\Gamma$. Pick $\gamma' \in S'$. 
\smallbreak

Second, using the transversality hypothesis on $\rho_i$, we exhibit a simultaneously biproximal element in $\Lambda \gamma_0$ acting transversely to $\Ant_i$ and $\Bnt_i$ for all $i \in I$. 
By \Cref{lem:transversality}, for every $i \in I$ and every $h \in \Ant_i \cup \Bnt_i$ the sets
\[
U_{i,h,{\gamma'}^{\pm 1}} = \{x \in \bH(\Fi) \mid x_i h_i x_i^{-1} \rA({\gamma'_i}^{\pm 1}) \not\in \rA'(\gamma'_i) \cup \rA'({\gamma'_i}^{-1})\}
\]
are Zariski-open and non-empty. 
In consequence, we can pick an element $\lambda$ in the Zariski-dense set $\Lambda \cap U_{\gamma'}$, where $U_{\gamma'} = \bigcap_{i \in I} \bigcap_{h \in \Ant_i \cup \Bnt_i} (U_{i,h,\gamma'} \cap U_{i,h,{\gamma'}^{-1}})$. 
Setting $\gamma = \lambda^{-1} \gamma' \lambda$, we see that $\gamma \in S'$, while for any $h \in \Ant_i \cup \Bnt_i$, 
\[
h_i \rA(\gamma_i) \not\in \rA'(\gamma_i) \cup \rA'({\gamma_i}^{-1}) \quad \text{and} \quad h_i \rA({\gamma_i}^{-1}) \not\in \rA'(\gamma_i) \cup \rA'({\gamma_i}^{-1}).
\]
\medbreak

Next, we construct the sets that will allow us to apply \Cref{lem:pingpongamalgam}. 
Given $i \in I$, let $P_i^\pm$ be a compact neighborhood of $\rA({\gamma_i}^{\pm 1})$ in $\P(V_i)$ small enough to achieve 
\[
\big((\Ant_i \cup \Bnt_i) \cdot P_i^\pm \big) \cap \big(\rA'(\gamma_i) \cup \rA'({\gamma_i}^{-1})\big) = \emptyset. 
\]
Such a set exists by construction of $\gamma$: by local compactness, the complement of the closed set $(\Ant_i \cup \Bnt_i) \cdot \big(\rA'(\gamma_i) \cup \rA'({\gamma_i}^{-1})\big)$ contains a compact neighborhood of $\rA({\gamma_i}^{\pm 1})$. 
In the same way, we can arrange that also 
\[
\big((\Ant_i \cup \Bnt_i) \cdot P_i^\pm \big) \cap (P_i^+ \cup P_i^-) = \emptyset.
\]
Note that $\bZ(\Fi)$ fixes $\rA(\gamma_i)$ and $\rA({\gamma_i}^{-1})$. 
The finite intersection $\bigcap_{c \in C_i} (c \cdot P_i^\pm)$ is thus again a compact neighborhood of $\rA({\gamma_i}^{\pm 1})$. Replacing $P_i^\pm$ by this intersection, we will further assume that $P_i^\pm$ is stable under $C_i$, hence under $C_{A_i}$ and $C_{B_i}$. 

Set $P_i = P_i^+ \cup P_i^-$ and set 
\[
Q_i = (\Ant_i \cup \Bnt_i) \cdot P_i;
\]
these two subsets of $\P(V_i)$ are compact, disjoint, and preserved by $C_i$. 
As $Q_i$ is disjoint from $\rA'(\gamma_i) \cup \rA'({\gamma_i}^{-1})$, \Cref{lem:proximalcriterion}.\ref{lem:proximalcriterion1} shows that there exists $N \in \N$ such that for any $n > N$, 
\[
\gamma_i^n Q_i \subset P_i^+ \quad \text{and} \quad \gamma_i^{-n} Q_i \subset P_i^- \quad \text{for each $i \in I$}.
\]
Pick $N_1 > N$ with $N_1 = 1 \mod \indx{\Gamma}{\Lambda}$, so that ${\gamma}^{N_1 + n\indx{\Gamma}{\Lambda}} \in \Lambda \gamma_0$ for every $n \in \Z$. 
For $n \in \N$, put $[n] = {N_1 + n\indx{\Gamma}{\Lambda}}$ and note that $[n] > N$ and $[n] = 1 \mod \indx{\Gamma}{\Lambda}$. 
\smallbreak

For each $i \in I$, \Cref{lem:pingpongamalgam} now applies to the following triples of subgroups of $\bG(\Fi)$:
\begin{equation*} \tag{$\ast$} \label{eq:listamalgams}
\begin{gathered}
\langle \gamma^{[n]}, C_{A_i} \rangle \text{ and } A_i \text{ along } C_{A_i}, \quad
\langle \gamma^{[n]}, C_{B_i} \rangle \text{ and } B_i \text{ along } C_{B_i}, \\ 
\langle A_i, C_i \rangle \text{ and } \langle \gamma^{[n]} B_i \gamma^{-[n]}, C_i \rangle \text{ along } C_{i}, 
\end{gathered}
\end{equation*}
with the same sets $P_i$ and $Q_i$ constructed above! 
Indeed, by construction, for all $m \in \Z \setminus \{0\}$ we have
\begin{equation*}
\begin{gathered}
\gamma_i^{m[n]} Q_i \subset P_i, \quad
\Ant_i \cdot P_i \subset Q_i, \quad 
\Bnt_i \cdot P_i \subset Q_i, \\
\gamma_i^{[n]} \Bnt_i \gamma_i^{-[n]} \cdot Q_i \subset \gamma_i^{[n]} \Bnt_i \cdot P_i^- \subset \gamma_i^{[n]} \cdot Q_i \subset P_i^+ \subset P_i.
\end{gathered}
\end{equation*}
But $\{\gamma^{m[n]} \mid m \neq 0 \}$ clearly represents every non-trivial coset of $C_{A_i}$ (resp.\ of $C_{B_i}$) in the group they generate, and the same holds for $\Ant_i$ and $\gamma^{[n]} \Bnt_i \gamma^{-[n]}$ with respect to $C_i$. 
We conclude that for every $i \in I$ and all $n \in \N$, the subgroups $\langle \gamma^{[n]}, A_i \rangle$, $\langle \gamma^{[n]}, B_i \rangle$, and $\langle A_i, \gamma^{[n]} B_i \gamma^{-[n]} \rangle$ are the free amalgamated products of the triples given in (\ref{eq:listamalgams}) above. 
\medbreak

This establishes that $S \cap \Lambda \gamma_0$ contains $\gamma^{[n]}$ for every $n \in \N$; it remains to show that $S \cap \Lambda \gamma_0$ is Zariski-dense. 

The Zariski closure $Z$ of $\{\gamma^{[n]} \mid n \in \N \}$ satisfies ${\gamma}^{\indx{\Gamma}{\Lambda}} Z \subset Z$. Since the Zariski topology is Noetherian, it follows that ${\gamma}^{(m+1) \indx{\Gamma}{\Lambda}} Z = {\gamma}^{m \indx{\Gamma}{\Lambda}} Z$ for some $m \in \N$, and in turn that ${\gamma} \in Z$. 

We have seen that $S'$ is Zariski-dense, and that for each $\gamma' \in S'$, the set $\Lambda \cap U_{\gamma'}$ is Zariski-dense. 
In consequence, the set $S'' = \{(\gamma', \lambda) \in \Gamma \times \Gamma \mid \gamma' \in S', \lambda \in \Lambda \cap U_{\gamma'}\}$ is Zariski-dense in $\Gamma \times \Gamma$. 
Indeed, its closure contains $\overline{\{\gamma'\} \times \Lambda \cap U_{\gamma'}} = \{\gamma'\} \times \Gamma$ for each $\gamma' \in S'$, therefore contains $\overline{S' \times \{\gamma\}} = \Gamma \times \{\gamma\}$ for each $\gamma \in \Gamma$. 

Since the conjugation map $\bH \times \bH \to \bH: (x,y) \mapsto y^{-1}xy$ is dominant, it sends $S''$ to a Zariski-dense subset of $\Gamma$. 
Following the thread of the argument, we see that the Zariski closure of $S \cap \Lambda \gamma_0$ contains the image of $S''$. 
This proves the theorem. 
\end{proof}

\medbreak

\begin{remark} \label{rem:necessityofconditions}
Each of the two properties assumed in \Cref{thm:pingpongdense} can be satisfied individually. 
Given a finitely generated Zariski-dense subgroup of a (connected) semisimple algebraic group, the existence of a local field and a representation satisfying the proximality property was first shown by Tits (see the proof of \cite[Proposition 4.3]{Tits72}). A refinement to non-connected simple groups can also be found in \cite[Theorem 1]{MargulisSoifer81}. 

The second property, transversality, can be established for a given non-central element using representation-theoretic techniques. 
However, it is not always possible to find a representation that works for all $h \in A_i \cup B_i$ at the same time. 

Even so, it may not always be possible to find a single representation which satisfies both properties of \Cref{thm:pingpongdense} simultaneously. 
Our next task will be to construct for real inner forms of $\SL_n$ and $\Res_{\C/\R} (\SL_n)$ a representation which does. 
This will be sufficient for the applications appearing in \S 4 \& \S 5. 
\end{remark}

\subsection{Constructing a proximal and transverse representation for \except{toc}{inner \texorpdfstring{$\R$}{R}-forms of }\texorpdfstring{$\SL_n$}{SLn} and \texorpdfstring{$\GL_n$}{GLn}} \label{subsec:constructingproximaltransverse}

Let $D$ be a finite-dimensional division $\R$-algebra and set $d = \dim_\R D$. Let $n \geq 2$ and let $\bH$ be any algebraic $\R$-group in the isogeny class of $\SL_{D^n}$ or $\GL_{D^n}$, viewing $D^n$ as a right $D$-module. 
For example, if $D = \C$ this means that $\bH$ is a quotient of the $\R$-group $\Res_{\C/\R} (\SL_n)$ or $\Res_{\C/\R} (\GL_n)$ by a (finite) central subgroup. 
The \emph{standard projective representation} of $\bH$ is the canonical morphism $\rhos: \bH \to \PGL_{D^n}$. 
This is the projective representation which will exhibit both proximal and transverse elements. 
\medbreak

First, we recall that an element $g \in \bG(\R)$, in some reductive $\R$-group $\bG$, is called \emph{$\R$-regular} if the number of eigenvalues (counted with multiplicity) of $\Ad(g)$ of absolute value 1 is minimal. 
Any \emph{$\R$-regular} element is semisimple (see \cite[Remark 1.6.1]{PrasadRaghunathan72}), and when $\bG$ is split, every {$\R$-regular} element is regular. 

With $\bH$ as specified above, an element $g \in \bH(\R)$ is $\R$-regular if and only if some (any) representative of $\rhos(g)$ in $\GL_{D^n}(\R)$ is conjugate to a diagonal $n$-by-$n$ matrix with entries in $D$ of distinct absolute values. 
Indeed, if $\rhos(g)$ is represented by $\diag(a_1, \dots, a_n)$ with $|a_i| \neq |a_j|$ for $i \neq j$, the absolute values of the eigenvalues of $\Ad(g)$ are $\{|a_i a_j^{-1}|\}_{1 \leq i, j \leq 1}$ (with the correct multiplicities) and are equal to $1$ only for $i = j$, which are the least possible occurrences. 
Conversely, if $g$ is $\R$-regular, the centralizer of the $\R$-regular element $\rhos(g)$ contains a unique maximal $\R$-split torus $\bS$ of $\PGL_{D^n}$ (see \cite[Lemma 1.5]{PrasadRaghunathan72}). 
Thus $\rhos(g)$ belongs to the centralizer of $\bS(\R)$, which, up to conjugation, is the subgroup of (classes of) diagonal $n$-by-$n$ matrices with entries in $D$; say $\rhos(g)$ is represented by $\diag(a_1, \dots, a_n)$. 
The absolute values of the eigenvalues of $\Ad(g)$ are again $\{|a_i a_j^{-1}|\}_{1 \leq i, j \leq 1}$. 
From the $\R$-regularity of $\rhos(g)$, we deduce that each value $|a_i a_j^{-1}|$ with $i \neq j$ must differ from $1$, as claimed. 

It follows from this description that if $\ell_\mathrm{max}$ (resp.\ $\ell_\mathrm{min}$) denotes the $D$-line in $D^n$ on which a $\R$-regular element $g \in \bH(\R)$ acts by multiplication by an element of $D^\times$ of largest (resp.\ smallest) absolute value, then $\ell_\mathrm{max} = \rA(g)$ is the attracting subspace of $g$ (resp.\ $\ell_\mathrm{min} = \rA(g^{-1})$), so that $g$ is biproximal.\footnote{Conversely, there exists a representation under which any proximal element is $\R$-regular, see \cite[Lemma 3.4]{PrasadRaghunathan72}. }
We record this here. 

\begin{lemma} \label{lem:Rregularstandardproximal}
Let $\bH$ and $\rhos$ be as above. 
Any $\R$-regular element $g \in \bH(\R)$ is biproximal under $\rhos$. 
\end{lemma}

So, in order to exhibit proximal elements in $\rhos(\Gamma)$ for $\Gamma \leq \bH(\R)$ a Zariski-dense subgroup, it suffices to show $\Gamma$ admits a $\R$-regular element. 
This is the content of the following theorem, due to Benoist and Labourie \cite[A.1 Th\'{e}or\`{e}me]{BenoistLabourie93}. We also refer the reader to the direct proof given by Prasad in \cite{Prasad94}. 

\begin{theorem}[Abundance of $\R$-regular elements, A.1 {Th\'{e}or\`{e}me} in \cite{BenoistLabourie93}] \label{thm:Rregulardense}
Let $\bG$ be a reductive $\R$-group. 
Let $\Gamma$ be a Zariski-dense subgroup of $\bG(\R)$. 
The subset of $\R$-regular elements in $\Gamma$ is Zariski-dense. 
\end{theorem}

\begin{corollary} \label{cor:standardproximal}
Let $\bH$ and $\rhos$ be as above. 
Let $\Gamma$ be a Zariski-dense subgroup of $\bH(\R)$. 
The elements $g \in \Gamma$ such that $\rhos(g)$ is biproximal, form a Zariski-dense subset of $\Gamma$. 
\end{corollary}

\begin{remark}
The existence of elements proximal under $\rhos$ in any Zariski-dense sub(semi)group can also be established using the results of Goldsheid and Margulis \cite[Theorem 6.3]{GoldsheidMargulis89} (see also \cite[3.12--14]{AbelsMargulisSoifer95}). This approach is more tedious, as the standard representation of $\GL_{D^n}$ does not admit proximal elements if $D^n$ is seen as a vector $\R$-space 
(which is in fact one of the motivations to extend the framework of \cite{Tits72} to division algebras). 
Instead, one should embed $\P_D(D^n)$ inside $\P_\R(\bigwedge_\R^d D^n)$ via the Plücker embedding, and exhibit proximal elements in that projective representation. 
\end{remark}
\medbreak

Next, we move on to the question of transversality. 
It turns out that under $\rhos$, every non-central element $h \in \bH(\R)$ satisfies the transversality condition of \Cref{thm:pingpongdense}. 

\begin{proposition} \label{prop:standardtransversal}
Let $\bH$ and $\rhos$ be as above. 
Let $h \in \bH(\R)$ be non-central. 
For every $p \in \P(D^n)$, the span of $\{\rhos(xhx^{-1})p \mid x \in \bH(\R) \}$ is the whole of $\P(D^n)$. 
\end{proposition}
\begin{proof}
Taking preimages in $\GL_{D^n}$, we may without loss of generality work with the action of $\GL_{D^n}$ on $D^n$ instead of $\rhos(\bH) = \PGL_{D^n}$ on $\P(D^n)$. 
We will show in this setting that, for every non-zero $v \in D^n$ and every non-central $h \in \GL_{D^n}(\R)$, the $\R$-span of $\{xhx^{-1} \cdot v \mid x \in \SL_{D^n}(\R) \}$ is the whole of $D^n$. 
The statement of the proposition then follows immediately by projectivization. 

Viewing $\End_D(D^n)$ as a vector $\R$-space, the conjugation action defines a linear representation of $\SL_{D^n}$ on $\End_D(D^n)$. 
This representation decomposes into two irreducible components: a copy of the trivial representation given by the action of $\SL_{D^n}$ on the center of $\End_D{D^n}$, and a copy of the adjoint representation given by the action of $\SL_{D^n}$ on the subspace $\mathfrak{sl}_n(D)$ of traceless endomorphisms. 

When $h$ is not central, it admits a distinct conjugate $xhx^{-1}$ of the same trace, hence the $\R$-span $W_h$ of $\{xhx^{-1} \mid x \in \SL_{D^n}(\R)\}$ contains for some $g \in \SL_{D^n}(\R)$ the nonzero traceless element $h' = h - ghg^{-1}$. 
In turn, $W_h$ contains the $\R$-span $W_{h'}$ of $\{xh'x^{-1} \mid x \in \SL_{D^n}(\R)\}$, a $\SL_{D^n}$-stable subspace of $\mathfrak{sl}_n(D)$ which must equal $\mathfrak{sl}_n(D)$, as the latter is irreducible for the adjoint action. 
Thus, either $W_h = \mathfrak{sl}_n(D)$ if $\Tr(h) = 0$, or $W_h = \End_D(D^n)$ if $\Tr(h) \neq 0$. 

Finally, for any non-zero $v \in D^n$ we have that $\mathfrak{sl}_n(D) \cdot v = D^n$, from which we conclude that the $\R$-span of $\{xhx^{-1} \cdot v \mid x \in \SL_{D^n}(\R) \}$ contains $W_h \cdot v = D^n$. 
\end{proof}
\medbreak

\begin{definition} \label{def:almostembed}
Let $\bG$ be a reductive $\Fi$-group with center $\bZ$, and $H \leq \bG(\Fi)$. 
For the purposes of this paper, we will say that \emph{$H$ almost embeds in a quotient $\bQ$ of $\bG$} if the kernel of the restriction $H \to \bQ(\Fi)$ of the quotient map is contained in $\bZ(\Fi)$. 
If there exists a simple (resp.\ adjoint simple) quotient $\bQ$ of $\bG$ in which $H$ almost embeds, we will say for short that \emph{$H$ almost embeds in a simple} (resp.\ \emph{adjoint simple}) \emph{quotient}. 

It is clear that if $\bQ$ is a simple factor of $\bG$ and $H$ is a subgroup of $\bQ(\Fi)$, then $H$ embeds in $\bQ$, and almost embeds in $\Ad \bQ$. 
In particular, when $\bG$ is itself simple, every subgroup of $\bG(\Fi)$ almost embeds in an adjoint simple quotient. 
\end{definition}

With this, we are ready to prove the following application of \Cref{thm:pingpongdense}, establishing the abundance of simultaneous ping-pong partners for finite subgroups in products of inner forms of $\SL_n$ and $\GL_n$ which almost embed in an adjoint simple quotient. 

\begin{theorem} \label{thm:pingpongdenseSLn}
Let $\bG$ be a reductive $\R$-group with center $\bZ$. 
Let $\Gamma$ be a subgroup of $\bG(\R)$ whose image in $\Ad \bG$ is Zariski-dense. 
Let $(A_i, B_i)_{i \in I}$ be a finite collection of finite subgroups of $\bG(\R)$. 

Suppose that for each $i \in I$, there exists a quotient of $\bG$ of the form $\PGL_{\LDi_i^{n_i}}$ for $\LDi_i$ some finite-dimensional division $\R$-algebra and $n_i \geq 2$, for which the kernels of the restrictions $A_i, B_i \to \PGL_{\LDi_i^{n_i}}(\R)$ are contained in $\bZ(\R)$. 
Then the set of regular semisimple elements $\gamma \in \Gamma$ of infinite order such that for all $i \in I$, the canonical maps
\begin{align*}
\langle \gamma, C_{A_i} \rangle \ast_{C_{A_i}} A_i &\to \langle \gamma, A_i \rangle && \text{where $C_{A_i} = A_i \cap \bZ(F)$,}\\
\langle \gamma, C_{B_i} \rangle \ast_{C_{B_i}} B_i &\to \langle \gamma, B_i \rangle && \text{where $C_{B_i} = B_i \cap \bZ(F)$,}\\
\intertext{and provided $\indx{A_i}{C_{A_i}} > 2$ or $\indx{B_i}{C_{B_i}} > 2$, also}
\langle A_i, C_i \rangle \ast_{C_i} \langle B_i, C_i \rangle &\to \langle A_i, B_i^{\gamma} \rangle && \text{where $C_i = C_{A_i} \cdot C_{B_i}$,}
\end{align*}
are all isomorphisms, is dense in $\Gamma$ for the join of the profinite topology and the Zariski topology. 
\end{theorem}

\begin{proof}
For $i \in I$, let $\rho_i$ denote the quotient map $\bG \to \PGL_{\LDi_i^{n_{i}}}$ afforded by the statement. 
The morphism $\rho_i$ is an irreducible projective representation of $\bG$ over $\LDi_{i}$. 
Note that $\rho_i$ factorizes via $\bG \to \Ad \bG \to \PGL_{\LDi_i^{n_{i}}}$.
In particular the image of $\Gamma$ in $\PGL_{\LDi_i^{n_{i}}}$ is Zariski-dense. 

\Cref{cor:standardproximal} shows that the set of elements in $\rho_i(\Gamma)$ which are biproximal is Zariski-dense in $\PGL_{\LDi_i^{n_{i}}}$; a fortiori, $\rho_i(\Gamma)$ contains a proximal element. 

Moreover, since by assumption the kernel of the restriction of $\rho_i$ to $A_i$ (resp.\ $B_i$) is contained in $\bZ(\R)$, every $h \in (A_i \cup B_i) \setminus \bZ(\R)$ maps to a non-trivial element in $\PGL_{\LDi_i^{n_{i}}}(\R)$. 
\Cref{prop:standardtransversal} then precisely states that $\rho_i$ satisfies the transversality condition of \Cref{thm:pingpongdense} (see \Cref{rem:transversalitycondition}). 
We are thus at liberty to apply \Cref{thm:pingpongdense} to $\Gamma \leq \bG(\R)$ and the collection $(A_i, B_i)_{i \in I}$, deducing this theorem. 
\end{proof}

\begin{remark} \label{rem:almostembeddingnecessary}
Let $\Fi$ be any field, and let $\bG$ be a reductive $\Fi$-group with center $\bZ$. 
In order for a subgroup $H \leq \bG(\Fi)$ to admit a ping-pong partner in $\bG(\Fi)$, it is necessary that $H$ embeds in an adjoint simple factor. 
In fact, if the subgroup $\langle \gamma, H \rangle$ is the free amalgamated product of $\langle \gamma, C \rangle$ and $H$ over $C = H \cap \bZ(\Fi)$, then in the quotient $\bG / \bZ$, the image of $\langle \gamma, H \rangle$ is the free product of the images of $\langle \gamma \rangle$ and $H$. 
Indeed, the kernel of the quotient map being central in $\langle \gamma, C \rangle \ast_C H$, it lies in $C$, hence coincides with $C$. 
But $\bG / \bZ$ is the direct product of adjoint simple quotients of $\bG$, so by \Cref{prop:amalgaminproduct}, $H / C$ embeds in (the $\Fi$-points of) one of these factors. 

Similarly, if $A$ and $B$ are such that $\langle A, B \rangle \leq \bG(\Fi)$ is the free amalgamated product of $\langle A, C \rangle$ and $\langle B, C \rangle$ over $C = \langle A \cap \bZ(\Fi), B \cap \bZ(\Fi) \rangle$, then in the quotient $\bG / \bZ$, the image of $\langle A, B \rangle$ is again the free product of the images of $A$ and $B$. 
By \Cref{prop:amalgaminproduct} once more, $A / C$ and $B / C$ both embed in (the $\Fi$-points of) a common adjoint simple factor. 

In other words, \Cref{thm:pingpongdenseSLn} states that a collection of finite subgroups $(H_i)_{i \in I}$ in an $\R$-group $\bG$ whose simple quotients are each isogenous to some $\PGL_{n}(D)$ with $n \geq 2$, admits simultaneous ping-pong partners in $\Gamma$ \emph{if and only if} each $H_i$ embeds in an adjoint simple factor. 
Similarly, two finite subgroups $A$, $B$ of $\bG(\Fi)$ admit conjugates by $\Gamma$ which play ping-pong \emph{if and only if} $A$ and $B$ embed in the same adjoint simple factor. 
\end{remark}

\begin{remark} \label{rem:otherrepresentations}
There are versions of \Cref{thm:pingpongdenseSLn} for semisimple $\R$-groups of other types, but proving them requires a more careful study of the representation theory of $\bG$ to exhibit a representation playing the role of $\rhos$. 
However, as indicated in \Cref{rem:necessityofconditions}, there are also cases where one needs additional information on the $H_i$ to get a representation satisfying the transversality assumption of \Cref{thm:pingpongdense}. 

There are also versions of the theorem for other local fields. 
However, to prove those one needs additional information on $\Gamma$. Indeed, over a local field different from $\R$, bounded Zariski-dense subgroups exist, and a bounded subgroup obviously never admits proximal elements. 
Nevertheless, in the extreme case where $\bG$ is defined over a number field $\Fi$ and $\Gamma = \bG(\Fi)$, we prove below a version of \Cref{thm:pingpongdenseSLn} using the same method, which only requires to assume $n_i \geq 2$ when $\LDi_i$ is a field. 
\end{remark}

\begin{theorem} \label{thm:pingpongdenseSLnrationalpoints}
Let $\bG$ be a reductive group defined over a number field $\Fi$ with center $\bZ$. 
Let $(A_i,B_i)_{i \in I}$ be a finite collection of finite subgroups of $\bG(\Fi)$. 

Suppose that for each $i \in I$, there exists a (non-trivial) simple quotient of $\bG$ of the form $\PGL_{D_i^{n_i}}$ for $D_i$ some finite-dimensional division $\Fi$-algebra, for which the kernels of the restrictions $A_i, B_i \to \PGL_{D_i^{n_i}}(\Fi)$ are contained in $\bZ(\Fi)$. 
Then the set of regular semisimple elements $\gamma \in \bG(\Fi)$ of infinite order such that for all $i \in I$, the canonical maps
\begin{align*}
\langle \gamma, C_{A_i} \rangle \ast_{C_{A_i}} A_i &\to \langle \gamma, A_i \rangle && \text{where $C_{A_i} = A_i \cap \bZ(F)$,}\\
\langle \gamma, C_{B_i} \rangle \ast_{C_{B_i}} B_i &\to \langle \gamma, B_i \rangle && \text{where $C_{B_i} = B_i \cap \bZ(F)$,}\\
\intertext{and provided $\indx{A_i}{C_{A_i}} > 2$ or $\indx{B_i}{C_{B_i}} > 2$, also}
\langle A_i, C_i \rangle \ast_{C_i} \langle B_i, C_i \rangle &\to \langle A_i, B_i^{\gamma} \rangle && \text{where $C_i = C_{A_i} \cdot C_{B_i}$,}
\end{align*}
are all isomorphisms, is dense in $\bG(\Fi)$ for the join of the profinite topology and the Zariski topology. 
\end{theorem}

\begin{proof}
For $i \in I$, let $\rho_i$ denote the quotient map $\bG \to \PGL_{D_i^{n_i}}$ afforded by the statement. 

Let $S_0$ be the finite (possibly empty) set of places of $\Fi$ where the weak approximation property fails for $\bG$ (see \cite[\S7.3 \& Theorem 7.7]{PlatonovRapinchuk94}). 
Pick for each $i \in I$ a completion $\LFi_i$ of $\Fi$, not belonging to $S_0$, and over which the division algebra $D_i$ splits; 
this means that $\PGL_{D_i^{n_i}}$ becomes isomorphic to $\PGL_{n_i d_i}$ over $\LFi_i$, for $d_i = \deg_\Fi(D_i)$. 

For each $i \in I$, the set $R$ of $\LFi_i$-regular elements in $\PGL_{n_i d_i}(\LFi_i)$ (that is, of elements $g$ for which the number of eigenvalues of $\Ad(g)$ of absolute value $1$ is minimal), is open in the local topology. 
Moreover, $R$ obviously intersects the image of the simply connected group $\SL_{n_i d_i}(\LFi_i)$, as witnessed by any diagonal matrix of determinant $1$ whose entries have pairwise distinct absolute values. 
In consequence, the preimage $\rho_i^{-1}(R)$ is a non-empty, open subset of $\bG(\LFi_i)$ (for the local topology). 
By weak approximation, $\bG(\Fi)$ is dense in $\bG(\LFi_i)$, hence we can find an element $\gamma_i \in \bG(\Fi) \cap \rho_i^{-1}(R)$. 
By construction, $\gamma_i$ is $\LFi_i$-regular, hence biproximal, under $\rho_i$. 

In order to apply \Cref{thm:pingpongdense} to $\Gamma = \bG(\Fi)$, the finite groups $(A_i,B_i)_{i \in I}$, and the pairs $(\LFi_i, \rho_i)_{i \in I}$, it only remains to check the transversality condition. 
Taking into account \Cref{rem:transversalitycondition}, the latter is a direct consequence of the homologue of \Cref{prop:standardtransversal} over a general local field, whose proof is identical (having replaced $\R$ by said local field). 
\end{proof}

As highlighted in \Cref{rem:almostembeddingnecessary}, the existence of a free amalgamated product $A \ast_C B$ in $\bG$ is in general subject to the existence of almost embeddings of $A$ and $B$ in an appropriate common factor of $\Ad \bG$. 
Of course, if $\Ad \bG$ is a simple group to begin with, this embedding condition is void. 

In this regard, \Cref{thm:pingpongdenseSLnrationalpoints} implies at once that if $\bG$ is isogenous to $\PGL_{\Di^n}$ with $\Di$ a finite-dimensional division $\Fi$-algebra, and $A$, $B$ are finite subgroups of $\bG(\Fi)$, then there exists $\gamma \in \bG(\Fi)$ for which $\langle A, B^{\gamma} \rangle$ is freely amalgamated along $C = (A \cap \bZ(\Fi)) \cdot (B \cap\bZ(\Fi))$. 

One can wonder whether in the same setting, there is a conjugate of $B$ for which the subgroup $\langle A, B^{\gamma} \rangle$ is freely amalgamated along $A \cap B$ instead. 
This fairly natural question currently eludes us.

\section{Free products between finite subgroups of units in a semisimple algebra}\label{sec:pingpongsemisimplealgebra}

In this section and the next ones, we adopt the following notation. 
\begin{enumerate}[label=$\bullet$,leftmargin=\parindent]
	\item $\PCI(\Alg)$ is the set of primitive central idempotents of a finite-dimensional (semisimple) $\Fi$-algebra $\Alg$. 
Each $e \in \PCI(\Alg)$ corresponds to an irreducible factor algebra $Me$ of $M$, and $\pi_e : \Alg \to \Alg e$ denotes the projection of $\Alg$ onto $\Alg e$. 
	\item $\U(\Alg) = \Alg^\times = \GL_1(\Alg) = \GL_\Alg(\Fi)$ all denote the group of units of the $\Fi$-algebra $\Alg$. 
	\item $RG$ denotes the group ring of $G$ with coefficients in a ring $R$, and $\V(RG)$ the group of units in $RG$ whose augmentation equals 1. 
\end{enumerate}

\subsection{Simultaneous partners in the unit group of an order }\label{subsec:pingpongunitgroupoforder}

\noindent By Wedderburn's theorem, every semisimple $\Fi$-algebra $\Alg$ factors as
\[
\Alg = \End(V_1) \times \cdots \times \End(V_m),
\]
for $V_i$ an $n_i$-dimensional right module over some finite-dimensional division $\Fi$-algebra $\Di_i$, $i = 1, \dots m$. 
In consequence, the $\Fi$-group of units of $\Alg$ identifies with the reductive group
\begin{align*} \label{eq:decompositionunitsalgebra}
\bG &= \GL_M = \GL_{\Di_1^{n_1}} \times \cdots \times \GL_{\Di_m^{n_m}}, 
\intertext{whose adjoint is}
\Ad \bG &= \PGL_M = \PGL_{\Di_1^{n_1}} \times \cdots \times \PGL_{\Di_m^{n_m}}. 
\end{align*}
The factors of $\bG$ are in one-to-one correspondence with those of $\Alg$, while the non-trivial factors of $\Ad \bG$ correspond to the factors of $\Alg$ which are not fields. 

The original motivation for this project was the study of free amalgamated products inside $\U(\O)$, the unit group of an order $\O$ in $\Alg$; more precisely, the aim was to answer \Cref{conj:bicyclicgenericallyfree}. 
In this context, the following neat necessary and sufficient condition for a given finite subgroup of $\U(\O)$ to admit a ping-pong partner, is a direct consequence of \Cref{thm:pingpongdenseSLn}. 

\begin{theorem} \label{thm:pingpongdensesemisimplealgebra}
Let $\Fi$ be a number field, $\Alg$ be a finite-dimensional semisimple $\Fi$-algebra, and $\O$ be an order in $\Alg$. 
Let $\Gamma$ be a subgroup of $\U(\O)$ for which the Zariski closure of the image of $\Gamma$ in $\PGL_\Alg$ contains the Zariski-connected component of the image of $\U(\O)$.\footnote{This condition is satisfied if $\Gamma$ is Zariski-dense in $\U(\O) = \GL_1(\O)$ or in $\SL_1(\O)$, e.g.\ $\Gamma$ is a finite-index subgroup of $\U(\O)$. } 
Let $A$ and $B$ be finite subgroups of $\U(\O)$, and assume that $\indx{A}{A \cap \ZZ(\Alg)} > 2$. 

There exists $\gamma \in \Gamma$ of infinite order with the property that the canonical maps
\begin{align*}
	\langle \gamma, C_{A} \rangle \ast_{C_{A}} A &\to \langle \gamma, A \rangle && \text{where $C_{A} = A \cap \ZZ(\Alg)$,}\\
	\langle \gamma, C_{B} \rangle \ast_{C_{B}} B &\to \langle \gamma, B \rangle && \text{where $C_{B} = B \cap \ZZ(\Alg)$,}\\
	\langle A, C \rangle \ast_{C} \langle B, C \rangle &\to \langle A, B^{\gamma} \rangle && \text{where $C = C_{A} \cdot C_{B}$,}
\end{align*}
are all isomorphisms, if and only if $A$ and $B$ almost embed in $\PGL_{\Alg e}$ for some common $e \in \PCI(\Alg)$ for which $\Alg e$ is neither a field nor a totally definite quaternion algebra. 

Moreover, in the affirmative, the set of such elements $\gamma$ which are regular semisimple and of infinite order, is dense in $\Gamma$ for the join of the Zariski and the profinite topology. 
\end{theorem}

\begin{proof}
We can base-change $\bG = \GL_\Alg$ to the $\R$-group $\Res_{\Fi / \Q} \bG \times_\Q \R$,
whose $\R$-points $\bG(\Fi \ot_\Q \R)$ are a product of groups of the form $\GL_n(\R)$, $\GL_n(\C)$, or $\GL_n(\H)$, for various $n \geq 1$. 

\medbreak

Any subgroup $H$ of $\U(\Alg) = \bG(\Fi)$ embeds in $\bG(\Fi \ot_\Q \R)$. 
In fact, $H$ (almost) embeds in a $\Fi$-simple factor of $\bG$ if and only if it does so in any $\R$-simple factor of $\Res_{\Fi / \Q} \bG \times_\Q \R$. 
More precisely, let $\LFi_1, \dots, \LFi_s$ denote the summands of the étale $\R$-algebra $\Fi \ot_\Q \R$; they are precisely the different archimedean completions of $\Fi$. 
Given a finite-dimensional division algebra $\Di$ over $\Fi$, let $\Di_{ij}$ be the division $\R$-algebras such that $\Di \ot_\Fi K_i \cong \prod_{j=1}^{m_i} \rM_{r_{ij}}(\Di_{ij})$ as $\R$-algebras. 
The group ${\Res_{\Fi / \Q} \GL_{\Di^n}} \times_{\Q} \R$ then factors into the product $\prod_{i=1}^s \prod_{j=1}^{m_i} \GL_{\Di_{ij}^{n r_{ij}}}$. 
The image of $\GL_{\Di^n}(\Fi)$ in this product is obtained by embedding it diagonally using the canonical maps $\GL_{\Di^n}(\Fi) \to \GL_{\Di^n}(K_i) \to \GL_{\Di_{ij}^{n r_{ij}}}(\R)$. 
Thus if $H$ (almost) embeds in a factor $\PGL_{\Di^n}$ over $\Fi$, then it does so in any of the $\PGL_{\Di_{ij}^{n r_{ij}}}$ over $\R$, and the converse is obvious. 

Now, an adjoint simple quotient $\PGL_{\Di_{ij}^{nr_{ij}}}$ over $\R$ of a given factor $\GL_{\Di^{n}}$ of $\bG$ satisfies $nr_{ij} = 1$, if and only if the $j$th factor in $\Alg e \ot_\Fi \LFi_i$ is a division algebra, where $e$ is the projection onto the factor of $\Alg$ corresponding to $\GL_{\Di^{n}}$. 
In other words, the factor $\GL_{\Di^{n}}$ has a simple quotient $\PGL_{\Di_{ij}^{nr_{ij}}}$ with $nr_{ij} \geq 2$ for some $i, j$, if and only if $\Alg e$ is not a division algebra that remains so under every archimedean completion of its center.
This amounts in turn to $\Alg e$ not being a field nor a totally definite quaternion algebra. 
\medbreak

Next, let $\bG_{\mathrm{is}}$ denote the $\R$-subgroup of $\Res_{\Fi / \Q} \bG \times_\Q \R$ which is the direct product of those subgroups $\GL_{D_{ij}^{nr_{ij}}}$ for which $nr_{ij} \geq 2$. 
Since $\U(\O)$ is an arithmetic subgroup of $\U(\Alg) = \bG(\Fi)$, a classical theorem of Borel and Harish-Chandra \cite[Theorem 7.8]{BorelHarishChandra62} attests that the connected component of $\U(\O)$ in $\Res_{\Fi / \Q} \bG \times_\Q \R$ is a lattice in the derived subgroup $\cD \bG_{\mathrm{is}}$ of $\bG_{\mathrm{is}}$. 
In consequence, the image of $\U(\O)$ hence of $\Gamma$ in $\Ad \bG_{\mathrm{is}}$ is Zariski-dense. 
Let $f$ denote the canonical map $\bG(\R) \to \bG_{\mathrm{is}}(\R)$, whose kernel is the product of the compact factors of $\bG(\R)$. 
Note that $\ker f$ commutes with $\bG_{\mathrm{is}}(\R)$, and that $\ker f \cap \Gamma$ is finite. 

In view of all the above, given that the subgroups $A$ and $B$ satisfy the embedding condition stated in the theorem, we deduce from \Cref{thm:pingpongdenseSLn} applied to $\bG_{\mathrm{is}}$ the existence of a dense set $S \subset f(\Gamma)$ of ping-pong partners for $f(A)$ and $f(B)$. 
By \Cref{lem:preimageamalgam}, the preimage $f^{-1}(S) \cap \Gamma$ consists of elements $\gamma \in \Gamma$ for which the canonical maps
\begin{equation*}
\langle \gamma, C_{A} \rangle \ast_{C_{A}} A \to \langle \gamma, A \rangle, \quad
\langle \gamma, C_{B} \rangle \ast_{C_{B}} B \to \langle \gamma, B \rangle, \quad
\langle A, C \rangle \ast_{C} \langle B, C \rangle \to \langle A, B^{\gamma} \rangle
\end{equation*}
are isomorphisms. 

As $S$ is dense in the join of the Zariski and the profinite topology, the same holds for $f^{-1}(S) \cap \Gamma$. 
Indeed, if $\Lambda \gamma_0$ is a coset of finite index in $\Gamma$, and $U$ is a Zariski-open subset of $\Gamma$ intersecting it, perhaps after shrinking and translating by $\ker f \cap \Gamma$, we can arrange that $\Lambda \gamma_0$ and $U$ are contained in the connected component $\Gamma^\circ$ of $\Gamma$, and that $(\ker f \cap \Gamma^\circ) \cdot U = U$. 
Then $f(\Lambda \gamma_0 \cap U)$ equals the open set $f(\Lambda \gamma_0) \cap f(U)$. 
We may thus pick $x \in S \cap f(\Lambda \gamma_0 \cap U)$, implying that $f^{-1}(S) \cap \Lambda \gamma_0 \cap U$ is non-empty. 
\medbreak

It remains to verify that the embedding condition is necessary. 
Suppose $\gamma \in \Gamma$ is such that the canonical map $A \ast_C B \to \langle A, B^{\gamma} \rangle$ is an isomorphism. 
Let $\bG_1$ (resp.\ $\bG_2$) denote the product of the factors of $\bG$ over $\Fi$ for which the corresponding factor $\Alg e$ of $\Alg$ is not (resp.\ is) a field or a totally definite quaternion algebra. 
Because this product decomposition is defined over $\Fi$, the projections of $\U(\O)$ in $\bG_1(F \ot_\Q \R)$ and $\bG_2(F \ot_\Q \R)$ are discrete. 
Since $\cD \bG_2(F \ot_\Q \R)$ is compact, the image of $\U(\O)$ in $\bG_2(F \ot_\Q \R)$ is in fact finite. 

As $\bG = \bG_1 \times \bG_2$, \Cref{prop:amalgaminproduct} shows that one of the kernels $N_1$, $N_2$ of the respective projections $\pi_i: \langle A, B^{\gamma} \rangle \to \bG_i(F \ot_\Q \R)$, is contained in $C$. 
Of course, $N_2$ can not be contained in $C$, otherwise the image of $\U(\O)$ in $\bG_2(F \ot_\Q \R)$ would contain the infinite group $(A/N_2) \ast_{C/N_2} (B/N_2)$. 
We deduce that $N_1 \subset C$, that is, $\langle A, B^{\gamma} \rangle$ almost embeds in $\bG_1$. 
Another application of \Cref{prop:amalgaminproduct} then shows that $\langle A, B^{\gamma} \rangle$ almost embeds in some factor of $\Ad \bG_1$ over $\Fi$, associated with a factor of $\Alg$ which is not a field nor a totally definite quaternion algebra as claimed. 
\end{proof}

\begin{remark} \label{rem:partialpingpongsemisimplealgebra}
In the unfortunate event that $\indx{A}{A \cap \ZZ(\Alg)} = \indx{B}{B \cap \ZZ(\Alg)} = 2$, one can draw no conclusion about the map $\langle A, C \rangle \ast_{C} \langle B, C \rangle \to \langle A, B^{\gamma} \rangle$. 
However, having removed this line from the statement, the remainder of \Cref{thm:pingpongdensesemisimplealgebra} still holds (with the same proof). 
Indeed, in the same manner as in the proof, \Cref{prop:amalgaminproduct} can be used to show that the existence of a free product $\langle \gamma \rangle \ast_{C_A} A$ inside $\U(\O)$ implies that $A$ almost embeds inside $\PGL_{\Alg e}$ for some factor $\Alg e$ which is neither a field nor a totally definite quaternion algebra. 

Thus, formally speaking, if there were a free product of the form $\langle \gamma \rangle \ast A$ in $\U(\O)$, one could infer that $A$ embeds in an appropriate quotient $\PGL_{\Alg e}$ of $\GL_{\Alg}$, and then apply \Cref{thm:pingpongdensesemisimplealgebra} with $A = B$ to deduce the existence of $A \ast A$ in $\U(\O)$. 
Similarly, if $A \ast B$ occurs in $\U(\O)$, one could infer that $A$ and $B$ embed in an appropriate quotient $\PGL_{\Alg e}$, and then apply \Cref{thm:pingpongdensesemisimplealgebra} to deduce that $\langle \gamma \rangle \ast A$ and $\langle \gamma \rangle \ast B$ also occur. 
But of course both these deductions are self-evident, as $\langle A, A^{\gamma} \rangle$ is a freely generated subgroup of $\langle \gamma \rangle \ast A$, and the same holds for $\langle ab, A \rangle$ and $\langle ab, B \rangle$ in $A \ast B$ (for non-trivial $a \in A$, $b \in B$). 
\end{remark}

\begin{remark}
There is an $S$-arithmetic version of \Cref{thm:pingpongdensesemisimplealgebra}, in which a ping-pong partner $\gamma$ for $A$ and $B$ exists in a Zariski-dense subgroup $\Gamma$ of an $S$-arithmetic order $\cO[S^{-1}]$, if and only if over some completion $\Fi_v$ with $v \in S$, $A$ and $B$ almost embed in a common factor of $\PGL_\Alg(\Fi_v)$ in which $\Gamma$ admits proximal elements. 
The proof is essentially the same, but requires an $S$-arithmetic version of \Cref{thm:pingpongdenseSLnrationalpoints}, which in turn requires a more careful analysis of the dynamics of $\Gamma$ at the places of $S$. 
\end{remark}

Although \Cref{thm:pingpongdenseSLn} and \Cref{thm:pingpongdensesemisimplealgebra} are nice existence results, they leave open the following two questions.
\begin{questions}\label{que:existencefaithfulembeddings}
With the notation of \Cref{thm:pingpongdensesemisimplealgebra}:
\begin{enumerate}[itemsep=1ex,topsep=1ex,label=\textup{(\roman*)},leftmargin=2em]
	\item \label{que:existencefaithfulembeddings1}
When do $A$ and $B$ (almost) embed in a simple factor of $\PGL_\Alg$? 
If so, do they (almost) embed in a common simple factor of $\PGL_\Alg$? 
	\item \label{que:existencefaithfulembeddings2}
How can we construct a ping-pong partner $\gamma$ concretely? 
\end{enumerate}
\end{questions}
 
In the remainder of this section, we present a method to approach question \ref{que:existencefaithfulembeddings2}, which will reduce the problem to constructing certain \emph{deformations} of $A$ or $B$ (see \Cref{def:firstorderdeformation}); the main result is \Cref{thm:pingpongdeformations}. 
Question \ref{que:existencefaithfulembeddings1} will be addressed in \Cref{sec:faithfulembedding}.

\subsection{Deforming finite subgroups and subalgebras} \label{subsec:deformations}

Keeping the notation of \Cref{subsec:pingpongunitgroupoforder}, the aim of this section is to introduce an explicit linear method that allows to replace a finite subgroup $H \leq \U(\Alg)$ by an isomorphic copy which has the necessary ping-pong dynamics. 
Concretely, we want to construct a group morphism of the form
\begin{equation*}
\fod : H \rightarrow \U(\Alg): h \mapsto \fod(h) = h + \delta_h.
\end{equation*}
This is possible when the map $\delta = \fod - 1$ satisfies the following conditions.

\begin{definition} \label{def:firstorderdeformation}
Let $H$ be a subgroup of $\U(\Alg)$. 
We call an $\Fi$-linear map $\fod: H \rightarrow \Alg$ \emph{a first-order deformation of $H$} if the map $\delta = \fod - 1$ satisfies the following conditions:
\begin{enumerate}[leftmargin=3cm,itemsep=1ex,topsep=1ex]
	\item[\textup{(Derivation)}] \label{eq:derivation}
$\delta_{hk} = \delta_{h} k + h \delta_{k}$ for all $h,k \in H$; 
	\item[\textup{(Order~1)}] \label{eq:order1}
$\delta_h \delta_k = 0$ for all $h,k \in H$. 
\end{enumerate}

A straightforward calculation shows that if $\fod$ is a first-order deformation of $H$, then the maps 
\(
\fod_t(h) = h + t \delta_h
\)
for $t \in \Fi$ are {group morphisms} from $H$ to $\U(\Alg)$, interpolating between the identity $\fod_0$ and $\fod_1 = \fod$. 
In fact, if a linear map $\fod: H \to \Alg$ is a group morphism and $\delta = \fod - 1$ satisfies either (Derivation) or (Order~1), then $\delta$ also satisfies the remaining property. 
Moreover, since $\fod_t$ is assumed to be linear, it extends uniquely to an algebra morphism $\Fi H \to \Alg$. 
We define a \emph{first-order deformation} of a subalgebra $\Subalg$ of $\Alg$ analogously, so that first-order deformations of subalgebras are algebra morphisms, and the linear extension of a deformation of $H$ is a deformation of $\Fi H$. 
We say that $\fod$ is an \emph{inner (first-order) deformation} when the derivation $\delta$ is inner over $\Alg$, that is, when $\delta_h = [n,h]$ for some $n \in \Alg$. 
\end{definition}

\begin{examples}\label{ex:elementarydeformations}
Let $H \leq \U(\Alg)$. 
\begin{enumerate}[itemsep=1ex,topsep=1ex,label=\textup{(\arabic*)},leftmargin=2em]
	\item \label{ex:elementarydeformations1}
If $n$ is an element of $\Alg$ satisfying $n h n = 0$ for all $h \in H$, then the assignment
\[
\delta_h = [n,h]
\]
defines a first-order deformation of $H$, which is actually given by the conjugation
\[
\fod(h) = (1 + n) h (1+n)^{-1} = h + [n, h]. 
\]
This deformation is inner by construction. 

	\item \label{ex:elementarydeformations2}
If $m \in \Alg$ satisfies $m h = m$ for all $h\in H$, then the assignment
\[
\delta_h = (1-h) m
\]
defines a first-order deformation of $H$. 
(The assignment $\delta_h = m(1-h) = 0$ defines the trivial deformation.) 

Assume for a moment that $H$ is finite and that $\Char \Fi$ does not divide $|H|$; set $e = \frac{1}{|H|} \sum_{h \in H} h$. 
Then this deformation is in fact of the first kind with $n = (1-e)m$, as $(1-e)m \cdot h \cdot (1-e)m = 0$ since $m(1-e) = 0$, and
\[
\delta_h = [(1-e)m , h] = (1-e)m - (h-e)m = (1-h)m. 
\]
Note that under this additional assumption on $H$, the condition $mh=m$ for all $h \in H$ is equivalent to $me = m$. 
This deformation might be trivial, for instance when $m = e$, but if $H$ is $\Fi$-linearly independent and not central, one can find some $m$ for which it is not.

	\item \label{ex:elementarydeformations3}
Any example of the second kind obviously satisfies $\delta_g h = \delta_g$ for all $g, h \in H$. 
The converse holds under the assumption that $H$ is finite and that $\Char \Fi$ does not divide $|H|$. Set $e = \frac{1}{|H|} \sum_{h \in H} h$ as above. 
If $\fod = 1 + \delta$ is a first-order deformation which happens to satisfy $\delta_g h = \delta_g$ for all $g, h \in H$, then the equation $\delta_e = \delta_{he} = \delta_h e + h \delta_e = \delta_h + h \delta_e$ implies that
\[
\delta_h = (1-h) \delta_e. 
\]
Thus, this deformation is of the second kind with $m = \delta_e$. 
Since $\delta_{e} = e \delta_e + \delta_e e$ implies $n = (1-e) \delta_e = \delta_e$, we deduce as in the first two examples that
\[
\fod(h) = (1+ \delta_e) h (1+\delta_e)^{-1}. 
\]
\end{enumerate}

Note that when $\fod$ is an inner first-order deformation, $\fod(h) = h$ for every $h \in H \cap \ZZ(\Alg)$. 
Also, examples (ii) and (iii) above can only occur if $H \cap \Fi = \{1\}$. 
\end{examples}

\begin{lemma}\label{lem:basicdeformationinjective}
The kernel of a first-order deformation $\fod: H \to \Alg$ consists of unipotent elements. 
In consequence, if $H$ is finite and $\Char \Fi$ does not divide $|H|$, every first-order deformation $H \to \Alg$ is injective. 
\end{lemma}
\begin{proof}
An element $h$ lies in the kernel of $\fod$ if and only if $\delta_h = 1-h$. 
By assumption, $(\delta_h)^2 = (1-h)^2 = 0$, showing that $h$ is a unipotent element of $\Alg^\times$. 

Over a field of characteristic $p$ (resp.\ 0), any non-trivial unipotent element has order $p$ (resp.\ infinite order); so if $H$ has no elements of order $p$ nor $\infty$, the map $\fod$ is injective. 
\end{proof}

If $H$ is infinite, or if $\Char \Fi$ divides $|H|$, there are first-order deformations which are not injective. 
For instance, assuming $\Char \Fi \neq 2$, the trivial map
\[
H = \left\{{\begin{bmatrix} 1&x\\0&1\end{bmatrix}} \middle| \, x \in \Fi \right \} \to \rM_2(F): h \to 1
\]
is an inner first-order deformation, associated with $\delta_h = [n,h] = 1-h$ for $n = {\begin{smallbmatrix} 1/2&0\\0&-1/2\end{smallbmatrix}}$. 
\medskip

In more generality, we have the following description of first-order deformations of semisimple subalgebras, which may be of independent interest. 
\begin{theorem} \label{thm:separabledeformationsinner}
Let $\Alg$ be an $\Fi$-algebra, and let $\Subalg$ be a separable subalgebra of $\Alg$. 
Then every first-order deformation $\fod: \Subalg \to \Alg$ is given by conjugation, that is, there exists $a \in \Alg$ such that $\fod(b) = a b a^{-1}$ for every $b \in \Subalg$. 
In particular, any such deformation extends to an automorphism of $\Alg$ and fixes $\Subalg \cap \ZZ(\Alg)$. 
\end{theorem}

\begin{proof}
Recall that the separability of $\Subalg$ means that $\Subalg$ admits a \emph{separating idempotent}, that is, there is an element $e \in \Subalg \ot_\Fi \Subalg$ whose image under the multiplication map $\mu: \Subalg \ot_\Fi \Subalg \to \Subalg$ is $1$ and which satisfies $(b \ot 1)e = e(1 \ot b)$ for every $b \in \Subalg$. 

Let $I$ denote the kernel of the $(\Subalg,\Subalg)$-bimodule map $\mu$. 
Recall that the canonical map 
\[
d: \Subalg \to I: b \mapsto b \ot 1 - 1 \ot b
\]
identifies $I$ with the bimodule of (non-commutative) differentials of the algebra $\Subalg$. 
That is, for any derivation $\delta: \Subalg \to V$ to a $(\Subalg,\Subalg)$-bimodule $V$, there exists a unique bimodule map $f: I \to V$ such that $\delta = f \circ d$ (see \cite[A III §10 N°10 Proposition 17]{BourbakiAlgebre3}). 
Applying this to the derivation $\delta = \fod - 1$, for $\fod$ a first-order deformation $\Subalg \to \Alg$, we obtain existence of a unique bimodule map $f: I \to \Alg$ satisfying $\delta(b) = f(b \ot 1 - 1 \ot b)$ for every $b \in \Subalg$. 

Note that $e-1 \in I$ by construction. 
Inside $\Alg$, we now compute
\begin{align*}
[f(e-1),b] 	&= f(e-1)\cdot b - b \cdot f(e-1) = f\big((e-1)(1 \ot b) - (b \ot 1)(e-1)\big) \\
		&= f(e(1 \ot b) - 1 \ot b - e(1 \ot b) + b \ot 1) \\
		&= f(b \ot 1 - 1 \ot b) = f(d(b)) = \delta(b),
\end{align*}
showing that the derivation $\delta: \Subalg \to \Alg$ is inner, given by the adjoint of $f(e-1) \in \Alg$. 

The bimodule $I$ of differentials of $\Subalg$ is generated as a left (respectively right) module by the image of $d$ (see \cite[A III §10 N°10 Lemme 1]{BourbakiAlgebre3}), hence the same holds for the image of $f$; 
in other words, $f(e-1) \in \Subalg \cdot \delta(B) = \delta(\Subalg) \cdot \Subalg$. 
As by assumption $\delta(\Subalg) \cdot \delta(\Subalg) = 0$, we conclude that $f(e-1) b f(e-1) = 0$ for every $b \in \Subalg$. 
This shows that
\[
\fod(b) = b + f(e-1) b - b f(e-1) = (1+f(e-1)) b (1 - f(e-1))
\]
is given by conjugating $b$ by $1+f(e-1)$, proving the theorem. 
\end{proof}

\begin{remark}
Recall that an $\Fi$-algebra $\Subalg$ is separable if and only if $\Subalg$ is absolutely semisimple, which is in turn equivalent to $\Subalg$ being semisimple with étale center (see respectively \cite[A VIII §13 N°5 Théorème 2 \& N°3 Théorème 1]{BourbakiAlgebre8}). 
The center of a finite-dimensional semisimple algebra is always a product of fields, and is automatically étale if the base field $\Fi$ is perfect. 
The example above the theorem illustrates why it is essential to assume that $\Subalg$ is semisimple. 

It follows from the general theory that when $\Subalg$ is a separable algebra, every derivation of $\Subalg$ with values in a $(\Subalg,\Subalg)$-bimodule is inner (see \cite[A VIII §13 N°7, Corollaire]{BourbakiAlgebre8}). 
However, without additional information on $n$, this does not formally imply that a first-order deformation $\fod: \Subalg \to \Alg: b \mapsto b + [n,b]$ is given by conjugation. 
The proof of \Cref{thm:separabledeformationsinner} exhibits a suitable such element $n$. 
\end{remark}

\subsection{Ping-pong between two given finite subgroups of \texorpdfstring{$\U(\Alg)$}{U(M)}}
Given two finite subgroups $\Aa$, $\Bb$ in $\GL_n(\LDi)$, the aim of this subsection is to provide a constructive method to obtain a subgroup of $\GL_n(\LDi)$ isomorphic to $\Aa \ast \Bb$ by deforming $\Bb$ using the first-order deformations introduced in \Cref{subsec:deformations}.

\begin{theorem}\label{thm:pingpongdeformations}
Let $\LFi$ be a local field of characteristic $0$, $\LDi$ be a finite-dimensional division $\LFi$-algebra, and $\Aa$, $\Bb$ be finite subgroups of $\GL_n(\LDi)$. 
Set $C = \Aa \cap \Bb$, and assume that $\indx{\Aa}{C} > 2$ or $\indx{\Bb}{C} > 2$. 
Suppose that $\fod_t: \Bb \rightarrow \GL_n(\LDi):h \mapsto h + t \delta_h$ (with $t \in \Fi^\times$) is a family of first-order deformations satisfying
\begin{enumerate}[itemsep=1ex,topsep=1ex,label=\textup{(\roman*)},leftmargin=2em]
	\item \label{item:pingpongdeformations1} $\delta_{g} = 0 \iff g \in C$, and
	\item \label{item:pingpongdeformations2} $a \im(\delta_h) \cap \ker(\delta_{h'}) = \{0\}$ for every $a \in \Aa \setminus C$, $h, h' \in \Bb \setminus C$. 
\end{enumerate}
Then there exists $N \in \R$ such that
\[
\langle \Aa, \fod_t(\Bb) \rangle \cong \Aa \ast_C \fod_t(\Bb) \cong \Aa \ast_{C} \Bb
\]
when $|t| \geq N$.
\end{theorem}

In practice, checking the conditions $a \im(\delta_h) \cap \ker(\delta_{h'}) = \{0\}$ can be difficult, but luckily many are superfluous. 
For example, one can prove that $\ker(\delta_h) = \ker(\delta_{h^t})$ when $\gcd(o(h),t)=1$. 
Building on \Cref{ex:elementarydeformations}, we will lay out in \Cref{subsec:Bovdimaps} a way to construct a first-order deformation $\fod_t$ to which \Cref{thm:pingpongdeformations} applies, in the case where $\Aa$ and $\Bb$ are finite subgroups of the unit group of a group ring.

Note that the conditions from \Cref{thm:pingpongdeformations} imply that 
\(
\Aa \cap \Fi^\times = \Bb \cap \Fi^\times \subset C.
\)
Indeed, if $a \in \Aa \cap \Fi^\times \setminus C$ and $h \in \Bb \setminus C$, then the second condition implies that $\im(\delta_h) \cap \ker(\delta_h) = \{0\}$. 
Since $\im(\delta_h) \subset \ker(\delta_h)$, it would follow that $\delta_h =0$, a contradiction. 
Similarly, if $h \in \Bb \cap \Fi^\times$, then it is a general property of first-order deformations that $\delta_h = 0$, hence $h \in C$ by the first condition.
\medskip

When $\LDi$ is a division algebra over a local field $\LFi$ of characteristic $0$, and $\Aa$, $\Bb$ are subgroups of $\GL_{\LDi^n}(\LFi)$, \Cref{thm:pingpongdenseSLnrationalpoints} implies on the nose that there is a free amalgamated product of the form $\Aa \ast_C \Bb^{\gamma}$ in $\GL_{\LDi^n}(\LFi)$. 
The point of \Cref{thm:pingpongdeformations} is to express such conjugate of $\Bb$ more explicitly. 

As bibliographical context, note that the related results in \cite{Passman04, GoncalvesPassman06} can be reformulated as follows: in the special case where $\Bb$ admits a conjugate $\Bb^x$ for which $\Aa \cap \Bb^x$ is trivial and $\langle \Aa, \Bb^x \rangle$ is finite, then there is a first-order deformation $\Bb'$ of $\Bb$ which plays ping-pong with $\Aa$. 
In this situation, in \cite{Passman04} is actually constructed a free product of the form $\langle A,B^x \rangle \ast \Z$. 

\bigskip

The next lemma will serve to replace \Cref{lem:proximaldynamics} in the proof of \Cref{thm:pingpongdeformations}. 

\begin{lemma} \label{lem:unipotentdynamics}
Let $h \in \GL_V(\LFi)$, let $n$ be a non-zero nilpotent transformation in $\End_\LDi(V)$, and set $h_t = h + tn$ for $t \in \LFi$. 
Let $C$ be a compact subset of $\P(V) \setminus \ker(n)$, and let $U$ be a neighborhood of $\im n$ in $\P(V)$. 
There exists $N > 0$ such that if $|t| \geq N$, then $h_t \in \GL_V(K)$ and $h_t C \subset U$. 
\end{lemma}
\begin{proof}
Since $\Nrd(h) \neq 0$, the polynomial $\Nrd(h+Xn) \in \LFi[X]$ is not zero, hence has only finitely many roots. 
For $|t|$ strictly larger than the maximum absolute value $N_0$ of these roots, $h+tn \in \GL_V(K)$. 

Let now $p \in \P(V) \setminus \ker n$, and let $v$ represent $p$ in $V$. 
Since $h_t(v) = h(v) + tn(v)$, it follows that $h_t p$ converges to $n p \in \P(V)$ as $|t| \to \infty$. 
So for each $p \in C$, there exists $N_p > 0$ and a neighborhood $U_p$ of $p$ in $\P(V)$ such that $h_t(U_p) \subset U$ if $|t| \geq N_p$. 
By the compactness of $C$, there is a finite collection $U_{p_1}, \dots, U_{p_r}$ covering $C$. 
Setting $N = \max\{N_0, N_{p_1}, \dots, N_{p_r}\}$, we see that $h_t C \subset U$ when $|t| \geq N$, as claimed. 
\end{proof}

\begin{remark}
When $\LFi$ is a non-archimedean local field, care has to be taken that the condition $|t| \geq N$ is not preserved by addition. 
In other words, if the conclusion of the lemma holds for $h_t$, one cannot deduce that $gC \subset U$ for every $g \in \langle h_t \rangle$. 
This fails already for $h = 1$, in which case the subgroup $\langle 1 + tn \rangle$ accumulates at 1. 

This mistake was made in \cite{GoncalvesPassman06}. 
Some of the results in this paper, namely \cite[Theorems 2.3, 2.6, and 2.7]{GoncalvesPassman06}, therefore only hold over archimedean local fields. 
\end{remark}

\begin{proof}[Proof of \Cref{thm:pingpongdeformations}]
As in the proof of \Cref{thm:pingpongdense}, it will be convenient to write $\Hnt = H \setminus C$ and $\Ant = A \setminus C$.

Let $W$ denote the union of the proper subspaces $\im(\delta_h)$ for $h \in \Hnt$, and let $W'$ denote the union of the proper subspaces $\ker(\delta_h)$ for $h \in \Hnt$, all viewed in $\P(V)$. 
Note right away that $W \subset W'$. 
By the first assumption, if $c \in C$ then $c \im(\delta_h) = \im(c \delta_h) = \im(\delta_{ch})$, so $W$ is stable under $C$. 
Moreover, $aW \cap W \subset aW \cap W' = \emptyset$ for every $a \in \Ant$. 
Indeed, if this last intersection were non-empty, then so would be
\(
a \im(\delta_h) \cap \ker(\delta_{h'})
\)
for some $h, h' \in \Hnt$, contradicting the second assumption. 

In consequence, we can construct a compact neighborhood $P$ of $W$ with the properties that $P$ is stable under $C$, and $aP \cap (P \cup W') = \emptyset$ for every $a \in \Ant$. 
For instance, start with a neighborhood of $W$ whose translates under $\Ant$ are disjoint from $W'$, remove from it a sufficiently small open neighborhood of the union of its translates under $\Ant$, then intersect the result with its translates under $C$. 
Set $Q = \Ant \cdot P = \bigcup_{a \in \Ant} aP$; by construction, $Q$ is a compact set disjoint from $P$, and from $\ker(\delta_h)$ for every $h \in \Hnt$. 

In order to conclude the proof, it remains to verify the conditions of \Cref{lem:pingpongamalgam}. 
We already arranged for $\Ant \cdot P \subset Q$ and $C \cdot P = P$.
The fact that $C \cdot Q = Q$ is an obvious consequence of $C \cdot \Ant = \Ant$. 
Lastly, since $Q$ is disjoint from $\ker(\delta_h)$, \Cref{lem:unipotentdynamics} yields for each choice of $h \in \Hnt$ a positive number $N_h$ such that $\fod_t(h) = h + t \delta_h$ sends $Q$ into $P$ when $|t| \geq N_h$. 
Set $N = \max_{h \in \Hnt} N_h$ and pick $|t| \geq N$, so that $\fod_t(h) \cdot Q \subset P$ for every $h \in \Hnt$. 
An application of \Cref{lem:pingpongamalgam} (to $A$ and $\fod_t(H)$ with the sets $P$ and $Q$) finally shows that when $|t| \geq N$, $\langle A, \im (\fod_t) \rangle$ is the free product of $A$ and $\fod_t(H)$ amalgamated along their intersection $C$.
\end{proof}

\section{The embedding condition for group rings} \label{sec:faithfulembedding}

In \Cref{sec:pingpongsemisimplealgebra} we established, subject to the appropriate embedding condition, the existence of ping-pong partners for finite subgroups of the unit group of an order in a finite-dimensional semisimple algebra $\Alg$. 
We also proposed, via the first-order deformations introduced in \Cref{def:firstorderdeformation}, a more constructive method to obtain free products between such finite subgroups. 

The next two sections focus on the case where $\Alg = \Fi G$ is the group algebra of a finite group $G$ and $\Gamma = \U(RG)$ is the group of units of $\Alg$ over some order $R$ in $\Fi$. 
To this choice of $\Alg$ and $\Gamma$ the results of \Cref{sec:pingpongreductivegroups,sec:pingpongsemisimplealgebra} are more readily applicable, thanks to the fact that the simple factors of $\Fi G$ are the images $\rho(\Fi G)$ for $\rho$ ranging over the irreducible representations of $G$. 
\smallskip

For now, our main task is to address \Cref{que:existencefaithfulembeddings}.\ref{que:existencefaithfulembeddings1}, that is, to determine when a subgroup $H \leq G$ embeds in an appropriate simple quotient of $\Fi G$. 
This will be the matter of \Cref{subsec:faithfulembedding}; 
\Cref{thm:almostfaithfulembedding} shows that this happens if $G$ is not a Dedekind group and if $H$ admits a faithful irreducible representation (e.g.\ any subgroup $H$ whose Sylow subgroups have cyclic center). 
Consequently, for such a pair $H \leq G$ we obtain in \Cref{cor:existenceamalgaminorder,cor:existenceamalgaminordersharp} the existence inside $\U(RG)$ of a non-trivial free amalgamated product $H \ast_{C} H$.
\Cref{subsec:centerpreservingrepresentations,subsec:notFrobeniuscomplement} consist of technical prerequisites needed for the proof of \Cref{thm:almostfaithfulembedding}, contained in \Cref{subsec:proofthmalmostfaithfulembedding}; but first, we review some families of groups that will play a role in \Cref{subsec:faithfulembedding}. 

\smallskip

Unless specified otherwise, in this entire section $H$ is a subgroup of the finite group $G$, and $\Fi$ is a field whose characteristic does not divide $|G|$.

\subsection{Recollections on Dedekind groups and Frobenius complements} \label{subsec:reviewDedekindFrobenius}

\subsubsection*{Dedekind groups}
A group is called \emph{Dedekind} if all of its subgroups are normal. 
The Baer--Dedekind classification theorem \cite{Dedekind97,Baer33} states that a finite Dedekind group is either abelian, or isomorphic to $Q_8 \times C_2^n \times A$ with $n \in \N$ and $A$ an abelian group of odd order.\footnote{The statement for finite groups is in fact Dedekind's original contribution; Baer extended it to infinite groups. }

\subsubsection*{Frobenius groups}
A finite group $G$ is a \emph{Frobenius group} if it is the semidirect product $G = H \ltimes K$ of two subgroups $H$ and $K$ such that every element in $G \setminus K$ is conjugate to an element of $H$. 
The normal subgroup $K$ is called the \emph{Frobenius kernel of $G$}, and $H$ is called a \emph{Frobenius complement of $G$ for $K$}. 
More generally, any group $K$ (resp.\ $H$) appearing as the Frobenius kernel (resp.\ a Frobenius complement) of some Frobenius group $G$ is called a \emph{Frobenius kernel} (resp.\ a \emph{Frobenius complement}). 

There are many equivalent definitions of a Frobenius group. 
For instance, a finite group $G$ is a {Frobenius group} if and only if it admits a non-trivial proper subgroup $H$ such that $H \cap H^{g} = 1$ for every $g \in G \setminus H$; if so, $K = \big(G \setminus \bigcup_{g \in G} H^g \big) \cup \{1\}$ is a Frobenius kernel of $G$ with Frobenius complement $H$. 
We refer the reader to \cite{Isaacs06,Serre22} for a more detailed overview; for the main structural results concerning Frobenius complements, we refer the reader to \cite[Section 11.4]{JespersdelRio16} and \cite[Section 18]{Passman68}. 
In the scope of this paper, Frobenius complements mostly play a role through their characterization in terms of fixed-point-free representations. 

\subsubsection*{Fixed-point-free groups and representations}
A $\Fi$-representation $\rho$ of a group $G$ is called \emph{fixed-point-free} if for every $g \in G$, either $\rho(g)$ is the identity or $\rho(g)$ has no non-zero fixed vectors. 

For $g \in G$, set $\wt{g} = \sum_{i=1}^{o(g)} g^i \in \Fi G$. 
When the characteristic of $\Fi$ does not divide $|G|$, it is straightforward to see that for a primitive central idempotent $e \in \PCI(\Fi G)$, the following are equivalent: 
\begin{equation} \begin{gathered} \label{eq:fpfviaidempotent}
\text{the corresponding irreducible representation $\pi_e$ of $G$ is fixed-point-free} \\
\iff \text{$\wt{g}e = 0$ for every $g \in G \setminus \ker \pi_e$} 
\iff \text{$\wt{g}e$ is central for every $g \in G$}. 
\end{gathered} \end{equation}
Indeed, the image of the linear transformation $\wt{g}e \in \Fi Ge$ is the set of vectors fixed by $\pi_e(g)$. 
Thus, if $\pi_e$ is fixed-point-free, then either $\pi_e(g) = 1$ or $\wt{g}e = 0$. 
It is obvious that both $\pi_e(g) = 1$ and $\wt{g}e = 0$ imply that $\wt{g}e$ is central. 
Lastly, if $\wt{g}e$ is central and $\pi_e(g)$ fixes a non-zero vector, then $\wt{g}e \neq 0$, hence the image of $\wt{g}e$ is the whole space (by Schur's lemma), implying that $g \in \ker \pi_e$. 
\smallskip

We say that the group $G$ itself is \emph{fixed-point-free} if it admits a faithful fixed-point-free representation over some field $\Fi$. 
If so, the characteristic of $\Fi$ is necessarily prime to the order of $G$, and such a representation exists as well over any field whose characteristic is prime to $|G|$ (see \cite[Section 6.5]{Serre22}). 
If $G$ is fixed-point-free, it also admits an irreducible faithful fixed-point-free representation (over any field whose characteristic does not divide $|G|$): any irreducible constituent of a faithful fixed-point-free representation, to wit. 
\smallskip

Zassenhaus \cite{Zassenhaus35} (see also \cite[Theorem 6.13]{Serre22}) proved that a group is fixed-point-free if and only if it is a Frobenius complement (of some Frobenius group). 
This fact will be used repeatedly in what follows, perhaps without mention.

\subsection{On the embedding condition for group rings} \label{subsec:faithfulembedding}
In this subsection, we wish to determine when a finite subgroup $H$ of $\U(RG)$ satisfies the almost embedding condition from \Cref{thm:pingpongdenseSLn,thm:pingpongdensesemisimplealgebra}, aiming to find a ping-pong partner for $H$. 

We start by stating a three-headed theorem, establishing the existence of a certain center-preserving irreducible representation of $G$ under increasingly weaker assumptions. 
The three parts will be proved together, after some technical preliminaries concerning faithful irreducible representations and Frobenius complements. 

\begin{definition}
As usual, a representation $\rho$ of $G$ is called \emph{faithful on $H$} if $H \cap \ker(\rho) = 1$. 
We will say that $\rho$ is \emph{center-preserving on $H$} if $H \cap \ZZ(\rho) = H \cap \ZZ(G)$. 
Here and elsewhere, $\ZZ(\rho)$ denotes the \emph{center of $\rho$}, that is, $\ZZ(\rho) = \rho^{-1}(\ZZ(\rho(G)))$. 
If a representation $\rho$ of $G$ is faithful on $G$ then it is of course center-preserving on $G$, but on proper subgroups the two notions are distinct. 
\end{definition}

\begin{theorem} \label{thm:almostfaithfulembedding}
Let $G$ be a non-abelian finite group. 
Let $\Fi$ be a field whose characteristic does not divide $|G|$, assumed to be a number field for the purpose of part \textup{(iii)}. 
Let $H$ be a subgroup of $G$, and suppose that $G$ admits an irreducible $\Fi$-representation which is center-preserving on $H$. 
If respectively
\begin{enumerate}[itemsep=1ex,topsep=1ex,label=\textup{(\roman*)},leftmargin=2em]
	\item \label{item:notFrobenius} 
$G$ is not a Dedekind group,
	\item \label{item:notdivision} 
$\Fi G$ is not a product of division algebras,
	\item \label{item:notquaternion} 
$\Fi$ is not totally real, or $G$ is not isomorphic to $Q_8 \times C_2^n$ for any $n \in \N$,
\end{enumerate}
then there exists an irreducible $\Fi$-representation $\rho$ of $G$ such that $\indx{H\cap \ZZ(\rho)}{H \cap \ZZ(G)} \leq 2$ and satisfying respectively
\begin{enumerate}[itemsep=1ex,topsep=1ex,label=\textup{(\roman*)},leftmargin=2em]
	\item $\rho(G)$ is {not} a Frobenius complement,
	\item $\rho(\Fi G)$ is not a division algebra,
	\item $\rho(\Fi G)$ is neither a field nor a totally definite quaternion algebra. 
\end{enumerate}

If moreover, 
\begin{enumerate}[itemsep=1ex,topsep=1ex,label=\textup{(\roman*)},leftmargin=3em]
\customitem[condition]{cond:NQ}{\textup{(NQ)}}
there does not exist an irreducible $F$-representation $\sigma$ of $G$ such that $\sigma(G)$ is a non-abelian Frobenius complement having a generalized quaternion group $Q_{2^n}$ as $2$-Sylow subgroup, and for which $\sigma(H) \cap Q_{2^n}$ has an element of order $4$ generating a normal subgroup,
\end{enumerate}
then in fact $\rho$ can be taken to be center-preserving on $H$. 
\end{theorem}

For the sake of brevity, we will refer to the cumbersome condition on the pair $H \leq G$ ensuring the equality $H \cap \ZZ(\rho) = H \cap \ZZ(G)$ in \Cref{thm:almostfaithfulembedding}, as \Cref{cond:NQ}. 
It holds in particular if: no image of $G$ under an irreducible representation is a non-abelian Frobenius complement, or $G$ does not have a $2$-Sylow subgroup mapping onto $Q_8$, or $H$ has no element of order $4$, or $2$ does not divide $|H \cap \ZZ(G)|$, or $2$ does not divide $|H / H \cap \ZZ(G)|$. 

\begin{remark}
If $G$ is a Dedekind group (resp.\ $\Fi G$ is a product of division algebras, $\Fi$ is totally real and $G \cong Q_8 \times C_2^n$), then for every $\rho \in \Irr_\Fi(G)$, the image $\rho(G)$ is a Frobenius complement (resp.\ $\rho(\Fi G)$ is a division algebra, $\rho(\Fi G)$ is either a field or a totally definite quaternion algebra), hence no irreducible representation of $G$ can possibly satisfy Conclusion \ref{item:notFrobenius} (resp.\ \ref{item:notdivision}, \ref{item:notquaternion}). 
In this sense, the assumptions made on $G$ in \Cref{thm:almostfaithfulembedding} are sharp. 

Conversely, if $\rho(G)$ is a Frobenius complement for every $\rho \in \Irr_\Fi(G)$, then the particular case $H = 1$ in \Cref{thm:almostfaithfulembedding}.\ref{item:notFrobenius} implies that $G$ is a Dedekind group. 
We will state and prove this characterization of Dedekind groups separately in \Cref{thm:DedekindiffimagesFrobeniuscomplement}, before proving \Cref{thm:almostfaithfulembedding}. 

Note also that some condition is necessary to ensure that $H \cap \ZZ(\rho) = H \cap \ZZ(G)$, even when $G$ itself has a faithful irreducible representation. 
Indeed, the irreducible representations of the generalized quaternion group $Q_{2^n}$ over a totally real number field $\Fi$ fall into two categories: those inflated from the dihedral central quotient $D_{2^{n-1}} \cong Q_{2^n} / \ZZ(Q_{2^n})$, and the quaternionic representations. 
Under a quaternionic representation, $\rho(\Fi G)$ is evidently a totally definite quaternion algebra, whereas the center of any representation inflated from $D_{2^{n-1}}$ contains the preimage of $\ZZ(D_{2^{n-1}})$, which is cyclic of order $4$ and properly contains $\ZZ(Q_{2^n})$. 
\end{remark}

\begin{remark} \label{rem:Frobeniuscomplement}
Note that if $\rho$ is a representation such that $\rho(G)$ is not a Frobenius complement (for instance, if $\rho$ is afforded by \Cref{thm:almostfaithfulembedding}), then $\rho(\Fi G)$ is certainly not a division algebra.
Indeed, it is easy to see that a finite subgroup of a division $\Fi$-algebra is {fixed-point-free}. 
(Recall that a group is fixed-point-free if and only if it is a Frobenius complement  \cite{Zassenhaus35}.) 

The condition that $\Fi G$ be not a product of division algebras can also be reformulated in terms of $G$. 
Indeed, by \cite[Theorem 3.5]{Sehgal75}, $\Q G$ is a product of division algebras if and only if $G$ is abelian or isomorphic to $Q_8 \times C_2^n \times A$ with $n \in \N$ and $A$ an abelian group such that $|A|$ and the order of $2$ in $(\Z / |A|\Z)^\times$ are both odd. 
It follows that $\Fi G$ is a product of division algebras if and only if $G$ is of this form, and when $G$ is not abelian, $\Fi$ is not a splitting field of Hamilton's quaternion algebra $\qa{-1}{-1}{\Q}$. 

In the frame of this paper, the importance of having an irreducible representation $\rho$ for which $\rho(\Fi G)$ is neither a field nor a totally definite quaternion algebra stems from \Cref{thm:pingpongdenseSLn,thm:pingpongdensesemisimplealgebra}, which will be used to construct ping-pong partners. 
The relevance of the stronger property that $\rho(G)$ be not a Frobenius complement lies in its use to obtain a \emph{bicyclic} ping-pong partner; the details can be found in \Cref{subsec:bicyclicgenericallyfree}. 
For now, combining \Cref{thm:almostfaithfulembedding}.\ref{item:notquaternion} and \Cref{thm:pingpongdensesemisimplealgebra}, we record the following interesting results.
\end{remark}

\begin{corollary}\label{cor:existenceamalgaminorder}
Let $\Fi$ be a number field and $R$ be its ring of integers. 
Let $H \leq G$ be finite groups, and suppose that $G$ admits an irreducible $\Fi$-representation which is center-preserving on $H$. 
Assume that either $\Fi$ is not totally real, or $G \ncong Q_8 \times C_2^n $ for any $n \in \N$.
Then there exists a subgroup $C \leq H$ with $\indx{C}{H \cap \ZZ (G)} \leq 2$, and an element $\gamma \in \U(RG)$ of infinite order such that $\langle \gamma, H \rangle$ and $\langle \gamma, C \rangle \ast_C H$ are canonically isomorphic. 
\end{corollary}

\begin{proof}
Let $\rho$ be a representation afforded by \Cref{thm:almostfaithfulembedding}.\ref{item:notquaternion} and set $C = H \cap \ZZ(\rho)$. 
By construction, the simple $\Fi$-algebra $\rho(\Fi G)$ is not a field nor a totally definite quaternion algebra. 

Because the Wedderburn decomposition $\Fi G = \bigoplus_{\sigma \in \Irr_\Fi(G)} \sigma(\Fi G)$ is defined over $\Fi$, the intersection $RG \cap \rho(\Fi G)$ is an order in $\rho(\Fi G)$, and $\Gamma = \U(R G \cap \rho(\Fi G))$ is an arithmetic subgroup of $\GL_{\rho(\Fi G)}(\Fi) = \U(\rho(\Fi G))$. 
In particular, the image of $\Gamma$ in the adjoint group $\PGL_{\rho(\Fi G)}$ is Zariski-dense, and we can apply \Cref{thm:pingpongdensesemisimplealgebra} (see also \Cref{rem:partialpingpongsemisimplealgebra}) with $\Alg = \rho(\Fi G)$ and $A = B = \rho(H)$ to deduce the existence of an element $\gamma \in \Gamma$ of infinite order for which the canonical map
\[
\langle {\gamma}, \rho(C) \rangle \ast_{\rho(C)} \rho(H) \to \langle {\gamma}, \rho(H) \rangle
\]
is an isomorphism. 
Note that by construction, $C$ is the preimage in $H$ of $\rho(C)$ under $\rho$, and $\gamma = \rho(\gamma)$ commutes with $C$. 
An application of \Cref{lem:preimageamalgam} to the surjective map 
\[
\rho: \langle \gamma, H \rangle \to \langle {\gamma}, \rho(C) \rangle \ast_{\rho(C)} \rho(H)
\]
then shows that $\langle {\gamma}, H \rangle$ and $\langle {\gamma}, C \rangle \ast_C H$ are canonically isomorphic. 
\end{proof}

When \Cref{cond:NQ} is satisfied, the proof of \Cref{cor:existenceamalgaminorder} can be cut short, for one can then apply \Cref{thm:pingpongdensesemisimplealgebra} to $\Alg = \Fi G$ straight away. 
We record this more precise result next. 

\begin{corollary}\label{cor:existenceamalgaminordersharp}
Let $\Fi$, $R$, $G$, and $H$ be as in \Cref{cor:existenceamalgaminorder}, and assume in addition that \Cref{cond:NQ} holds. 
Set $C = H \cap \ZZ (G)$.
Then the set 
\[
S = \set{\gamma \in \U(RG) \textup{ of infinite order}}{\langle \gamma, H \rangle \cong \langle \gamma, C \rangle \ast_C H \textup{ canonically}}
\]
is dense in $\U(RG)$ for the join of the Zariski and profinite topologies. 
\end{corollary}
\begin{proof}
We can take the representation $\rho$ afforded by \Cref{thm:almostfaithfulembedding}.\ref{item:notquaternion} to satisfy $H \cap \ZZ(\rho) = H \cap \ZZ(G)$, because $G$, $H$ satisfy \ref{cond:NQ}. 
This means that $H$ almost embeds in $\PGL_{\rho(\Fi G)}$, and as $\rho(\Fi G)$ is not a field nor a totally definite quaternion algebra, the statement follows at once from \Cref{thm:pingpongdensesemisimplealgebra}. 
\end{proof}

Note again that if $\Fi$ is totally real and $G \cong Q_8 \times C_2^n$, then every simple factor of $\Fi G$ is either a field or a totally definite quaternion algebra, implying that $\U(RG) / \ZZ(\U(RG))$ is a finite group. 
The attentive reader will notice that the converse must also hold: if $\U(RG) / \ZZ(\U(RG))$ is finite, then either $G$ is abelian or $\Fi$ is totally real and $G \cong Q_8 \times C_2^n$ for some $n \in \N$. 
Indeed, $\U(RG) / \ZZ(\U(RG))$ is commensurable with $\SL_1(RG)$ by \cite[Proposition 5.5.1]{JespersdelRio16}, and the same criterion for the finiteness of $\SL_1(RG)$ was established by Kleinert \cite{Kleinert00} (see also \cite[Corollary 5.5.7]{JespersdelRio16}). 
Thus, in this situation, there is no hope to observe any non-trivial free amalgamated products inside $\U(RG)$. 
\medskip

For $H$ a cyclic group of prime order $p$, the existence of a free amalgamated product $H \ast_{H \cap \ZZ(G)} H$ inside $\U(\Z G)$ was first obtained by Gon\c{c}alves and Passman \cite{GoncalvesPassman04}. 
They proved furthermore that $C_p \ast C_\infty$ is a subgroup of $\U(\Z G)$ if and only if $G$ has a non-central element of order $p$. This last fact also follows from our results thanks to the positive solution to the Kimmerle problem for prime order elements \cite[Corollary 5.2.]{KimmerleMargolis17}, which states that any element of order $p$ in $\U(\Z G)$ can be conjugated (inside some larger algebra) to one in $G$. 

\begin{remark}
Knowing from the start that $G$ has an irreducible representation which is center-preserving on two subgroups $A$ and $B$ at the same time, one can establish the existence of $A \ast_C B$ in $\U(RG)$ using the same argument. 
Having this in mind, it would be interesting to know whether there is a variant of \Cref{thm:almostfaithfulembedding} for two subgroups simultaneously. 
\end{remark}

\begin{remark} \label{rem:conditionNQ}
As convenient as it was for regrouping reasonable sufficient conditions, \Cref{cond:NQ} is not strictly necessary to obtain the equality $H \cap \ZZ(\rho) = H \cap \ZZ(G)$ in \Cref{thm:almostfaithfulembedding}. 
The details of the proof of \Cref{thm:almostfaithfulembedding} do describe when exactly one can arrange that $H \cap \ZZ(\rho)$ and $H \cap \ZZ(G)$ coincide. 

For the purpose of \Cref{cor:existenceamalgaminorder}, the case $\indx{C}{H\cap \ZZ(G)} = 2$ need only be considered if every irreducible representation $\sigma$ of $G$ which is center-preserving on $H$, is such that $\sigma(FG)$ is a field or a totally definite quaternion algebra. 
This characterization is however not very practical, as it pretty much relies on the knowledge of all representations which are center-preserving on $H$. 
Nevertheless, below \Cref{thm:almostfaithfulembedding} are indicated more user-friendly conditions under which \ref{cond:NQ} holds, hence the representation $\rho$ from \Cref{thm:almostfaithfulembedding} can be taken to satisfy $H \cap \ZZ(\rho)= H \cap \ZZ(G)$, and in turn, the conclusion of \Cref{cor:existenceamalgaminordersharp} holds. 
\end{remark}

\begin{remark}
\Cref{thm:almostfaithfulembedding} (and its corollaries) can also be applied to subgroups $H$ of conjugates of $G$ in $\U(\Fi G)$. 
This extension to conjugates of $G$ is reminiscent of the Zassenhaus conjectures. 
The strongest Zassenhaus conjecture asserts that for a finite group $G$, any subgroup $H \leq \U(\Z G)$ is conjugated under $\U(\Q G)$ to a subgroup of $\pm G$. 
Counterexamples to this conjecture for $H$ not cyclic were obtained by Roggenkamp and Scott \cite{Scott92}, and for $H$ cyclic by Eisele and Margolis \cite{EiseleMargolis18}. 
However, the strongest Zassenhaus conjecture does hold for some classes of groups: for instance, if $G$ is {nilpotent} and $H$ is any subgroup \cite{Weiss88,Weiss91}, or for $G$ {cyclic-by-abelian} and $H$ {cyclic} \cite{CaicedoMargolisdelRio13}. 
We refer to \cite{MargolisdelRio19} for an overview of this topic. 

As a consequence of these facts, when $\Fi = \Q$ and $G$ is nilpotent, the statements of \Cref{thm:almostfaithfulembedding} hold for all finite subgroups $H$ of $\U(\Fi G)$, not just for those contained in a group basis.
\end{remark}
\bigbreak

Recall that for every primitive central idempotent $e \in \PCI(\Fi G)$, we denote by
\[
\pi_e : \Fi G \to \Fi Ge \cong \Mat_{n_e}(D_e) : x \mapsto xe
\]
the projection onto the the associated simple quotient $\Fi Ge$ of $\Fi G$. 
Somewhat abusively, we will also use $\pi_e$ to denote its restriction $\U(\Fi G) \to \GL_{n_e}(D_e)$ between unit groups, as well as the corresponding irreducible $\Fi$-representation of $G$ on $D_e^{n_e}$. 
It will be convenient to switch back and forth between the language of primitive central idempotents of $\Fi G$, and of irreducible $\Fi$-representations of $G$. 
On the representation side, $\Irr_\Fi(G)$ denotes the set of irreducible $\Fi$-representations of $G$; so the two sets $\Irr_\Fi(G)$ and $\PCI(\Fi G)$ can naturally be identified. 

In the remainder of the article, we will make use of the following two sets, gathering the irreducible $\Fi$-representations of $G$ which are faithful, respectively center-preserving, on $H$:
\begin{equation} \label{eq:notationEmb}
\begin{aligned}
\Fir^G_\Fi(H) &= \{ e \in \PCI(\Fi G) \mid H \cap \ker(\pi_e)=1 \}, \\
\Cir^G_\Fi(H) &= \{ e \in \PCI(\Fi G) \mid H \cap \ZZ(\pi_e) \leq \ZZ(G)\}.
\end{aligned}
\end{equation}
When $G = H$, we will simply write $\Fir_\Fi(H)$ and $\Cir_{\Fi}(H)$; when the field $\Fi$ is clear from context, we will write $\Fir^{G}(H)$ and $\Cir^{G}(H)$. 

With this notation, \Cref{thm:almostfaithfulembedding}.\ref{item:notFrobenius} implies for example that if $\Cir_{\Fi}^G(H) \neq \emptyset$ and \Cref{cond:NQ} holds, there exists $e \in \Cir^G_\Fi(H)$ such that $\pi_e(G)$ is not a Frobenius complement.
\smallbreak

\begin{remark}
In \Cref{thm:almostfaithfulembedding}, one cannot replace the condition that $\Cir^G_{\Fi}(H) \neq \emptyset$ by the weaker condition that there exists a (not necessarily faithful) irreducible $\Fi$-representation $\sigma$ of $G$ such that $\ker(\sigma) \leq H \cap \ZZ(G)$. 
Indeed, for $G$ a group which satisfies \ref{cond:NQ} and has no image under an irreducible representation that is a non-abelian Frobenius complement, \Cref{thm:almostfaithfulembedding} tautologically states that $\Cir^G_{\Fi}(H) \neq \emptyset$. 

Pairs $H \leq G$ admitting such an irreducible representation $\sigma$, but no center-preserving one exist: for example, consider the group
\[
G = C_4 \rtimes D_8 = \langle a,b,c \mid c^4=a^4=b^2=1, b^{-1}ab=a^{-1}, a^{-1}ca=c^{-1} ,b^{-1}cb=c^{-1} \rangle,
\]
and take $H = G$. 
It can be shown that the irreducible representations $\rho$ of $G$ of degree at least $2$ are inflated either from $D_8$ or from the central product $C_4 \circ D_8$. 
In particular, these images $\rho(G)$ are not Frobenius complements. 
Those inflated from $C_4 \circ D_8$ have a central kernel, but none of the irreducible representations of $G$ are center-preserving! 
\end{remark}
\medbreak

Before starting the proof of \Cref{thm:almostfaithfulembedding}, we gather several needed results about faithful irreducible representations and Frobenius complements.

\subsection{Faithful irreducible representations and their center} \label{subsec:centerpreservingrepresentations}
The existence of faithful irreducible (complex) representations of finite groups has been intensively studied, see \cite[Section 2]{Szechtman16} for a brief survey of its history. 
To prove \Cref{thm:almostfaithfulembedding}, we will need to understand the existence of such representations over an arbitrary field $\Fi$ of characteristic not dividing $|G|$, and determine when they preserve the center. 
The following lemma summarizes some basic facts concerning the existence of faithful representations over $\Fi$. 

\begin{lemma} \label{lem:firdifferentfields}
Let $G$ be a finite group, $F \subseteq L$ be fields of characteristic not dividing $|G|$, and $H$ be a finite subgroup of $\U(FG)$. 
\begin{enumerate}[itemsep=1ex,topsep=1ex,label=\textup{(\roman*)},leftmargin=2em]
	\item If $e \in \PCI(FG)$, then $\ZZ (G)/\ker(\pi_e)$ is cyclic.
	\item If $\Fir^G_L(H) \neq \emptyset$ then $\Fir^G_F(H) \neq \emptyset$; if $\Cir^G_L(H) \neq \emptyset$ then $\Cir^G_F(H) \neq \emptyset$.
	\item If all Sylow subgroups of $G$ have a cyclic center, then $\Fir_{F}(G) \neq \emptyset$.
	\item If $G$ is nilpotent, then $\Fir_F(G) \neq \emptyset$ if and only if $\ZZ(G)$ is cyclic.
\end{enumerate}
\end{lemma}
\begin{proof}
Since $FGe$ is simple, $\ZZ(FGe)$ is a field. 
Thus, $\pi_e(\ZZ(G)) \subseteq \ZZ(FGe)^\times$ is a finite subgroup of the multiplicative group of a field, hence is cyclic. 
This proves (i). 
\smallbreak

Next, note that 
\[ \textstyle
\bigoplus_{e \in \PCI(LG)} LGe \cong LG \cong L \ot_F FG \cong \bigoplus_{f \in \PCI(FG)} \left(L \ot_F FGf \right).
\]
Pick $e \in \PCI(LG)$, and let $f \in \PCI(FG)$ be the idempotent for which the quotient map $\pi_e: LG \to LGe$ factors through the summand $L \ot_F FGf$. 
The kernel of the restriction to $G$ of $\pi_e$ then obviously contains that of $\pi_f$. 
Similarly, the preimage of the center of $LGe$ under $\pi_e$ contains that of $FGf$ under $\pi_f$. 
Hence if $e \in \Fir_L^G(H)$, resp.\ $\Cir_L^G(H)$, we deduce that $f \in \Fir_F^G(H)$, resp.\ $\Cir_F^G(H)$, proving (ii). 
\smallbreak

Part (iii) is well-known, and can be deduced from the classical result of Gasch\"utz \cite{Gaschutz54}, stating that $\Fir_{\Fi}(G) \neq \emptyset$ if and only (the abelian part of) the socle of $G$ is generated by a single conjugacy class. 
(Although Gasch\"utz proves this result over the complex numbers, the argument works identically over any field of characteristic not dividing $|G|$. 
For a short proof in the present generality, we also refer the reader to \cite{CapraceJanssensThilmany25}.) 

Let $\SocA(G)$ denote the abelian part of the socle of $G$. 
For $p$ prime divisor of $|\SocA(G)|$, let $S_p$ denote the $p$-Sylow subgroup of $\SocA(G)$, so that $\SocA(G) = \prod_{p \text{ prime}} S_p$. 
Pick $N_p$ a minimal normal subgroup of $G$ contained in $S_p$, and $G_p$ a $p$-Sylow of $G$ containing $S_p$. 
Because $N_p$ is normal, it intersects the center of $G_p$ non-trivially. 
By assumption, the center of $G_p$ contains a unique minimal non-trivial subgroup, namely the subgroup generated by any element $x_p$ of order $p$ in the cyclic group $\ZZ(G_p)$. 
Therefore $x_p \in N_p$, and since $N_p$ is a minimal normal subgroup, $\llangle x_p \rrangle_G = N_p$. 
It follows that any two minimal normal subgroups of $S_p$ coincide, i.e.\ $S_p = N_p$.
Let now $x = \prod_{p \text{ prime}} x_p$ be the product of the $x_p$ over all prime divisors of $|\SocA(G)|$. 
As they commute, $x_p \in \langle x \rangle$, hence $N_p = \llangle x_p \rrangle_G \leq \llangle x \rrangle_G$. 
This implies at once that $\SocA(G)$ coincides with $\llangle x \rrangle_G$, as desired. 
\smallbreak

Finally, for part (iv), the necessity of $\ZZ(G)$ being cyclic follows from (i). 
The sufficiency follows from (iii), since nilpotent groups are the direct product of their Sylow subgroups. 
\end{proof}

The existence of irreducible representations faithful of a group $G$ can be characterized in several ways. 
The most common criterion is due to Gasch\"utz, and states that $G$ has a faithful irreducible representation over some (any) field of characteristic not dividing $|G|$, if and only if the socle of $G$ is generated by a single conjugacy class. 
The literature contains multiple generalizations of this fact, but somewhat surprisingly, we could not find any results providing sufficient control on the center of the representation in order to apply \Cref{thm:pingpongdenseSLn}. 

With the aim to fill this gap, it is shown in \cite{CapraceJanssensThilmany25} that if $H$ admits a faithful irreducible $\Fi$-representation, then in fact $\Fir^G_F(H) \cap \Cir^G_F(H)$ is non-empty for any finite group $G$ containing $H$. 
This result will form the first step in proving \Cref{thm:almostfaithfulembedding}, and the remainder of the proof essentially amounts to upgrading this representation to fit our needs. 

\begin{theorem}[{\cite{CapraceJanssensThilmany25}}] \label{thm:faithfulcenterpreservingrepresentation}
Let $G$ be a finite group and let $\Fi$ be a field whose characteristic does not divide $|G|$. 
Let $H \leq G$ be a subgroup possessing a faithful irreducible $\Fi$-representation. 
Then $G$ has an irreducible $\Fi$-representation $\sigma$ whose restriction to $H$ is faithful, and such that $H \cap \ZZ(\sigma) \leq \ZZ(G)$. 
\end{theorem}

We refer the reader to \cite{CapraceJanssensThilmany25} for the proof of this theorem, as well as additional context and a reminder of the proof of Gasch\"utz' criterion. 
\medbreak

Let us immediately record a consequence of \Cref{thm:faithfulcenterpreservingrepresentation} concerning the existence of free products in $\U(\Fi G)$. 
This corollary is less precise than \Cref{cor:existenceamalgaminorder}, as it exhibits ping-pong partners that need not lie in a proper subring of $\Fi G$, but in return presents the advantage of avoiding the intricacies of the representation theory of $\Fi G$. 

\begin{corollary} \label{cor:existenceamalgamarbitraryfield}
Let $\Fi$ be a field of characteristic $0$. 
Let $G$ be a finite group, and let $H \leq G$ be a subgroup possessing a faithful irreducible $F$-representation. 
Set $C = H \cap \ZZ(G)$. 
The set of regular semisimple elements $\gamma \in \U(\Fi G)$ of infinite order with the property that the canonical map
\[
\langle \gamma, C \rangle \ast_{C} H \to \langle \gamma, H \rangle
\]
is an isomorphism, is Zariski-dense in $\U(FG)$. 
\end{corollary}
\begin{proof}
Let $\bG = \GL_{\Q G}$.
Recall that $\bG$ is a reductive group, and that the simple factors of $\Ad \bG$, each one of the form $\PGL_{D^n}$ for $D$ some finite-dimensional division algebra over $\Q$, are in natural correspondence with the irreducible $\Q$-representations of $G$. 

As the field $\Fi$ contains $\Q$, it suffices to prove the corollary inside $\U(\Q G)$. 
Indeed, $\U(\Q G)$ is Zariski-dense in $\bG$, hence also in the extension of scalars of $\bG$ to $\Fi$ (whose $\Fi$-points form the group $\U(\Fi G)$). 

Since $H$ possesses a faithful irreducible $\Fi$-representation, it also has a faithful $\Q$-representation by \Cref{lem:firdifferentfields}. 
\Cref{thm:faithfulcenterpreservingrepresentation} then produces a representation $\rho \in \Cir_\Q^G(H)$.\footnote{\Cref{thm:faithfulcenterpreservingrepresentation} allows us to arrange that $\rho$ be also faithful on $H$, although this will not be used here.} 
In other words, the kernel $C$ of the projection from $H$ to the simple adjoint factor $\Ad \bG_\rho = \PGL_{D_\rho^{n_\rho}}$ corresponding to the representation $\rho$, is contained in $\ZZ(G)$, hence equals $H \cap \ZZ(G)$. 
The corollary is now an immediate consequence of \Cref{thm:pingpongdenseSLnrationalpoints} applied to $\bG$, the collection of subgroups $\{(H,H)\}$, and the simple quotient $\Ad \bG_\rho$. 
\end{proof}
\medbreak

\subsection{Representations whose images are (not) Frobenius complements} \label{subsec:notFrobeniuscomplement}
In order to prove \Cref{thm:almostfaithfulembedding}, we need to understand how to avoid representations whose images are Frobenius complements. 

We start by investigating groups which are themselves Frobenius complements. 
Recall that by the Frobenius--Thompson--Zassenhaus Theorem \cite[Theorem 11.4.5]{JespersdelRio16}, any Frobenius complement $B$ has the following restrictions on its Sylow subgroups: 
\begin{enumerate}[label=$\bullet$,leftmargin=2em]
	\item Every odd‐prime Sylow subgroup of $B$ is cyclic.
	\item The Sylow $2$-subgroups of $B$ are either cyclic or a generalized quaternion group 
\begin{equation}\label{eq:presentationquaternion}
Q_{4m} = \langle a,b \mid a^{2m} = 1, b^2= a^{m}, b^{-1}ab = a^{-1} \rangle, \quad \text{for $m = 2^{n-2}$ and $n \geq 2$. }
\end{equation}
\end{enumerate}
Note in addition that when $|B|$ is even, $a^m$ is the unique element of order $2$ in $B$ (see \cite[Theorem 11.4.5]{JespersdelRio16}), and is therefore central in $B$. 
The shape of the Sylow $2$-subgroup of $B$ turns out to play an important role, as illustrated by the following technical lemma. 

\begin{lemma}\label{lem:centralquotientFrobeniuscomplement}
Let $B$ be a non-abelian Frobenius complement. 
There exists a cyclic $p$-group $N \leq \ZZ(B)$, with $p$ prime, such that the following properties hold:
\begin{enumerate}[label=$\bullet$,leftmargin=2em]
	\item $N \cong C_2$ if some Sylow $2$-subgroup of $B$ is non-abelian,
	\item $B/N$ is {not} a Frobenius complement, 
	\item $B/N$ is abelian if and only if $B$ is isomorphic to $Q_8 \times C_m$ with $m$ odd.
\end{enumerate}
If moreover $B \ncong Q_8 \times C_m$ with $m$ odd, then $B/N$ has a faithful irreducible representation, $\indx{\pi^{-1}(\ZZ(B/N))}{\ZZ(B)} \leq 2$ for $\pi$ the quotient map $B \to B/N$, and the following are equivalent:
\begin{enumerate}[itemsep=1ex,topsep=1ex,label=\textup{(\roman*)},leftmargin=2em]
	\item \label{item:centralquotientFrobeniuscomplement1} 
    $\pi^{-1}(\ZZ(B/N)) = \ZZ(B)$,
	\item \label{item:centralquotientFrobeniuscomplement2} 
    $B$ does {not} have a Sylow $2$-subgroup isomorphic to $Q_{2^n}$ ($n \geq 3$) such that $\langle a^{2^{n-3}} \rangle$ is normal in $B$,
	\item \label{item:centralquotientFrobeniuscomplement3} 
    $B$ is a Z-group, or $\Fit(B)_2$, the Sylow $2$-subgroup of the Fitting subgroup of $B$, is isomorphic to $Q_8$ and $3$ divides $\indx{B}{\rC_B(\Fit(B)_2)}$. 
\end{enumerate}
If $B \ncong Q_8 \times C_m$ with $m$ odd, and the equivalent conditions \labelcref{item:centralquotientFrobeniuscomplement1,item:centralquotientFrobeniuscomplement2,item:centralquotientFrobeniuscomplement3} do not hold, then $\pi^{-1}(\ZZ(B/N)) = \langle \ZZ(B), a^{2^{n-3}} \rangle$ with $a$ as in \eqref{eq:presentationquaternion} above. 
\end{lemma}

Note that in the last assertion, the subgroup $\langle \ZZ(B), a^{2^{n-3}} \rangle$ does not depend on the choice of the Sylow $2$-subgroup, as $\langle a^{2^{n-3}} \rangle$ lies in the intersection of all of them when it is normal in $B$. 

\begin{proof}
We will use the notation $x_p$ (resp. $x_{p'}$) to denote the $p$-part (resp $p'$-part) of an element $x \in B$, with $p$ some prime. 
We distinguish two cases straight away, based on the isomorphism type of the Sylow $2$-subgroups of $B$. 
\medskip

\noindent \underline{{Case A:}} All Sylow subgroups of $B$ are cyclic. 
\smallskip

Such groups are called $Z$-groups, and it was proven by Zassenhaus that they are of the form
\begin{equation*} \label{eq:definitionZ-group}
B \cong C_m \rtimes_r C_n
:=\langle a,b\mid a^m=1,\;b^n=1,\;b^{-1}ab=a^r\rangle,
\end{equation*}
with $r^n\equiv1\mod m$ and $\gcd(m,n)=\gcd(m,r-1)=1$ (see \cite[Theorem 18.2]{Passman68}). 
The presentation shows that Sylow $p$-subgroups are normal if and only if $p \mid m$. 
Furthermore, 
\begin{equation}\label{eq:Z-groupFrobenius}
C_m \rtimes_r C_n \text{ is a Frobenius complement} \quad \iff \quad r^{n/\rad(n)} \equiv 1 \mod m, 
\end{equation}
where $\rad(n)$ is the product of all prime divisors of $n$. 
In other words, $C_m \rtimes_r C_n$ is a Frobenius complement if and only if all elements of prime order commute (cf.\ the proof of \cite[Theorem 18.2]{Passman68}); or yet reformulated, if and only if $r^{n/p} \equiv 1\mod m$ for all primes $p$ dividing $n$ (see \cite[Theorem 11.4.9]{JespersdelRio16}).

First note that the center of $B$ is
\begin{equation}\label{eq:centerZ-group}
\ZZ(B)=\langle b^d\rangle, \text{ for $d$ the order of $r$ in $(\Z/m\Z)^\times$}.
\end{equation}
Indeed, let $a^i b^j$ with $0 \leq i < m$ and $0 \leq j < n$, be an arbitrary element in $B$. 
Since $b^{-1}(a^i b^j) b = a^{i r} b^{j}$, in order for $a^i b^j$ to be central it must be $ir \equiv i \mod m$. 
Now, $\gcd(m,r-1) = 1$ implies that $i \equiv 0 \mod m$, hence $a^i b^j$ is a power of $b$. 
Since $b^{-j} a b^{j} = a^{r^j}$, the elements $a$ and $b^j$ commute if and only if $r^j \equiv 1 \mod m$. 
In other words, $b^j$ is central if and only if $d \mid j$. 
\smallskip

Next, pick $p$ a prime divisor of $d$ and write $o(b^d) = n/d = s p^{\ell}$ with $\gcd(p,s)=1$, so that the element $b^{sd} = b^{n/p^{\ell}}$ is central of order $p^{\ell}$. 
From \eqref{eq:centerZ-group} it follows that 
\(
B/\langle b^{n/p^{\ell}} \rangle \cong C_m \rtimes_r C_{n/p^{\ell}},
\)
and $p$ does not divide the order of the kernel of the action of $C_{n/p^{\ell}}$ on $C_m$. 
However, $p$ divides $n / p^{\ell} = sd$, while $r^{(n/p^{\ell})/p} \not\equiv 1 \mod m$ by construction. 
Therefore the criterion in \eqref{eq:Z-groupFrobenius} shows that $B/\langle b^{n/p^{\ell}}\rangle$ is not a Frobenius complement when $d \neq 1$, that is, when $B$ is non-abelian. 

Set $N := \langle b^{n/p^{\ell}}\rangle$; it remains to verify the other two properties. 
Since all Sylow subgroups of $B/N$ are cyclic, it follows from \Cref{lem:firdifferentfields} that $B/N$ has a faithful irreducible $\Fi$-representation. 
Lastly, the description $B/N \cong C_m \rtimes_r C_{n/p^{\ell}}$ combined with \eqref{eq:centerZ-group} shows that $\pi^{-1}(\ZZ(B/N)) = \ZZ(B)$; in particular, $B/N$ is not abelian. 
\medskip

\noindent \underline{{Case B:}} The Sylow $2$-subgroups of $B$ are generalized quaternion groups. 
\smallskip

Pick a Sylow $2$-subgroup of $B$, and identify it with $Q_{4m}$. 
Set $N := \ZZ(Q_{4m}) = \langle a^m \rangle \cong C_2$ and recall that $Q_{4m}/N \cong D_{2m}$ is a dihedral group. 
When $|B|$ is even, $a^m$ is the unique involution in $B$ (see \cite[Theorem 11.4.5]{JespersdelRio16}), and is therefore central in $B$. 
Thus the quotient $B/N$ (which does not depend on the choice of Sylow $2$-subgroup), is not a Frobenius complement as its Sylow $2$-subgroups are dihedral. 
Next, if $B \ncong Q_8 \times C$ with $C$ some odd-order cyclic group, then all Sylow subgroups of $B/N$ have cyclic center; hence $B/N$ has a faithful irreducible $\Fi$-representation by \Cref{lem:firdifferentfields}.

If $B/N$ were abelian, then $B$ would be nilpotent of class $2$. 
However, nilpotent Frobenius complements are either cyclic or isomorphic to $Q_{2^n} \times C_k$ for some odd $k$ and $n \geq 3$, see \cite[Corollary 11.4.7]{JespersdelRio16}. 
Since the nilpotency class of $Q_{2^n} \times C_k$ is $n-1$, we deduce that $B/N$ is abelian precisely when $k$ is odd and $n=3$. 
\medskip

For the last part of the statement, we assume that $B \ncong Q_8 \times C_k$ and set $\pi: B \to B/N$ the quotient map. 
We want to determine when $\pi^{-1}(\ZZ(B/N)) \neq \ZZ(B)$. 
Pick $x \in \pi^{-1}(\ZZ(B/N)) \setminus \ZZ(B)$. 
For $g \in B$, we have by definition $\pi(x^g)=\pi(x)$, hence either $x^g = x$ or $x^g = x a^m$; fix $g \in B$ for which $x^g = x a^m$. 

Firstly, suppose $x$ has odd order. 
As $a^m$ is a central involution and $o(x)$ is odd, we have $o(x) = o(xa^m)= 2o(x)$, a contradiction. 
Therefore, all odd-order elements in $\pi^{-1}(\ZZ(B/N))$ are central. 
Switching the roles of $x$ and $g$, the same argument also shows that every element of $\pi^{-1}(\ZZ(B/N))$ commutes with every odd-order element of $B$. 

Next, let $x$ be arbitrary, and write $x = x_2 x_{2'}$ and $g = g_2 g_{2'}$ the decompositions of $x$ and $g$ in $2$- and $2'$-parts. 
Since $g_{2'}$ has odd order, we have just seen that $[x,g_{2'}]=1$. 
Moreover, as $x_{2'} \in \langle x \rangle \leq \pi^{-1}(\ZZ(B/N))$, the previous paragraph also shows that $x_{2'}\in \ZZ(B)$. 
Hence, we have that $[x,g] = [x,g_2]\cdot [x,g_{2'}]^{g_2} = [x,g_2] = [x_2, g_2]$. 
Thus it only remains to understand the case where $x$ and $g$ both have order a power of $2$, which we now assume. 

Let $Q$ be a Sylow $2$-subgroup of $B$ containing $x$. 
Observe that since the order of $x$ is a power of $2$, $\langle x \rangle$ is normal in $B$. 
Indeed, the unique involution $a^{m}$ of $B$ must be contained in $\langle x \rangle$, implying that every conjugate ($x$ or $x a^m$) of $x$ belongs to $\langle x \rangle$. 
In consequence, the following reasoning is actually independent of the choice of $Q$. 

By assumption of this second case, $Q \cong Q_{4m}$ has presentation given by \eqref{eq:presentationquaternion}, and $\pi(x)$ is contained in the Sylow $2$-subgroup $Q/N \cong D_{2m}$ of $B/N$. 
If $m > 2$, the center of $Q/N$ is $\langle \pi(a^{m/2})\rangle$ and $\langle a^{m/2} \rangle = \pi^{-1}(\langle \pi(a^{m/2})\rangle)$. 
Thus if $m >2$ we deduce that 
\begin{equation} \label{eq:centerquaternioncase}
\pi^{-1}(\ZZ(B/N)) \leq \langle \ZZ(B), a^{m/2} \rangle,
\end{equation}
and the right hand side is independent of the choice of $Q$. 
It remains to understand the case $m=2$. 
This will require a finer analysis of the structure of Frobenius complements.
\medskip

First we claim that $\pi^{-1}(\ZZ(B/N))$ is a subgroup of the Fitting subgroup $\Fit(B)$ of $B$. 
Indeed, for $x \in \pi^{-1}(\ZZ(B/N))$ all the conjugates $\pi(x)^g$ for $g \in B$ commute. 
Consequently, all the commutators $[x,x^g] \in N = \langle a^m \rangle$ lie in a central subgroup of order $2$. 
Therefore, the normal closure $\langle x^g \mid g \in B \rangle$ is a normal subgroup of nilpotency class $2$ and hence a subgroup of $\Fit(B)$.

Now we distinguish several cases. 
To start, suppose that $B$ is non-solvable. 
Then by \cite[Theorem 11.4.10]{JespersdelRio16} the group $B$ is isomorphic to $Y \times M$ with $Y$ a $Z$-group of order prime to $60$ and $M$ either equal to $\SL_2(\F_5)$ or its index-$2$ supergroup $\langle v,d \mid d^8=v^3=1, (d^{-2}v)^5=d^5, v^d=vd^2(v,d^{-2}) \rangle$\footnote{The identifier for GAP or Magma's Small Groups Library is [240,90].}. 
If $x \in Y \cap \pi^{-1}(\ZZ(B/N))$, then by the odd-order case above, $x \in \ZZ(B)$. 
On the other hand, if $x \in M \cap \in \pi^{-1}(\ZZ(B/N))$, then we have shown that $x \in \Fit(M)$, and as it turns out, $\Fit(M) = N = \ZZ(M) \leq \ZZ(B)$ for both possibilities for $M$. 
In summary, for $B$ non-solvable we have $\pi^{-1}(\ZZ(B/N)) = \ZZ(B)$.

Next suppose that $B$ is nilpotent. 
As $B$ is assumed non-abelian and not isomorphic to $Q_8 \times C_m$, we must have by \cite[Corollary 11.4.7]{JespersdelRio16} that $B \cong Q_{2^n} \times C_m$ with $m$ odd and $n \geq 4$. 
Hence in this case $\pi(a^{m/2})= \ZZ(Q_{2^n}/\ZZ(Q_{2^n})) \leq \ZZ(B/N)$ and thus $\pi^{-1}(\ZZ(B/N)) = \langle \ZZ(B), a^{m/2} \rangle$. 

Finally, suppose that $B$ is solvable, but not nilpotent. 
In that case, it is shown in the proof of \cite[Theorem 18.2]{Passman68} that exactly one of the following cases occur, denoting $\Fit(B)_2$ the Sylow $2$-subgroup of $\Fit(B)$:
\begin{enumerate}[itemsep=1ex,topsep=1ex,label=\textup{(\arabic*)},leftmargin=2em]
	\item If $\Fit(B)_2$ is cyclic, then $\Fit(B) \leq G_0$ with $G_0$ a Z-group such that $G/G_0$ is $C_2$ or $C_2 \times C_2.$
	\item If $\Fit(B)_2 \cong Q_8$, then $\rC_G(\Fit(B)_2)$ is a Z-group and $G/\rC_G(\Fit(B)_2)$ is a subgroup of $\Aut(Q_8) \cong S_4$.
	\item If $\Fit(B)_2 \cong Q_{2^n} = \langle \overline{a}, \overline{b} \rangle$ with $n \geq 4$, then $\rC_G(\langle \overline{a} \rangle)$ is a Z-group and $G/\rC_G(\langle \overline{a} \rangle) \cong C_2.$
\end{enumerate}

Take $x \in \pi^{-1}(\ZZ(B/N)) \setminus \ZZ(B)$ an element of order a power of $2$. 
As noted earlier, $x \in \Fit(B)_2$. 
Pick a Sylow $2$-subgroup containing $\Fit(B)_2$, and identify it with $Q_{2^n} = \langle a, b \rangle$. 

In Case (1), $\Fit(B)_2 \leq \langle a \rangle$ and every cyclic subgroup of $\Fit(B)_2$ is normal in $B$. 
This leaves for only possibility that $x = a^{m/2}$. 
As $\Fit(B)$ is contained in the Z-group $G_0$, we have $x \in \pi^{-1}(\ZZ(G_0/N))$ and thus $x \in \ZZ(G_0)$ by virtue of Case A above. 
Since the exponent of $G/G_0$ is $2$, this implies that $x$ commutes with all elements of odd order. 
As the image of $a^{m/2}$ is central in the Sylow $2$-subgroups of $B/N$, we obtain altogether that $a^{m/2} \in \pi^{-1}(\ZZ(B/N)) \setminus \ZZ(B)$. 

For the remaining cases write $\Fit(B)_2 = \langle \overline{a}, \overline{b} \rangle$ and pick $Q \cong Q_{4m} = \langle a, b \rangle$ a Sylow $2$-subgroup of $B$, necessarily containing $\Fit(B)_2$.

In Case (3) we can apply \eqref{eq:centerquaternioncase}. 
It only remains to verify whether $a^{m/2} \in \pi^{-1}(\ZZ(B/N)) \setminus \ZZ(B)$. 
Again, as the image of $a^{m/2}$ is central in $Q/N$ it suffices to verify that $a^{m/2}$ commutes with all odd order elements, which holds because $\indx{G}{\rC_G(\langle \overline{a} \rangle)}=2$.

Finally, consider Case (2). 
Note that $\Fit(B)_2 = \langle a^{m/2}, y \rangle$ with $y$ equal to either $b$ or $a^{2^j}b$ for some $j$. 
If $3$ does not divide $\indx{B}{\Fit(B)_2}$, then all odd order elements of $B$ lie in $\rC_G(\Fit(B)_2)$. Hence, $\pi(a^{m/2})$ commutes with the elements of $\ZZ(B/N)$ whose order is either odd, or a power of 2, that is, $a^{m/2} \in \pi^{-1}(\ZZ(B/N))$. 
Now note that $Q$ properly contains $\Fit(B)_2$. 
Indeed, if $Q = \Fit(B)_2 \cong Q_8$, the proof of Case (1) shows that this would imply that $B$ is nilpotent, which was excluded. 
Since $Q \not\cong Q_8$, the only element whose order is a power of $2$ and commutes with $\pi(Q)$ is $\pi(a^{m/2})$. 
The only possibilities for $x \in \pi^{-1}(\ZZ(B/N))\setminus \ZZ(B)$ are therefore $x = a^{m/2}$ or $a^{3m/2}$. 
On the other hand, if $3$ divides $\indx{B}{\Fit(B)_2}$, then such element of order $3$ transitively permutes all the elements of order $4$ \cite[Proposition 9.9]{Passman68} and in particular, non-central elements of $\Fit(B)_2$ do not generate a normal subgroup in $B$. 
So $\pi^{-1}(\ZZ(B/N)) = \ZZ(B)$.

In summary, for $B$ solvable but not nilpotent, we have shown that 
\begin{align*}
a^{m/2} \notin \pi^{-1}(\ZZ(B/N)) &\iff \langle a^{m/2} \rangle \text{ is not normal in } B \\
	&\iff \Fit(B)_2 \cong Q_8 \text{ and $3$ divides } \indx{G}{\rC_G(\Fit(B)_2)}.
\end{align*}
As the same equivalence holds for both the non-solvable and the nilpotent cases, this finishes the proof.
\end{proof}
\medbreak

Next, we determine the groups whose image is a Frobenius complement under every irreducible representation. 
Formally speaking, this theorem is the particular case $H = 1$ in \Cref{thm:almostfaithfulembedding}.\ref{item:notFrobenius}; we record it here as it will be used later, and could be of independent interest. 

\begin{theorem} \label{thm:DedekindiffimagesFrobeniuscomplement}
Let $G$ be a finite group, and $\Fi$ be a field whose characteristic does not divide $|G|$. 
The following are equivalent. 
\begin{enumerate}[itemsep=1ex,topsep=1ex,label=\textup{(\roman*)},leftmargin=2em]
	\item $G$ is a Dedekind group. 
	\item Every irreducible $\Fi$-representation of $G$ is fixed-point-free. 
	\item The image $\rho(G)$ of $G$ under every irreducible representation $\rho \in \Irr_F(G)$ is a Frobenius complement. 
\end{enumerate}
\end{theorem}
\begin{proof}
If $G$ is a Dedekind group, then for any $\rho \in \Irr_F(G)$, $\rho(G)$ is a Dedekind group with cyclic center. 
By the classification theorem, $\rho(G)$ is either cyclic or isomorphic to $Q_8 \times C_n$ with $n$ odd. 
This makes apparent that $\rho$ is indeed fixed-point-free: the faithful irreducible $\Fi$-representations of both cyclic groups and groups of the form $Q_8 \times C_n$ (with $n$ odd) are all fixed-point-free. 
This shows (i) $\implies$ (ii). 
\smallskip

For the converse, suppose that every irreducible $\Fi$-representation of $G$ is fixed-point-free. 
As we have seen in \eqref{eq:fpfviaidempotent} above, this means that for every $e \in \PCI(\Fi G)$, the element $\wt{g}e$ is central in $\Fi G$. 
But then $\wt{g} = \sum_{e \in \PCI(\Fi G)} \wt{g}e$ is central in $\Fi G$.
As $G$ is linearly independent in $\Fi G$, this is equivalent to the subset $\langle g \rangle$ of $G$ being stable under conjugation. 
In other words, every cyclic subgroup of $G$ is normal, or yet, $G$ is Dedekind.
\smallskip

The implication (ii) $\implies$ (iii) is immediate, since a group is a Frobenius complement if and only if it is fixed-point-free, that is, admits a faithful (irreducible) fixed-point-free representation. 
\smallskip

It remains to prove (iii) $\implies$ (ii). 
By assumption, $\rho(G)$ is a Frobenius complement for any $\rho \in \Irr_\Fi(G)$. 
If some $\rho(G)$ were a non-abelian group distinct from $Q_8 \times C_m$ with $m$ odd, 
\Cref{lem:centralquotientFrobeniuscomplement} would exhibit a quotient of $G$ which is not a Frobenius complement, but which does admit a faithful irreducible representation $\sigma$. 
The inflation of $\sigma$ to $G$ would then contradict the hypothesis. 
It must thus be that $\rho(G)$ is cyclic or isomorphic to $Q_8 \times C_m$ with $m$ odd.
For these groups, we have already observed that every (irreducible) representation is fixed-point-free. 
\end{proof}

The last technical lemma before the proof of \Cref{thm:almostfaithfulembedding} concerns the specificities of Dedekind groups. 

\begin{lemma} \label{lem:embeddingalmostdedekind}
Let $G$ be a subgroup of $Q_8^t \times A$ with $t \in \N$ and $A$ a finite abelian group. 
Let $\Fi$ be a field whose characteristic does not divide $|G|$. 
\begin{enumerate}[itemsep=1ex,topsep=1ex,label=\textup{(\roman*)},leftmargin=2em]
	\item For $\rho \in \Irr_\Fi(G)$, the images $\rho(G)$ are either cyclic or isomorphic to $E \times C_m$, where $E$ is an extraspecial $2$-group and $m$ is odd. 
If $G$ is Dedekind, then in fact $E \cong Q_8$. 
	\item If $G$ is not Dedekind, then there exists $\rho \in \Irr_{\Fi}(G)$ with $\ZZ(\rho) \leq \ZZ(G)$ and such that the Sylow $2$-subgroup of $\rho(G)$ is an extraspecial $2$-group different from $Q_8$.
	\item If $G$ is a non-abelian Dedekind group, then there exists $\rho \in \Irr_{\Fi}(G)$ with $\ZZ(\rho) \leq \ZZ(G)$ and
\begin{enumerate}[label=$\bullet$,leftmargin=\parindent]
	\item $\rho(G)$ not a division algebra, if and only if $FG$ is not a product of division algebras.
	\item (when $F$ is a number field) $\rho(G)$ is neither a field nor a totally definite quaternion algebra, if and only if $F$ is not totally real, or $G \ncong Q_8 \times C_2^n$ for any $n \in \N$. 
\end{enumerate}
\end{enumerate}
\end{lemma}
\begin{proof}
Before approaching Statements (i) to (iii), we first need to classify the subgroups of $Q_8^t \times A$. 
To start with, as $G$ is nilpotent it decomposes into $G_{2} \times G_{2'}$, where $G_2$ is the Sylow $2$-subgroup of $G$ and $G_{2'}$ its Hall $2'$-subgroup (which is abelian). 
There is no restriction on $G_{2'} \leq A$, so we focus on classifying the possibilities for $G_{2}$. 
\smallskip

\noindent \underline{Preliminary work:} Classification of subgroups of $Q_8^{t} \times A_2$ with $A_2$ an abelian $2$-group. 
\smallskip

Denote $Z_0 := \ZZ(Q_8^t)$. Recall that $Z_0 \cong C_2^t$ and $V := Q_8^t/Z_0 \cong C_2^{2t}$; we view $V$ as a $\F_2$-vector space. 
Since the nilpotency class of $Q_8^t$ is $2$, the commutator induces a non-degenerate alternating bilinear form
\[
[{\cdot}\, ,{\cdot}]: V \times V \rightarrow Z_0. 
\]
Every subgroup $M$ of $ Q_8^t$ is uniquely described by a triple $(U,K_0,\lambda)$ where: 
\begin{enumerate}[label=$\bullet$,leftmargin=2em]
	\item $U$ is an $\F_2$-subspace of $V$,
	\item $K_0 \leq Z_0$ is a subgroup containing the commutator $[U,U]$ (i.e.\ $U$ is $K_0$-isotropic for $[{\cdot}\, ,{\cdot}]$).
	\item $\lambda$ is a cocycle in $Z^2(U,K_0)$, or equivalently a section $\tau : U \rightarrow Q_8^t$ is given such that 
\[
\tau (u) \cdot \tau(v) = [u,v] \tau(u + v),
\]
and $o(\tau(u))=4$ for all non-trivial $u\in U$. 
\end{enumerate}
Concretely, to such triple one associates the set
\[
M_{(U,K_0)} := \{ \tau(u) k \mid u \in U, \, k\in K_0\},
\]
which is a subgroup thanks to the $2$-cocycle condition. 
Note that, up to isomorphism, the section $\tau$ involved in the definition is unique up to a $1$-cocycle, so in the sequel we will omit to describe it. 
The condition $o(\tau(u))=4$ comes from the fact the exponent of $Q_8^t$ is $4$, and all elements of order $2$ are central. 
Thus $U$ in fact corresponds to choosing order $4$ elements.

Next, for a given $G \leq Q_8^t\times A_2$ we consider the projection $\pi: G \to Q_8^t: (g,b) \mapsto g$ and $B := \ker(\pi)= G \cap A_2$. Then the subgroup $G \le Q_8^t \times A_2$ fits into a central extension
\begin{equation}\label{eq:descriptionascentralext}
1 \to B \to G \to G_t \to 1,
\end{equation}
where $G_{t} :=\im(\pi) \leq Q_8^t$. In other words, $G$ corresponds to a class in the second cohomology group $H^2(G_t,B)$, where $G_t$ can be described via a triple as above.
\medskip

\noindent \underline{Statement (i):} Description of the images of $G$ under irreducible representations. 
\smallskip

If $\rho$ is an irreducible $\Fi$-representation of $G$, then the image $\rho(G)$ obviously has a faithful irreducible representation. 
Furthermore, $G$ is nilpotent and thus by \Cref{lem:firdifferentfields}, a quotient $G/N$ of $G$ has a faithful irreducible $\Fi$-representation if and only if $\ZZ(G/N)$ is cyclic. 
\smallskip

\textit{Claim:} For $N \triangleleft G$ such that $\ZZ(G/N)$ is cyclic, the quotient $G/N$ is either cyclic or isomorphic to $E \times C_m$ with $E$ an extraspecial $2$-group and $m$ odd.
\smallskip

From the above claim follows already the second part of statement (i). 
Indeed, if $G$ is Dedekind, it is isomorphic to $Q_8 \times C_2^k \times A_\mathrm{odd}$ with $k \in \N$ and $A_\mathrm{odd}$ some odd-order abelian group. 
Therefore $E$ is a quotient of $Q_8 \times C_2^k$, hence is isomorphic to $Q_8$. 
As $\ZZ(G/N)$ is cyclic, we also see that $A_\mathrm{odd}$ is sent to a cyclic group $C_m$ with $m$ odd.
\smallskip

We now focus on proving the above claim. 
Recall the description of $G$ obtained earlier: $G \cong G_2 \times G_{2'}$ with $G_{2'}$ abelian, the group $G_2$ is given by the central extension \eqref{eq:descriptionascentralext}, and $\pi(G_2)$ corresponds to a tuple $(U,K_0)$. 

The orders of $G_2$ and $G_{2'}$ being coprime, a quotient $(G_2 \times G_{2'})/N$ has cyclic center if and only if $N = N_2 \times N_{2'}$ with $\ZZ(G_2/N_2)$ and $G_{2'}/N_{2'}$ cyclic. 
We are thus reduced to understanding quotients of $G_2$ which have cyclic center. 

We start by investigating the quotients $G_2/N$ with $B \leq N$ and $\ZZ(G_2/N)$ cyclic. For such choice of $N$ one has that $G_2/N \cong \pi(G_2)/N_{0}$ for some normal subgroup $N_{0}$ of $\pi(G_2)$. Since $K_0/K_0 \cap N \leq \ZZ (G_2/N)$, the quotient $K_0/K_0 \cap N$ must be cyclic. Next observe that
\[
\ZZ(\pi(G_2)/N_0) = \pi(U_c) \cdot K_0/N_0
\]
with $U_c := \Span_{\F_2} \{ \tau(u) \mid u \in U \text{ and } [u,U] \in N_0\}$. 
Thus if $N_0 \geq K_0$, then $\ZZ(\pi(G_2)/N_0)= \pi(U_c) = \pi(G_2)/N_0$, i.e.\ $G/BN_0$ is abelian. 
Next consider the case where $K_0 \gneq N_0 \cap K_0$. 
In order for $\ZZ(\pi(G_2)/N_0)$ to be cyclic, we need that $K_0/N_0\cap K_0$ be cyclic and $\pi(U_c)$ be trivial. 
Therefore, in that case $\pi(G_2)/N_0$ has a center of order $2$ such that the quotient is an elementary abelian $2$-group, that is, $\pi(G_2)/N_0$ is extraspecial. 

Now consider a quotient $G_2/N$ with $B \nleq N$. 
In this case, in order to have a cyclic center one needs that $\indx{B}{N \cap B}=2$. 
Let $\mathcal{T}_{N\cap B}^B=\{1,x\}$ be coset representatives of $N \cap B$, and let $\langle y \rangle \leq B$ be a maximal cyclic subgroup of $B$ above $\langle x \rangle$. 
Note that $\tau(\pi(G_2)') = 1$ since $G_2' \cap B = 1$ (because $G$ is class $2$ and $B$ abelian). 
Thus $\ZZ (G_2) = \pi^{-1}(\ZZ(\pi(G_2)))$. 
Therefore we need that $\ZZ(G_2/N) = \langle z \rangle$ with $z$ such that $\langle y \rangle \leq \langle z \rangle$ and $\ZZ(\pi(G_2)) = \langle \pi(z) \rangle$. 
The latter is only possible for particular choices of $(U,K_0)$ for which the quotient will be an extraspecial group of order larger than $8$, except if $G$ is a split extension of $B$ with $G_t$; in that case the quotient is cyclic. 
\medskip

\noindent \underline{Statement (ii):} Existence of a center-preserving representation when $G$ is not Dedekind. 
\smallskip

As mentioned earlier, the set of images $\rho(G)$ coincide with the set of quotients $G/N$ such that $\ZZ(G/N)$ is cyclic. 
Denote by $\pi_N: G \rightarrow G/N$ the quotient map. 
We will now construct an explicit quotient such that $G/N$ is extraspecial, distinct from $Q_8$ (hence $\ZZ(G/N)$ is cyclic), and $\pi_N^{-1}(\ZZ(G/N))= \ZZ(G)$. 
In particular, $\pi_N$ composed with any faithful irreducible representation of $G/N$ will yield a center-preserving representation of $G$. 

Recall that to $\pi(G_2)$ was associated a triple $(U,K_0,\lambda)$, in which $K_0$ is central in $G$, $\tau(u)$ has order $4$ for each non-zero $u \in U$ and $\tau(u)^2 \in K_0$. 
Pick a basis $\{u_1, \ldots, u_{\ell}\}$ of $U$ over $\F_2$ and let $\mathcal{C}$ be the elements of $K_0$ which are not of the form $\tau(u_i)^2$ for some $1 \leq i \leq \ell$. 
Set
\(
N_0 := \langle \tau(u_i)^2\tau(u_j)^{-2}, x \mid x \in \mathcal{C}, \, 1 \leq i \neq j \leq \ell \rangle \triangleleft K_0,
\)
so that $K_0/N_0$ is cyclic and $\pi(G_2)/N_0$ is extraspecial. 

Next decompose $G_{2'} = \langle y \rangle \times R$ with $\langle y \rangle$ maximal cyclic. 
We claim that the normal subgroup $N := BN_0R$ of $G$ is such that $G/N$ has the properties required. 

First, note that $N$ is a central subgroup of $G$ with $N \cap G_2=BN_0$ an elementary abelian $2$-group. 
Moreover, by construction, $G/N \cong \pi(G_2)/N_0 \times \langle y \rangle$ with $\pi(G_2)/N_0$ an extraspecial $2$-group. 
From this we deduce that $\ZZ(G/N)= K_0/N_0$. 
Since $N \cap G_2=BN_0$, it follows that $\pi^{-1}_N(\ZZ(G/N))$ contains no element of order $4$ outside of $B$. 
Therefore, $\pi^{-1}_N(\ZZ(G/N)) \leq \ZZ(G)$, as desired. 

It remains to verify that $\pi(G_2)/N_0 \ncong Q_8$. 
For this we analyze further the couple $(U,K_0)$. 
If $\dim_{\F_2} U \geq 3$, then $|\pi(G_2)/N| \geq 2^{\dim_{\F_2} U} \cdot |K_0| \geq 16$ is too large to be $Q_8$. 
If $\dim_{\F_2} U \leq 1$, then $G_2/BK_0$ is cyclic with $BK_0 \leq \ZZ(G)$ and hence $G_2$ abelian, a contradiction to the non-Dedekind assumption. 
Finally, if $\dim_{\F_2} U = 2$, the quotient must be $Q_8$ or $D_8$. 
The former would imply that $G$ is Dedekind, and the latter case is fine, finishing the proof of the claim. 
\medskip

\noindent \underline{Statement (iii):} Existence of a center-preserving representation when $G$ is Dedekind. 
\smallskip

For the first point, it is clearly necessary that $\Fi G$ has a component which is not a division algebra. 
Also, if $G \cong Q_8 \times C_2^n$, then $\Fi G \cong \Fi^{\oplus 2^{n+1}} \oplus \qa{-1}{-1}{\Fi}^{\oplus 2^n}$. 
Thus also in the second point the necessity follows. 

For the sufficiency, recall from the Baer--Dedekind classification \cite{Dedekind97,Baer33} that $G \cong Q_8 \times C_2^n \times A$ with $n \in \N$ and $A$ an odd abelian group. 
A theorem of Perlis and Walker \cite[Theorem 3.3.6]{JespersdelRio16} states that 
\[
\Fi A \cong \bigoplus_{d \, \mid \, |A|} k_d\frac{[\Q(\zeta_d):\Q]} {[\Fi (\zeta_d):\Fi]}\Fi (\zeta_d),
\]
where $k_d$ is the number of cyclic subgroups of order $d$ in $A$ (and $\zeta_d$ a $d$-th root of unity). 
Hence the simple components of $\Fi G$ are of the form $\Fi (\zeta_d)$ or $\qa{-1}{-1}{\Fi (\zeta_d)}$. 
Therefore, if $\Fi$ and $d$ are not such that $\qa{-1}{-1}{\Fi (\zeta_d)}$ is totally definite, then the associated representation, say $\rho_d$, satisfies the requirements. 
Moreover, by construction $\rho_d(G) \cong Q_8 \times C_d$. 
This implies that $\ZZ(\rho_d(G))= \langle a^2, x, \ker(\rho_d(G)) \rangle$ for some $x \in A$ and $\ker(\rho_d(G)) \leq \ZZ(G)$. 
In particular, $\rho_d$ is center-preserving.
\end{proof}
\medbreak

We are now finally ready to prove the embedding theorem.

\subsection{Proof of \texorpdfstring{\Cref{thm:almostfaithfulembedding}}{Theorem \ref{thm:almostfaithfulembedding}}} \label{subsec:proofthmalmostfaithfulembedding}
It suffices to prove \Cref{thm:almostfaithfulembedding} under the stronger assumption that $H$ has a faithful irreducible $\Fi$-representation (cf.\ \Cref{thm:faithfulcenterpreservingrepresentation}). 
Indeed, given an irreducible representation $\sigma$ of $G$ which is center-preserving on $H$, if $\sigma(G)$ does not already satisfy the desired conclusion \ref{item:notFrobenius}, \ref{item:notdivision} or \ref{item:notquaternion}, then $\sigma(G)$ must be a Frobenius complement. 
Thus $\sigma(G)$ and its subgroup $\sigma(H)$ are fixed-point-free groups, and in particular, $\sigma(H)$ admits a faithful irreducible representation. 

The statement (applied to the pair $\sigma(H) \leq \sigma(G)$, which now satisfies the stronger assumption) then yields the existence of an irreducible representation $\rho$ of $\sigma(G)$ satisfying 
\[
\indx{\sigma(H) \cap \ZZ(\rho)}{\sigma(H) \cap \ZZ(\sigma(G))} \leq 2,
\]
as well as Conclusion \ref{item:notFrobenius}, \ref{item:notdivision} or \ref{item:notquaternion} respectively. 
The inflation $\overline{\rho}$ of $\rho$ to $G$ obviously also satisfies the respective conclusion \ref{item:notFrobenius}, \ref{item:notdivision} or \ref{item:notquaternion}. 
By definition, the inverse image of $\sigma(H) \cap \ZZ(\rho)$ under $\restr{\sigma}{H}$ is $H \cap \ZZ(\overline{\rho})$. 
On the other hand, since $\sigma$ is center-preserving on $H$, the inverse image of $\sigma(H) \cap \ZZ(\sigma(G))$ under $\restr{\sigma}{H}$ is precisely $H \cap \ZZ(G)$. 
We deduce that
\[
\indx{H \cap \ZZ(\overline{\rho})}{H \cap \ZZ(G)} = \indx{\sigma(H) \cap \ZZ(\rho)}{\sigma(H) \cap \ZZ(\sigma(G))} \leq 2,
\]
showing that $\overline{\rho}$ is the desired representation. 
Moreover, $H \cap \ZZ(\overline{\rho}) = H \cap \ZZ(G)$ if and only if $\sigma(H) \cap \ZZ(\rho) = \sigma(H) \cap \ZZ(\sigma(G))$, and since \Cref{cond:NQ} descends to quotients, the second equality does indeed hold if $G$, $H$ satisfy \ref{cond:NQ}. 

In addition, note that we may also assume that $H$ is not central in $G$. 
Indeed, if $H \leq \ZZ(G)$, then of course $H \cap \ZZ(\rho) = H \cap \ZZ(G) = H$ is automatically satisfied for any representation $\rho$ of $G$. 
The rest of the statement, which concerns only $G$, would then follow by replacing $H$ by any non-central cyclic subgroup of $G$, for example (by assumption, $G$ is not abelian). 
\medskip

Taking these two reductions into account, we start by supposing that $G$ is not a Dedekind group and tackle part \ref{item:notFrobenius}. 
By \Cref{thm:faithfulcenterpreservingrepresentation} there exists an irreducible $\Fi$-representation $\sigma$ of $G$ such that $H \cap \ZZ(\sigma) = H \cap \ZZ(G)$ and $H \cap \ker(\sigma) =1$. 
Since we assumed $H$ to be non-central, $\sigma(G)$ is certainly not abelian. 

If to begin with $\sigma(G)$ is not a Frobenius complement, then there is nothing more to prove; so we suppose that $\sigma(G)$ is a Frobenius complement. 
In consequence, \Cref{lem:centralquotientFrobeniuscomplement} applies to $\sigma(G)$ and, unless $\sigma(G) \cong Q_8 \times C_m$ with $m$ odd, yields a subgroup $N \leq \ZZ(\sigma(G))$ and a faithful irreducible representation $\tau$ of $\sigma(G)/N$. 
The inflation $\rho$ of $\tau$ to $G$ is then an irreducible representation of $G$ satisfying $\indx{H \cap \ZZ(\rho)}{H \cap \ZZ(G)} \leq 2$ as desired, unless $\sigma(G) \cong Q_8 \times C_m$ with $m$ odd. 
In view of the properties of $\sigma$, the latter case would imply that $H \cong \sigma(H)$ is a subgroup of $Q_8 \times C_m$, and in turn that $H / H \cap \ZZ(\sigma) = H / H \cap \ZZ(G)$ is a non-trivial subgroup of $C_2 \times C_2 \cong \sigma(G) / \ZZ(\sigma(G))$. 
We now deal with such subgroups $H$ of $\sigma(G) \cong Q_8 \times C_m$ separately, taking into account the other irreducible representations of $G$.

Let $\psi$ be a faithful irreducible representation of $H$ and decompose the induced representation
\begin{equation*}
\Ind_H^G(\psi) = \sigma_1 \oplus \cdots \oplus \sigma_{\ell}
\end{equation*}
into irreducible summands. 
By Frobenius reciprocity, $\psi$ is a constituent of the restriction of $\sigma_i$ to $H$, hence the restriction of every $\sigma_i$ to $H$ is faithful. 

Since $H / H \cap \ZZ(G) \cong C_2$ or $C_2\times C_2$, the index $\indx{H \cap \ZZ(\sigma_i)}{H \cap \ZZ(G)}$ equals $1$, $2$ or $4$, with $4$ only occurring if $H / H \cap \ZZ(G) \cong C_2 \times C_2$ and $H \subseteq \ZZ(\sigma_i)$. 
But if $H / H \cap \ZZ(G) \cong C_2 \times C_2$, then $H$ is not abelian (it contains a subgroup isomorphic to $Q_8$), hence cannot be contained in $\ZZ(\sigma_i)$, by faithfulness of $\sigma_i$ on $H$. 
The conditions $H / H \cap \ZZ(G) \cong C_2 \times C_2$ and $H \subseteq \ZZ(\sigma_i)$ are thus mutually exclusive, and as a result, $\indx{H \cap \ZZ(\sigma_i)}{H \cap \ZZ(G)} \leq 2$ for every $1 \leq i \leq \ell$. 
Hence, if some $\sigma_i(G)$ is not a Frobenius complement, then $\rho_i = \sigma_i$ is the desired representation. 
If some $\sigma_i(G)$ is a non-abelian Frobenius complement distinct from $Q_8 \times C_m$, then as before we can consider the inflation $\rho_i$ to $G$ of the faithful irreducible representation $\tau_i$ of $\sigma_i(G)/N_i$ afforded by \Cref{lem:centralquotientFrobeniuscomplement}. 
If $\ZZ(\rho_i) = \ZZ(\sigma_i)$, then $\indx{H \cap \ZZ(\rho_i)}{H \cap \ZZ(G)} = \indx{H \cap \ZZ(\sigma_i)}{H \cap \ZZ(G)} \leq 2$ is immediate. 
Otherwise, \Cref{lem:centralquotientFrobeniuscomplement} ensures that $\ZZ(\rho_i) \gneq \ZZ(\sigma_i)$ only when $\sigma_i(G)$ is a Frobenius complement having a quaternion Sylow $2$-subgroup $Q_{2^n}$ such that $\langle a^{2^{n-3}} \rangle$ is normal in $\sigma_i(G)$, in the notation of \eqref{eq:presentationquaternion}. 
In that case, since $H$ has an element of order $4$, it follows that $a^{2^{n-2}} \in \sigma(H \cap \ZZ(G))$, and if $H/H\cap \ZZ(G) \cong C_2$, then $\sigma_i(H) = \langle a^{2^{n-3}}, H \cap \ZZ(G) \rangle$, while if $H/H\cap \ZZ(G) \cong C_2 \times C_2$, then $\sigma_i(H) = \langle a^{2^{n-3}}, b, H \cap \ZZ(G) \rangle$. 
Because neither $a^{2^{n-3}}$ nor $b$ are central in $Q_{2^n}$, we see that $H \cap \ZZ(\sigma_i) = H \cap \ZZ(G)$, hence $\indx{H \cap \ZZ(\rho_i)}{H \cap \ZZ(G)} \leq 2$. 
Overall, $\rho_i$ constitutes the desired representation, so long as $\sigma_i(G) \not\cong Q_8 \times C_{m_i}$ for any odd $m_i$. 

On that account, we are reduced to the case where every $\sigma_i(G)$ is either abelian or a Frobenius complement of the form $Q_8 \times C_{m_i}$. 
As $\Ind_H^G(\psi)$ is faithful on $G$ (because $\ker(\Ind_H^G(\psi)) = \mathrm{core}_G(\ker(\psi)) = 1$), this means that $G$ is a non-abelian subgroup of $Q_8^t \times A$, with $A$ some abelian group. 
For such $G$, it was shown in \Cref{lem:embeddingalmostdedekind} that there even exists a representation $\rho$ which is center-preserving on the whole of $G$ and such that $\rho(G)$ is not a Frobenius complement, as long as $G$ is not a Dedekind group. 
\medskip

Second, for parts \labelcref{item:notdivision,item:notquaternion}, we examine the existence of an irreducible representation $\rho$ with the weaker property that $\rho(\Fi G)$ is not a product of division algebras, respectively $\rho(\Fi G)$ is neither a field, nor a totally definite quaternion algebra. 
In view of the first part, it remains only to treat the case where $G$ is a Dedekind group. 
For such $G$ (and for any of its subgroups $H$), the desired representation $\rho$ was shown to exist in \Cref{lem:embeddingalmostdedekind}. 
Moreover, the equality $H \cap \ZZ(\rho) = H \cap \ZZ(G)$ always holds in this case. 
\medskip

To conclude, we discuss the validity of the stronger property $H \cap \ZZ(\rho) = H \cap \ZZ(G)$. 
If at the start, $\sigma(G)$ was not a Frobenius complement, then $\rho = \sigma$ satisfies this equality. 

Next, suppose that $\sigma(G)$ is a Frobenius complement distinct from $Q_8 \times C_m$, and consider the representation $\rho$ that was constructed in this case. 
It follows from \Cref{lem:centralquotientFrobeniuscomplement} that $H\cap\ZZ(\rho)$ is larger than $H \cap \ZZ(\sigma)$ exactly when $\sigma(G)$ has a generalized quaternion Sylow $2$-subgroup $Q_{2^n}$, such that $\sigma(H)$ intersects $a^{2^{n-3}} \cdot \ZZ(\sigma(G))$ in the notation of \eqref{eq:presentationquaternion}. 
Let $x \in \sigma(H) \cap a^{2^{n-3}} \cdot \ZZ(\sigma(G))$ and assume, replacing $x$ by its $2$-part if need be, that the order of $x$ is a non-trivial power of $2$. 
Since the unique involution of $\sigma(G)$ is central (cf.\ \cite[Theorem 11.4.5]{JespersdelRio16}) while $x$ is not, the order of $x$ must be divisible by $4$. 
The normality of $\langle x \rangle$ in $\sigma(G)$ goes back to that of $\langle a^{2^{n-3}} \rangle$, see \Cref{lem:centralquotientFrobeniuscomplement}. 
This shows that \Cref{cond:NQ} stated in \Cref{thm:almostfaithfulembedding} is indeed sufficient to ensure $H\cap\ZZ(\rho) = H \cap \ZZ(G)$ in this case. 

The last case left to check is $\sigma(G) \cong Q_8 \times C_m$ for some odd $m$, in which case $H / H \cap \ZZ(G)$ is a non-trivial subgroup of $C_2 \times C_2$. 
But then obviously $\sigma(H) \cap Q_8$ must contain an element of order $4$, generating a necessarily normal subgroup. 
This completes the proof of \Cref{thm:almostfaithfulembedding}. 
\qed

\section{Free products in group rings} \label{sec:amalgamBovdimaps}

In \Cref{subsec:Bovdimaps} we develop further the first-order deformations from \Cref{ex:elementarydeformations}.\ref{ex:elementarydeformations2}, in connection with the construction of shifted bicyclic units (see \Cref{def:shiftedbicyclicunits}). 
We formulate in \Cref{conj:amalgamBovdimaps} our presumption that the images of two opposite shifted bicyclic maps always form a free product amalgamated along their intersection. 
This conjecture echoes \Cref{que:existencefaithfulembeddings}.\ref{que:existencefaithfulembeddings2}, putting forward two explicit conjugates of a subgroup of $\U(FG)$ that should play ping-pong with each other. 

Next, we show in \Cref{subsec:bicyclicgenericallyfree} that profinitely-generically, a bicyclic unit and a shifted bicyclic unit generate a free group. 
In addition, as a consequence of the work carried out in \Cref{sec:faithfulembedding,sec:amalgamBovdimaps}, we can precisely determine when a given finite subgroup admits a bicyclic unit as ping-pong partner.

\subsection{Shifted bicyclic units and a conjecture on free products in group rings} \label{subsec:Bovdimaps}
In this section, we will apply the results of \Cref{subsec:faithfulembedding} to the case of the group ring $\Alg = \Fi G$ over a number field $\Fi$, and finite subgroups of $\V(RG)$ where $R$ is an order in $\Fi$, via the construction from \Cref{ex:elementarydeformations}.\ref{ex:elementarydeformations3}. 

Recall that $\V(RG)$ denotes the kernel of the augmentation map $\epsilon: \Fi G \rightarrow \Fi: \sum_i a_i g_i \mapsto \sum a_i$ restricted to $\U(RG)$; thus $\U (RG) = \U(R) \cdot \V(RG)$. 
The group $\V(RG)$ presents the advantage that its finite subgroups are $R$-linearly independent, by a theorem of Cohn and Livingstone \cite{CohnLivingstone65}.\footnote{In \cite{CohnLivingstone65} this result is shown only for $\Fi$ a number field and $R$ its ring of integers. 
However the proof of \cite[Corollary 2.4]{delRio18} combined with the general version of Berman's theorem stated in \cite[Theorem III.1]{RoggenkampTaylor92} implies it in the generality claimed here. }

\begin{definition}\label{def:shiftedbicyclicunits}
Let $G$ be a finite group, $H$ be a finite subgroup of $\U(RG)$, and pick $x \in RG$. 
Set $\wt{H} = \sum_{h \in H} h$, and for cyclic subgroups, abbreviate $\wt{h} := \wt{\langle h \rangle} = h + h^2 \cdots + h^{o(h)}$ as before. 
The maps
\[
\sbic{H,x} : H \rightarrow \U(RG): h \mapsto h + \wt{H} x (1-h)
\]
and 
\[
\sbic{x,H} : H \rightarrow \U(RG): h \mapsto h + (1-h) x \wt{H}
\]
will be called the \emph{(left, resp.\ right) shifted bicyclic maps associated with $H$ and $x$}. 
An element in $\U(RG)$ of the form $\sbic{x,H}(h)$ or $\sbic{H,x}(h)$ will be called a \emph{(left, resp.\ right) shifted bicyclic unit}. 
\end{definition}

\begin{remark} \label{rem:bicylicunits}
When $H =\langle h \rangle$ is a cyclic group and $x \in G$, the elements $\sbic{\langle h \rangle,x}(h)$ and $\sbic{x,\langle h \rangle}(h)$ have been called Bovdi units in \cite{JanssensJespersTemmerman17} in honor of Victor Bovdi, who initiated the study of such units. 
In \cite{MarciniakSehgal23} these units were renamed shifted bicyclic units. 
Recall that \emph{bicyclic units} are elements of the form
\begin{equation*}
\bic{h,x} = 1 + \wt{\langle h \rangle} x (1-h) \quad \text{and} \quad \bic{x,h} = 1 + (1-h)x \wt{\langle h \rangle}
\end{equation*}
for $h \in G$ and $x \in RG$. 
Note that one can rewrite
\[
h (1 + \wt{H}x(1-h)) = \sbic{H,x}(h) = (1 + \wt{H} x h^{-1}(1-h))h,
\]
hence $\sbic{\langle h \rangle,x}(h) = h \bic{h,x} = \bic{h,xh^{-1}} h$ and similarly $\sbic{x, \langle h \rangle}(h) = \bic{x,h} h = h \bic{h^{-1}x,h}$.
In this sense, shifted bicyclic units are slight (torsion) adaptations of bicyclic units. 
As the terminology used in \cite{MarciniakSehgal23} better reflects the nature of these units, we adopt it here. 
\end{remark}

Note that the shifted bicyclic maps from \Cref{def:shiftedbicyclicunits} are instances of first-order deformations (with $\delta_h = \wt{H} x (1-h)$ or $(1-h) x \wt{H}$, cf.\ \Cref{ex:elementarydeformations}.\ref{ex:elementarydeformations2}). 
As such, the first two properties below follow from the considerations of \Cref{subsec:deformations}.

\begin{proposition}\label{prop:Bovdimaps}
Let $G$ be a finite group, $H$ be a finite subgroup of $\V(RG)$, and pick $x \in RG$. 
\begin{enumerate}[itemsep=1ex,topsep=1ex,label=\textup{(\roman*)},leftmargin=2em]
	\item \label{item:Bovdimaps1} 
The shifted bicyclic maps $\sbic{H,x}$ and $\sbic{x,H}$ are injective group morphisms. 
	\item \label{item:Bovdimaps2} 
The subgroups $H$, $\im(\sbic{H,x})$, and $\im(\sbic{x, H})$ are conjugate in $\U(\Fi G)$. 
	\item \label{item:Bovdimaps3} 
If in addition $H \leq G$ and $g \in G$, then the maps $\sbic{H,g}$ and $\sbic{g^{-1},H}$ are the identity on $H \cap H^{g}$, and $\im (\sbic{H,g}) \cap \im (\sbic{g^{-1},H}) = H \cap H^{g}$.
\end{enumerate}
\end{proposition}

\begin{proof}
As just noted, $\sbic{x, H}$ and $\sbic{H,x}$ are first-order deformations of $H$ in $\Fi G$; it follows from \Cref{def:firstorderdeformation} that they are group morphisms. 
\Cref{thm:separabledeformationsinner} states that the identity map and the maps $\sbic{x, H}$ and $\sbic{H,x}$ are all conjugate by $\U(\Fi G)$. 
In particular, $\sbic{x, H}$ and $\sbic{H,x}$ are injective, and parts (i) and (ii) are proved. 
\smallskip

For part (iii), let $h \in H$ and note that $\sbic{H,x}(h) = h$ if and only if $\wt{H}x = \wt{H}xh$; for $x \in G$, this is equivalent to $h \in H^x$. 
Similarly $\sbic{g^{-1},H}(h) = h$ if and only if $h \in H^g$, and so $h \in \im(\sbic{H,g}) \cap \im(\sbic{g^{-1},H})$ when $h \in H \cap H^g$. 
Conversely, suppose that
\[
h + \wt{H} g (1-h) = \sbic{H,g}(h) = \sbic{g^{-1},H}(k) = k + (1-k)g^{-1}\wt{H}
\]
for some $h,k \in H$. 
In other words,
\[
\wt{H} g (1-h) - (1-k)g^{-1}\wt{H} = k-h. 
\]
Multiplying on the left by $\wt{H}$, we deduce that $\wt{H} g (1-h) = 0$, equivalently $\wt{H} = \wt{H} ghg^{-1}$. 
This means that $ghg^{-1} \in H$, and by the above, that $\sbic{H,g}(h) = h \in H \cap H^g$ as desired. 
\end{proof}

The shifted bicyclic maps can be used to construct several kinds of subgroups of $\U(RG)$, without the explicit use of the group basis of $RG$. 
For example, they were used (under different terminology) in \cite[Proposition 3.2.]{JanssensJespersTemmerman17} to produce solvable subgroups and free subsemigroups of $\U(RG)$. 
The next proposition displays another construction that makes use of shifted bicyclic maps. 

Recall that $I(RG)$ denotes the kernel of the augmentation map $\epsilon: RG \to R$; as an $R$-module, $I(RG)$ has basis $\{1-g \mid g \in G\}$. 
When $H$ is a subgroup of $G$, we will use the same notation to denote the kernel $I(R[G/H])$ of the $RG$-module map $\epsilon: R[G/H] \to R: gH \mapsto 1$. 

\begin{proposition}\label{prop:Bovdigroup} 
Let $G$ be a finite group, $H \leq G$, $g \in G$, and set $C = H \cap H^g$. Then 
\[
\langle H, \im(\sbic{g^{-1},H}) \rangle \cong I(\Z[H/C]) \rtimes H,
\]
where $h\in H$ acts on $I(\Z[H/C])$ via left multiplication by $h^{-1}$. 
\end{proposition}

\begin{proof}
For the sake of convenience, we abbreviate $\sbic{} = \sbic{g^{-1},H}$ and $\bic{h} = \bic{g^{-1},h}$. 
Set $M = \langle H, \sbic{}(H) \rangle \leq \U(\Z G)$. 
Recall that a shifted bicyclic unit is the product of a (generalized) bicyclic unit and the corresponding element of $H$:
\[
\sbic{}(h) = h + (1-h)g^{-1}\wt{H} = (1 + (1-h)g^{-1}\wt{H})h = \bic{h} h. 
\] 
So, $M := \langle H, \im(\sbic{g^{-1},H}) \rangle = \langle h, \bic{k} \mid h, k \in H \rangle$; set $N := \langle \bic{k} \mid k \in H \rangle$. 

We first show that $N$ is a normal complement for $H$ in $M$. 
A short computation relying on the fact $\wt{H} \cdot I(\Z H) = 0$ shows that the map
\[
\varphi: I(\Z H) \to N: a \mapsto 1 + ag^{-1}\wt{H}
\]
is a morphism of groups. 
It is obviously surjective, as $\bic{k} = \varphi(1-k)$. 
Since $\varphi(a)^n = 1 + nag^{-1}\wt{H}$ is the identity only when $na = 0$, we deduce that $N$ is a finitely generated torsion-free abelian group. 
In particular, $N \cap H$ is trivial. 
For $h \in H$, we have 
\[
\varphi(a)^h = h^{-1}(1 + ag^{-1}\wt{H})h = 1 + h^{-1}ag^{-1}\wt{H} = \varphi(h^{-1}a),
\]
so $N$ is normal in $M$, and $M$ is isomorphic to the semidirect product $N \rtimes H$ with the above action of $H$. 

It remains to compute the kernel of $\varphi$. 
Given $a \in I(\Z H)$, we compute inside $\Z G$ that
\[
ag^{-1}\wt{H} = 0 \iff gag^{-1}\wt{H} = 0 \iff gag^{-1} \in I(\Z H) \iff a \in I(\Z H^g).
\]
Thus, $a \in \ker \varphi$ if and only if $a \in I(\Z H) \cap I(\Z H^g) = I(\Z C)$. 
We conclude that $N \cong I(\Z H) / I(\Z C) \cong I(\Z[H/C])$, with the action of $h \in H$ given by $aC \mapsto h^{-1} aC$. 
\end{proof}

Note that $\langle 1 + (1-h)g^{-1}\wt{H} \mid h \in H \rangle \cong I(\Z[H/C])$ is a free abelian group of rank $\indx{H}{H\cap H^g} - 1$. 
\Cref{prop:Bovdigroup} thus shows that $\langle H, \im(\sbic{g^{-1},H}) \rangle$ is the extension of a free abelian group of rank $\indx{H}{H\cap H^g} - 1$ by the finite group $H$. 
In particular, if $g$ normalizes $H$ then unsurprisingly $\langle H, \sbic{g^{-1},H}(H) \rangle = H$, whereas if $H \cap H^g = 1$ then $\langle H, \sbic{g^{-1},H}(H) \rangle \cong I (\Z H) \rtimes H$ is virtually a free abelian group of rank $|H|-1$. 

\begin{corollary}\label{cor:freeabeliansubgroup}
Let $G$ be a finite group and $H$ a cyclic subgroup of $G$ of prime order $p$. 
If $g \in G$ does not normalize $H$, then $\U(\Z G)$ contains a subgroup isomorphic to $\Z^{p-1} \rtimes C_p$, where the action of $C_p = \langle x \rangle$ is given by
\[
\begin{cases}
x \cdot e_i = e_{i+1} -e_1 &\text{for } 1 \leq i \leq p-2,\\
x \cdot e_{p-1} = -e_1,
\end{cases}
\]
denoting $(e_i)_{i=1}^{p-1}$ the canonical basis of $\Z^{p-1}$. 
In particular, if $p = 2$, then $\U (\Z G)$ contains
\(
\Z \rtimes C_2 \cong C_2 \ast C_2,
\)
the infinite dihedral group. 
\end{corollary}

\begin{remark}
In general, an abelian subgroup $H \leq G$ gives rise to a free abelian subgroup of $\U(\Z H) \leq \U(\Z G)$ of rank $\frac{1}{2}(|H| + 1 + n_2 - 2\ell)$, where $n_2$ is the number of involutions in $H$ and $\ell$ is the number of cyclic subgroups of $H$, cf.\ \cite[Exercise 8.3.1]{PolcinoSehgal02} or \cite[Theorem $7.1.6.$]{JespersdelRio16}. 
\Cref{cor:freeabeliansubgroup} therefore yields a free abelian subgroup larger than one would expect at first. 
\end{remark}
\medskip

\Cref{cor:freeabeliansubgroup} for $p=2$ hints that it might be possible to construct free products of finite subgroups of $\U(R G)$ using appropriate shifted bicyclic maps. 
This belief is supported, via \Cref{thm:pingpongdeformations}, by several results in the literature (and their reformulations in terms of first order deformations), see namely \cite{MarciniakSehgal97,Passman04,GoncalvesPassman06,GoncalvesdelRio08,JanssensJespersTemmerman17,MarciniakSehgal23}. 
The following conjecture, which echoes \Cref{que:existencefaithfulembeddings} in the case of a group ring, expresses this precisely. 

\begin{conjecture}\label{conj:amalgamBovdimaps}
Let $H \leq G$ be finite groups for which $\Cir_{\Q}^G(H)$ is non-empty.
Let $g \in G$ be such that $H \cap H^g = 1$.
Then
\[
\langle \im(\sbic{H,g}), \im(\sbic{g^{-1},H}) \rangle \cong H \ast H \cong \langle \im(\sbic{H,g}), \im(\sbic{H,g})^{\ast} \rangle,
\]
where ${}^{\ast}: g \mapsto g^{-1}$ denotes the canonical involution on $\Q G$.
\end{conjecture}

By \Cref{prop:amalgaminproduct}, the condition that $G$ have an irreducible representation which is center-preserving on $H$, is necessary in order for $H$ to appear in a free product. 
Without the condition $H \cap H^g = 1$, one could still hope for a free amalgamated product of the form
\[
\langle \im(\sbic{H,g}), \im(\sbic{g^{-1},H}) \rangle\cong H \ast_C H .
\]
Note that according to \Cref{prop:Bovdimaps}.\ref{item:Bovdimaps3}, the amalgamated subgroup $C$ must in any case contain $H \cap H^g$. 
\medskip

If $F=\Q$, $G$ is nilpotent of class $2$, $H \cong C_m$, and $g \in G$ is such that $H \cap H^g=1$, then \cite[Theorem 4.1]{JanssensJespersTemmerman17} shows that the hypotheses of \Cref{thm:pingpongdeformations} are satisfied, and so $H \ast H$ can be constructed in the conjectured way. 
If $m$ is prime, the same result was obtained in \cite[Theorem 4.2]{JanssensJespersTemmerman17} also for arbitrary (finite) nilpotent groups. 
In both these cases, an explicit embedding of $H$ in a simple component of $\Q G$ was constructed. 
Recently, Marciniak and Sehgal \cite{MarciniakSehgal23} were able to drop the primality condition on $m$, without the use of such an embedding. 
(Note that by virtue of \Cref{thm:faithfulcenterpreservingrepresentation}, $\Cir_\Q^G(H) \cap \Fir_\Q^G(H)$ is always non-empty when $H$ is cyclic.) 
The literature on constructing copies of a free group using bicyclic units is much richer; we will survey it briefly in the next section.

\subsection{Bicyclic units generically play ping-pong} \label{subsec:bicyclicgenericallyfree}

Recall the definition of the \emph{bicyclic units} introduced in \Cref{rem:bicylicunits}: for $x \in RG$, $h \in G$ and $\wt{h} := \sum_{i=1}^{o(h)} h^{i}$, 
\begin{equation*}
\bic{h,x} = 1 + \wt{h} x (1-h) \quad \text{and} \quad \bic{x,h} = 1 + (1-h)x \wt{h}. 
\end{equation*}
As $(1-h)\wt{h}=0 = \wt{h}(1-h)$, any bicyclic unit satisfies $(\bic{h,x} - 1)^2 = 0 = (\bic{x,h} - 1)^2$, hence is unipotent. 
We denote the group they generate by
\[
\Bic_R(G) := \langle \bic{h,x}, \bic{x,h} \mid x \in RG, h \in G \rangle.
\]
Recalling that the idempotent $\frac{\wt{g}}{o(g)}$ is central in $RG$ if and only if $\langle g \rangle$ is normal in $G$, we see that it follows from \Cref{thm:DedekindiffimagesFrobeniuscomplement} and \eqref{eq:fpfviaidempotent} that
\[
\Bic_R(G) \neq 1 \iff \text{$G$ is not a Dedekind group. }
\]

For many years, an overarching belief in the field of group rings has been that two bicyclic units should generically generate a free group:

\begin{conjecture}[Folklore] \label{conj:bicyclicgenericallyfree}
Let $G$ be a finite (non Dedekind) group and $\beta$ a bicyclic unit. 
Then the set $\{\alpha \in \Bic_R(G) \mid \langle \alpha, \beta \rangle \cong \langle \alpha \rangle \ast \langle \beta \rangle \textnormal{ canonically}\}$ is ``large'' in $\Bic_R (G)$. 
\end{conjecture}

This conjecture has been intensively investigated for $R = \Z$. 
The reader may consult \cite{GoncalvesdelRio13} for a quite complete survey (until 2013), as well as \cite{GoncalvesPassman06,GoncalvesdelRio11,GoncalvesdelRio08,GoncalvesGuralnickdelRio14,JespersOlteanudelRio12,Raczka21} and the references therein. The reason for such active interest in these specific class of units is that Bicyclic units constitute one of the few generic constructions known for units in $RG$. We record in passing the new construction given in \cite{JanssensJespersSchnabel24} which also exhibit the property to yield free products.
\bigskip

In this section, we reap the fruits of \Cref{sec:pingpongreductivegroups,sec:pingpongsemisimplealgebra,sec:faithfulembedding} and prove a variant of \Cref{conj:bicyclicgenericallyfree}, in which one looks for bicyclic units that play ping-pong with the image of a shifted bicyclic map. 

Before stating this result, we introduce a piece of notation. 
In the same style as in \eqref{eq:notationEmb}, let
\[
\Sir_\Fi^G(H) = \{ e \in \PCI (\Fi G) \mid \textnormal{the representation $\pi_e$ is not fixed-point-free} \}
\]
gather the irreducible representations of $G$ whose restriction to $H$ is not fixed-point-free. 
As before, we abbreviate $\Sir_\Fi(G) = \Sir_\Fi^G(G)$. 

\begin{theorem} \label{thm:pingpongbicyclicwithconjugate}
Let $\Fi$ be a number field, and $R$ be its ring of integers. 
Let $H \leq G$ be finite groups, and pick $x \in \U(\Fi G)$. 
Set $C = H^x \cap \ZZ(G)$. 
The set 
\[
S = \{ \alpha \in \Bic_R(G) \textnormal{ of infinite order} \mid \textnormal{$\langle \alpha, H^x \rangle \cong (\langle \alpha \rangle \times C) \ast_C H^x$ canonically} \}
\]
is non-empty if and only if $\Cir_\Fi^G(H) \cap \Sir_\Fi(G) \neq \emptyset$. 
If so, $S$ is dense in $\Bic_R(G)$ for the join of the profinite and Zariski topologies. 

Moreover, if $\Cir_{\Fi}^G(H) \neq \emptyset$, $\Bic_R(G) \neq \{1 \}$, and $G$, $H$ satisfy \Cref{cond:NQ} from \Cref{thm:almostfaithfulembedding}, then $S \neq \emptyset$. 
\end{theorem}

The appearance of \Cref{cond:NQ} in the statement of \Cref{thm:pingpongbicyclicwithconjugate} stems from \Cref{thm:almostfaithfulembedding}. 
If \ref{cond:NQ} were not satisfied, one can take for $C$ the subgroup provided by \Cref{thm:almostfaithfulembedding}, and still obtain that $S \neq \emptyset$ (but in this case, one can a priori draw no conclusion on $\Cir_\Fi^G(H) \cap \Sir_\Fi(G)$). 

\begin{remark} \label{rem:simultaneouspingpongbicyclic}
The proof of the first part of \Cref{thm:pingpongbicyclicwithconjugate} is ultimately an application of \Cref{thm:pingpongdenseSLn}. 
Hence the statement of \Cref{thm:pingpongbicyclicwithconjugate} holds, mutatis mutandis, for a finite family of finite subgroups $H_i \leq G$ simultaneously. 
\end{remark}

\Cref{thm:pingpongbicyclicwithconjugate} certainly applies to first-order deformations of $H$ (by virtue of \Cref{thm:separabledeformationsinner}), and in particular, to the image of $H$ under a shifted bicyclic map. 
As an immediate consequence, we obtain the following positive answer to the shifted variant of \Cref{conj:bicyclicgenericallyfree}. 

\begin{corollary} \label{cor:pingpongbicyclicwithdeformation}
Let $\Fi$ be a number field and $R$ be its rings of integers. 
Let $H \leq G$ be finite groups satisfying \ref{cond:NQ}, and pick $x \in RG$. 
Suppose that $\Cir_\Fi^G(H) \neq \emptyset$ and $H \cap \ZZ(G) = 1$. 
Then
\[
S = \{\alpha \in \Bic_R(G) \mid \langle \alpha, \sbic{H,x}(H) \rangle \cong \langle \alpha \rangle \ast \sbic{H,x}(H) \textnormal{ canonically}\}
\]
is dense in $\Bic_R(G)$ for the join of the profinite and Zariski topologies. 

In particular, if $H = \langle h \rangle$ is cyclic and no non-trivial power of $h$ is central, then with $\beta = \bic{h,x} = 1 + \wt{h}x(1-h)$, the set
\[
S_{\beta} = \{\alpha \in \Bic_R(G) \mid \langle \alpha, h \beta \rangle \cong \langle \alpha \rangle \ast \langle h \beta \rangle \textnormal{ canonically}\}
\]
is dense in $\Bic_R(G)$ for the join of the profinite and Zariski topologies. 
\end{corollary}

\Cref{cor:pingpongbicyclicwithdeformation} also gives a precise interpretation of the word ``large'' in \Cref{conj:bicyclicgenericallyfree}, in terms of (arguably the only) two natural topologies on $\Bic_R(G)$. 

Note that while \Cref{conj:bicyclicgenericallyfree} and its shifted version would follow from a positive answer to \Cref{que:delaHarpe}, the existence of a ping-pong partner for a prescribed element $\beta$ is always subject to the condition of $\langle \beta \rangle$ (almost) embedding in an adjoint simple quotient. 
When $\beta$ has infinite order (e.g.\ $\beta$ is a bicyclic unit), this is automatic; but when $\beta$ has finite order (e.g.\ $\beta$ is a shifted bicyclic unit), this condition cannot be ignored. 
Verifying it is an essential part of \Cref{thm:pingpongbicyclicwithconjugate}, that goes back to \Cref{thm:almostfaithfulembedding}. 

\begin{remark} \label{rem:pingpongwithunipotent}
By enlarging $\Bic_R(G)$, the condition that $e \in \Sir_\Fi(G)$ can be weakened to supposing that $\Fi Ge$ is not a division algebra, i.e.\ that $e$ belongs to
\[
\Mir_\Fi(G) := \{ e \in \PCI(\Fi G) \mid \text{$\Fi Ge$ is not a division algebra}\} \supseteq \Sir_\Fi(G). 
\]
More precisely, set
\[
\U(R G)^+ = \langle \alpha \in \U(R G) \mid \alpha \text{ unipotent}\rangle.
\]
Then running the same proof as for \Cref{thm:pingpongbicyclicwithconjugate} produces the following statement: 

Let $\Fi$, $G$, $H$, $C$ and $x$ be as in \Cref{thm:pingpongbicyclicwithconjugate}. 
Suppose that $\Cir_{\Fi}^G(H) \cap \Mir_\Fi(G) \neq \emptyset$. 
Then the set 
\[
S = \{ \alpha \in \U(R G)^+ \mid \langle \alpha, H^x \rangle \cong (\langle \alpha \rangle \times C) \ast_C H^x \textnormal{ canonically} \}
\]
is dense in $\U(R G)^+$ for the join of the profinite and Zariski topologies.
\end{remark}
\bigbreak

The next lemma will essentially enable us to use \Cref{thm:pingpongdenseSLn} with $\Gamma = \Bic_R(G)$ or $\U(R G)^+$, in order to prove \Cref{thm:pingpongbicyclicwithconjugate}. 
For $\Subalg$ a subring of a semisimple $\Fi$-algebra $\Alg$, let
\[
\SL_1(\Subalg) = \{ x \in \Subalg \mid \Nrd_{\Alg / \ZZ(\Alg)}(\pi_e(x)) = 1 \textnormal{ for every $e \in \PCI(\Alg)$} \}
\]
denote the set of elements of $\Alg$ whose projections in every simple factor have reduced norm 1. 

\begin{lemma}\label{lem:bicyclicZariskidense}
Let $\Fi$ be a number field, $R$ be its ring of integers, and $G$ be a finite group. 
\begin{enumerate}[itemsep=1ex,topsep=1ex,label=\textup{(\roman*)},leftmargin=2em]
	\item $\Bic_R(G)$ is a Zariski-dense subgroup of $\SL_1(\prod_{e \in \Sir_\Fi(G)} \Fi G e)$. 
	\item $\U(R G)^+$ is a Zariski-dense subgroup of $\SL_1(\prod_{e \in \Mir_\Fi(G)} \Fi G e)$. 
\end{enumerate}
\end{lemma}
\begin{proof}
Note that the image of a unipotent element under any algebra morphism remains unipotent. 
Moreover, unipotent elements have reduced norm $1$, and division algebras have no non-trivial unipotent elements. 
In consequence, $\U(R G)^+$ is a subgroup of $\SL_1(\prod_{e \in \Mir_\Fi(G)} RGe)$. 

Now if $e \notin \Sir_\Fi(G)$, that is, if $\pi_e(G)$ is fixed-point-free, then $\wt{h}e$ is central in $\Fi Ge$ for every $h \in G$. 
This implies that for any $h \in G$, $x \in RG$, the bicyclic units $\bic{h,x} = 1 + \wt{h} x (1-h)$ and $\bic{h,x}$ have trivial image under $\pi_e$. 
In other words, $\Bic_R(G)$ is indeed a subgroup of $\SL_1(\prod_{e \in \Sir_\Fi(G)} \Fi G e)$. 
\medskip

It remains to establish the Zariski-density of both groups. 
To this end, consider the Artin--Wedderburn decomposition
\[
\Fi G \cong \prod_{e \in \PCI(\Fi G)} \Mat_{n_e}(D_e),
\]
and pick for each $e \in \PCI(\Fi G)$ a maximal order $\mathcal{O}_e$ in $D_e$ such that $RGe \subseteq \Mat_{n_e}(\mathcal{O}_e)$. 

For every $e \in \Mir_\Fi(G)$ (that is, every $e$ for which $n_e \geq 2$), we choose an idempotent $f_e$ in $\Fi G$ such that $f_e e$ is non-central in $\Fi G e$. 
We have seen in \eqref{eq:fpfviaidempotent} that when $e \in \Sir_\Fi(G)$, there exists $g_e \in G$ for which $\wt{g_e}e$ not central. 
Accordingly, for $e \in \Mir_\Fi(G)$ we will choose the non-central idempotent $f_e = \widehat{g_e} := \frac{1}{o(g_e)}\wt{g_e}$. 

Associated with $f_e$ is the group of generalized bicyclic units $\mathrm{GBic}^{\{f_e \}}_R(G)$, generated by certain unipotent elements in $RG$ (see \cite[Section 11.2]{JespersdelRio16} for the definition). 
By definition, when $f_e = \widehat{g_e}$, one has $\mathrm{GBic}^{\{\widehat{g_e}\}}_R(G) \leq \Bic_R(G)$. 
Now thanks to \cite[Theorem 6.3]{JanssensJespersSchnabel24}, the group $\mathrm{GBic}^{\{f_e\}}_R(G)$ contains a subgroup $U_e$ of the form $1-e + E_{n_e}(I_e)$, for $I_e$ some non-zero ideal of $\O_e$. 
Here, and for $I$ an ideal of an arbitrary ring $\O$,
\[
E_{n}(I) := \langle e_{ij}(r) \mid 1 \leq i \neq j \leq n, \, r \in I \rangle
\]
is the subgroup of $\GL_{n}(\O)$ generated by the elementary matrices $e_{ij}(r)$ with $1$ on the diagonal and $r$ in the $(i,j)$-entry. 
\medskip

In consequence, we have reduced the lemma to proving that $E_n(I)$ is Zariski-dense in $\SL_n(D)$, for $I$ a non-zero ideal of an order $\O$ in a finite-dimensional division $\Fi$-algebra $D$ and $n \geq 2$. 
This is well-known, and follows from combining the facts that the Zariski closure of $e_{ij}(I)$ is precisely the subgroup $e_{ij}(D)$ of $\SL_n(D)$, and that $E_n(D)$ is Zariski-dense in $\SL_n(D)$.\footnote{In the present setting, it even holds that $E_n(D) = \SL_n(D)$ by virtue of the solution to the Kneser--Tits problem (due to Wang for division algebras over number fields). }
\end{proof}

\Cref{thm:pingpongbicyclicwithconjugate} now follows readily. 

\begin{proof}[Proof of \Cref{thm:pingpongbicyclicwithconjugate}.]
Suppose first that there exists $e \in \Cir_\Fi^G(H) \cap \Sir_\Fi(G)$; in other words, that $H$ (or equivalently, $H^x$) almost embeds in $\PGL_{\Fi Ge}$ and that $\pi_e$ is not a fixed-point-free representation of $G$. 
In view of \Cref{lem:bicyclicZariskidense}, we can apply \Cref{thm:pingpongdenseSLn} to $\bG = \GL_{\Alg}$ with $\Alg = \prod_{e \in \Sir_\Fi(G)} \Fi G e$, $\Gamma = \Bic_R(G)$, and $A = B$ the projection of $H^x$ to $\Alg$, using that adjoint simple quotient $\PGL_{\Fi Ge}$ of $\bG$. 
Set $C_B = B \cap \bZ(\Fi)$. 
We obtain that the set $S$ of regular semisimple elements $\alpha \in \Bic_R(G)$ of infinite order such that the canonical map
\(
\langle \alpha, C_B \rangle \ast_{C_B} B \to \langle \alpha, B \rangle
\)
is an isomorphism, is dense for the join of the Zariski and profinite topologies. 
But of course, since the projection of $\Bic_R(G)$ to $\prod_{e \notin \Sir_\Fi(G)} \Fi G e$ is trivial, this is equivalent to the canonical map
\(
\langle \alpha, C \rangle \ast_{C} H^x \to \langle \alpha, H^x \rangle
\)
being an isomorphism, for $C = H^x \cap \ZZ(G) = H \cap \ZZ(G)$. 
\medskip

Conversely, if there exists an element $\alpha \in \Bic_R(G)$ of infinite order for which $\langle \alpha, H^x \rangle$ is the free amalgamated product of $\langle \alpha \rangle$ and $H^x$ along $C$, then by \Cref{rem:almostembeddingnecessary} (going back to \Cref{prop:amalgaminproduct}), $H^x$ and $\langle \alpha \rangle$ must almost embed in the same adjoint simple quotient of $\bG$. 
As $\Bic_R(G)$ is a subgroup of $\SL_1(\prod_{e \in \Sir_\Fi(G)} \Fi G e)$ by \Cref{lem:bicyclicZariskidense}, this simple quotient must correspond to an element of $\Sir_\Fi(G)$. 
In other words, there exists $e \in \Cir_\Fi^G(H^x) \cap \Sir_\Fi(G) = \Cir_\Fi^G(H) \cap \Sir_\Fi(G)$. 
\medskip

Finally, if $\Bic_R(G) \neq \{1\}$, then $G$ is not Dedekind.
\Cref{thm:almostfaithfulembedding}.\ref{item:notFrobenius} thus shows that $\Fir_{\Fi}(H) \neq \emptyset$ implies $\Cir_{\Fi}^G(H) \cap \Sir_\Fi(G) \neq \emptyset$ under the appropriate hypothesis on $C$. 
\end{proof}

To close this section, we record the following consequence of \Cref{thm:pingpongbicyclicwithconjugate}, for which we can easily give a direct proof thanks to the stronger assumptions on $G$.
These assumptions hold in particular when $G$ is a non-abelian simple group: then $G$ has trivial center (hence is not a Frobenius complement, see e.g.\ \Cref{lem:centralquotientFrobeniuscomplement}), and every non-trivial representation of $G$ is faithful.

\begin{corollary} \label{cor:pingpongbicyclicwithsimple}
Let $\Fi$ be a number field and $R$ be its ring of integers. 
Let $G$ be a finite group. 
Suppose that $G$ is not a Frobenius complement, and possesses a faithful irreducible representation. 
Then the set of $\alpha \in \Bic_R(G)$ such that the canonical map
\[
\langle \alpha , \ZZ(G) \rangle \ast_{\ZZ(G)} G \to \langle \alpha, G \rangle 
\]
is an isomorphism, is dense in $\Bic_R(G)$ for the join of the Zariski topology and the profinite topology. 
\end{corollary}
\begin{proof}
Let $\rho$ be a faithful irreducible representation of $G$. 
In consequence, $\rho$ is center-preserving on $G$. 
By assumption, $\rho(G)$ is not a Frobenius complement; equivalently, $\rho(G)$ is not fixed-point-free. 
In particular, the representation $\rho$ is not fixed-point-free, hence $\rho(\Fi G)$ is not a division algebra. 

By \Cref{lem:bicyclicZariskidense}, $\rho(\Bic_R(G))$ is Zariski-dense in $\PGL_{\rho(\Fi G)}$. 
An application of \Cref{thm:pingpongdenseSLn} now proves that the set of $\alpha \in \Bic_R(G)$ such that the canonical map
\[
\langle \rho(\alpha), \rho(\ZZ(G)) \rangle \ast_{\rho(\ZZ(G))} \rho(G) \to \langle \rho(\alpha), \rho(G) \rangle 
\]
is an isomorphism, is dense in $\Bic_R(G)$ for the join of the Zariski topology and the profinite topology. 
Finally, \Cref{lem:preimageamalgam} enables us to lift this free amalgamated product back to $\U(RG)$. 
\end{proof}

%%%%% Bibliography %%%%%

\newpage
\bibliographystyle{alphaabbrv}
\bibliography{bibliography}

\end{document}